\documentclass[reqno,
12pt]{amsart}
\usepackage{
	amssymb,
	amsmath,
	amsthm,
	eucal,
	dsfont,
	multicol,
	mathrsfs,
	tikz,
	graphicx,
	hyperref
}
\usepackage{lipsum}
\makeatletter\renewcommand*{\eqref}[1]{%
	\hyperref[{#1}]{\textup{\tagform@{\!\!\ref*{#1}}}}%
}\makeatother 

\makeatletter

\@addtoreset{equation}{section}
\makeatother
\theoremstyle{plain}
\newtheorem{theorem}{Theorem}[section]
\newtheorem{lemma}[theorem]{Lemma}
\newtheorem{proposition}[theorem]{Proposition}
\newtheorem{corollary}[theorem]{Corollary}
\theoremstyle{definition}
\newtheorem{definition}[theorem]{Definition}
\newtheorem{remark}[theorem]{Remark}

\newcommand{\norm}[1]{{\|#1\|}}

\def\supp{\mathop{\mathrm{supp}}\nolimits}

\def\Id{\mathop{\mathrm{Id}}\nolimits}

\def\ac{\mathop{\mathrm{ac}}\nolimits}

\def\Im{\mathop{\mathrm{Im}}\nolimits}

\def\sgn{\mathop{\mathrm{sgn}}\nolimits}

\def\R{{\mathbb{R}}}

\def\N{{\mathbb{N}}}
\def\C{{\mathbb{C}}}

\def\S{{\mathcal{S}}}
\def\F{{\mathcal{F}}}
\def\H{{\mathcal{H}}}

\def\<{{\langle}}
\def\>{{\rangle}}

\def\ep{{\varepsilon}}

\DeclareMathOperator*{\slim}{s-lim}

\title[ $L^p$-boundedness of wave operators]{ The $L^p$-boundedness of wave operators for nonhomogeneous fourth-order Schr\"odinger operators in high dimensions}
\author{Zijun Wan}
\address[Z. Wan]{Department of Mathematics, Central China Normal University, Wuhan, 430079, P.R. China}
\email{zijunwan@mails.ccnu.edu.cn}
\author{Xiaohua Yao}

\address[X. Yao]{Department of Mathematics and Key Laboratory of Nonlinear Analysis and Applications(Ministry of Education), Central China Normal University, Wuhan, 430079, P.R. China}
\email{yaoxiaohua@ccnu.edu.cn}

\thanks{The work is partially supported by NSFC No. 12171182}

\begin{document}


	\begin{abstract}
			This paper investigates the $L^p$-boundedness of wave operators associated with the following nonhomogeneous fourth-order Sch\"odinger operator on $\mathbb{R}^n$: $$H = \Delta^2 - \Delta + V(x).$$
Assuming the real-valued potential \( V \) exhibits sufficient decay and regularity, we prove that for all dimensions \( n \geq 5 \), the wave operators \( W_{\pm}(H, H_0) \) are bounded on \( L^{p}(\mathbb{R}^{n}) \) for all \( 1 \leq p \leq \infty \), provided that zero is a regular threshold of \( H \).

			As applications, we derive the sharp $L^p$-$L^{p'}$ dispersive estimates for Schr\"odinger group $e^{-itH}$, as well as for the solutions operators $\cos(t \sqrt{H})$ and $\frac{\sin (t \sqrt{H})}{ \sqrt{H}}$ associated with the following  beam equations with potentials:
			\begin{equation*}
				\begin{cases}
					\partial_t^2 u + \left(\Delta^2 -\Delta+ V(x) \right) u = 0, & \\
					u(0, x) = f(x), \quad \partial_t u(0, x) = g(x),& (t, x) \in \mathbb{R} \times \mathbb{R}^n,\  n\geq5,
				\end{cases}
			\end{equation*}
		where $p'$ denotes the H\"older conjugate of $p$, with $1 \leq p \leq 2$.  

      Moreover,  we remark that the same results hold for the operator $ \epsilon \Delta^2 - \Delta + V$ with a parameter $\epsilon>0,$ 
 providing greater flexibility for the analysis of related equations.
	\end{abstract}
	\maketitle

    \tableofcontents
\section{Introduction and main results}
\subsection{Introduction} 
We consider the following  nonhomogeneous fourth order Schr\"odinger operator  $$H=H_0+V(x), \  H_0=\Delta^2- \Delta,\ \  x\in\R^n, \  n\geq5.$$     Here   $V(x)$ is a real-valued potential satisfying $|V(x)|\lesssim (1+|x|)^{-\beta}$ with some  $\beta>0$. For convenience, we use the notation $\langle x \rangle:=(1+|x|).$


The considered operator arises from  the following fourth order Schr\"odinger equation:
\begin{equation}\label{fourth_order_schrodinger}
	i\partial_t u  + \epsilon \Delta^2 u- \Delta u + |u|^p u = 0, \quad u: \mathbb{R} \times \mathbb{R}^n \to \mathbb{C}, \quad \epsilon \in \mathbb{R}.
\end{equation}
Equation \eqref{fourth_order_schrodinger} was introduced by Karpman \cite{Karpman1, Karpman2} and Karpman and Shagalov \cite{Karpman-Shagalov} to account for the role of small fourth order dispersion terms in the propagation of intense laser beams in a bulk medium with Kerr nonlinearity. When $\epsilon = 0$, $n= 2$, and $p = 1$, this corresponds to the canonical model. For $2 < np < 4$, Karpman and Shagalov \cite{Karpman-Shagalov} showed, among other things, that the waveguides induced by the nonlinear Schr\"odinger equation become stable when $|\epsilon|$ is taken sufficiently large. 
The study of nonhomogeneous fourth-order equations has garnered increasing interest from researchers (see {\it e.g.} \cite{Hayashi-Naumkin, Jiang-Pausader-Shao, Pausader-Xia, Ru-Wa-Zha, Schlag_3, Schlag_4} and references therein).

By applying the scaling transformation $U_ \epsilon$   defined by
$\big(U_ \epsilon u \big)(x)=
|\epsilon |^{-\frac{n}{4}} u\big( |\epsilon |^{-\frac{1}{2}}x \big),$ it is straightforward to verify that $  \epsilon \Delta^2-\Delta   =  \epsilon^{-1}U_ \epsilon\big( \Delta^2 - \Delta  \big)U_ \epsilon^{-1}$ if $\epsilon>0$ and $  \epsilon \Delta^2-\Delta = \epsilon^{-1}U_ \epsilon\big(  \Delta^2+\Delta  \big)U_ \epsilon^{-1}$ if $\epsilon<0.$  Hence, for a fixed  $\epsilon \in \R\setminus\{0\},$
the study of $ \epsilon \Delta^2-\Delta $ can be reduced to that of  $ \Delta^2\pm\Delta$.

Compared to the operator $\Delta^2 - \Delta$, which has only one threshold at $0$, the nonhomogeneous fourth-order operator $\Delta^2 + \Delta$ features two distinct thresholds: $0$ and $-\frac{1}{4}$ (the latter being degenerate). It is noteworthy that analyzing the operator $\Delta^2 + \Delta + V$ can present greater challenges in the presence of the potential $V$.
In this paper, our primary focus is on studying the $L^p$-boundedness of the wave operators $W_\pm(H, H_0)$, where $H_0=\Delta^2-\Delta$. Additionally, we aim to investigate the more intricate operator $\Delta^2+\Delta$ in other work. 

 Now let's recall that (see  e.g.  Kuroda \cite{Kuroda}) as $\beta>1$, {\it the wave operators}
\begin{align}\label{def-wave}
	W_\pm=W_\pm(H, H_0) :=\slim_{t\to\pm\infty}e^{itH}e^{-itH_0}
\end{align}
exist as partial isometries on  $L^2(\mathbb{R}^n)$ and are asymptotically complete, namely the range of $W_\pm$ coincides with the absolutely continuous spectral subspace $\H_{\ac}(H)$ of $H$. In particular,  $W_\pm$ satisfy the  identities
$
W_\pm W_\pm^* =P_{\ac}(H)$, $W_\pm^*W_\pm=I
$ 
and exhibit {\it the intertwining property} $f(H)W_\pm=W_\pm f(H_0)$ for any Borel measurable function $f$ on $\mathbb{R}$,
where $W_\pm^*$  denotes  the dual operator of  $W_\pm$  and  $P_{\ac}(H)$ denotes the projection onto   $\H_{\ac}(H).$ These formulas, in particular, imply
\begin{align}
	\label{intertwining_1}
	f(H)P_{\ac}(H)=W_\pm f(H_0)W_\pm^*.
\end{align}
We refer to \cite{ReSi} for a detailed description on the mathematical scattering theory. Based on the equality \eqref{intertwining_1}, the $L^p$-boundedness of $W_\pm$ and $W_\pm^*$ can be directly employed to establish the $L^p$-$L^q$ estimates for the perturbed operator $f(H)$ from the same estimates for the free operator $f(H_0)$. This can be seen as follows:
\begin{align}
	\label{Lp-bound of f(H)}
	\|f(H)P_{\ac}(H)\|_{L^p\to L^q}\le \|W_\pm\|_{L^q\to L^q}\ \|f(H_0)\|_{L^p\to L^q}\|W_\pm^*\|_{L^{p}\to L^{p}}.
\end{align}
	Actually, there has been extensive research on the $L^p$-boundedness of wave operators associated with homogeneous Schr\"odinger operators  $-\Delta + V(x)$ and $\Delta^2 + V(x)$ (see further related results below). However, only a few studies have addressed nonhomogeneous type Schr\"odinger operators $H$.  Therefore, this study aims to contribute to filling some gap.
	
In this paper, we demonstrate that the wave operators $W_\pm(H, H_0)\in\mathbb{B}(L^p(\R^n))$ for all $1\leq p\leq\infty$ and $n \geq 5 $ when  zero is a regular point of $H$ (see Theorem  \ref{main_theorem}).  As  applications,  we derive the $L^p$-$L^{p'}$ estimates for the Schr\"odinger group $e^{-itH}$
 and for  the solutions to   beam equation (see Corollaries \ref{perturation decay} and \ref{theorem beam}). In particular, it is worth noting that the same results hold for the perturbed  operator $ \epsilon \Delta^2 - \Delta + V$ with a parameter $\epsilon>0,$ 
 providing greater flexibility for the analysis of related equations.
  
 We emphasize that the proofs of these results heavily rely on the iteration of the Born series to deal with singularities. This approach incorporates techniques inspired by the works of Yajima \cite{Yajima-JMSJ-95} and Erdo\u{g}an and Green \cite{Erdogan-Green21, Erdogan-Green23}.

		\subsection{Notations}\label{Notations}
	In this subsection, we collect some notations used in the paper.
	\begin{itemize}
		\item
		For $a\in \mathbb{R}$, $a\pm$ denotes $a\pm \epsilon$ for  $0<\epsilon \ll 1$.  $\left\lfloor a \right\rfloor$ denotes the greatest integer less than or equal to $a$, while $\left\lceil a \right\rceil$ denotes the smallest integer greater than or equal to $a$.
		\vskip0.1cm
		\item For $a, b\in \mathbb{R}^{+}$, $a\lesssim b$ (resp. $a\gtrsim b$) means $a\le cb$ (resp. $a\ge cb$) with some constant $c>0$. Moreover, we denote $a \lesssim_m b$ if the constant depends on the variable $m$.
		\vskip0.1cm
		\item $\mathcal{S}(\R^n)$ denotes the Schwartz class and   $S^{n-1}$ denotes the unit sphere on $\R^n.$
			\vskip0.1cm
		\item  The Fourier transform  of $f$ is  defined as
		\begin{equation*}\hat{f}(\xi)=\mathcal{F}(f)(\xi)=\frac{1}{(2\pi)^{\frac{n}{2}}}\int_{\R^3}e^{-ix\xi}f(x)dx, \  \xi\in \R^n. 
		\end{equation*}
		We denote the inverse Fourier transform of $f$ by $\check{f}(\xi)$ or $\mathcal{F}^{-1}(f)(\xi)$.
			\vskip0.1cm
		\item 
		Let $\chi \in C_c^{\infty}(\mathbb{R})$  such that $\chi(\eta) = 1$ for $|\eta| < \eta_0$ with some sufficiently small $0<\eta_0 \ll 1$, and $\chi(\eta) = 0$ for $|\eta| > 2\eta_0$. For a subset $A \subseteq \mathbb{R}^n$ with $n \in \mathbb{N}^+$, we define $\chi_A := \chi_A(x)$ as the characteristic function of $A$, given by $\chi_A(x) = 1$ if $x \in A$, and $\chi_A(x) = 0$ if $x \notin A$.
		
		\vskip0.1cm
		\item A bounded operator $K \in \mathbb B(L^2)$ (with kernel $K(x,y)$) is said to be absolutely bounded ($K \in \mathbb{AB}(L^2)$ for short) if the integral operator $|K|$ with kernel $|K(x,y)|$ is also bounded on $L^2$. Note that $K \in \mathbb{AB}(L^2)$ if  $K$ is a Hilbert--Schmidt operator on $L^2$.
	\end{itemize}

	\subsection{Main results}Before presenting our main results, we firstly introduce the regular zero spectral condition  of $H$. Let $v(x)=|V(x)|^{1/2}$ and $U(x)=\sgn V(x)$, i.e.  $U(x)=1$ if $V(x)\ge0$ and $U(x)=-1$ if $V(x)<0$. Denote $R_0(0)=(-\Delta)^{-1}-(-\Delta+1)^{-1}.$
	
	\begin{definition}\label{definition regular}
		Let $n \geq 5$ and $H=\Delta^2-\Delta+V.$ Define $T_0 := U + vR_0(0)v.$ We say that zero is a \textbf{regular threshold} of  $H$ if $T_0$ is invertible on $L^2(\mathbb{R}^n)$.
	\end{definition}
	Next, we state our main results.
		\begin{theorem}\label{main_theorem}
		Let $n\geq5$ and $V(x)$  be a real-valued potential on $\R^n$ satisfying $|V(x)| \lesssim \left \langle  x\right \rangle^{-\beta}$ with $\beta>n+5.$  For a fixed  $0<\sigma\ll1,$  there exists a constant 
		$ C(n, \sigma)$ such that  
	\begin{align}\label{V assumption}
		\begin{cases}
			\big\|\langle \cdot\rangle^{\frac{9-n}{2}+\sigma+} V\big\| _{L^2}\lesssim C(n, \sigma) & {\rm if}\ \ n=5,6, \\
			\big\| \F\big(\langle\cdot\rangle^{\frac{2(n-4+\sigma)}{n-1}+2\sigma+} V\big)\big\|_{L^{\frac{n-1}{n-4+\sigma}}} \lesssim C(n, \sigma)  & {\rm if}\ \ n\geq7.
		\end{cases}
	\end{align}
	Assume that $H=\Delta^2-\Delta+V$ has no positive embedded eigenvalue and zero is a regular point of $H$. Then the wave  operators $W_\pm$ defined by \eqref{def-wave} can be extended to into  bounded operators on $L^p(\R^n)$ for all $1 \leq p \leq \infty.$ 
	\end{theorem}
    Several remarks are given as follows:
 \begin{remark}\label{W epsilon}
For fixed $\epsilon > 0$, let $V_\epsilon(x) = \epsilon V(\epsilon^{\frac{1}{2}}x)$,  $\mathcal{H}_\epsilon = \epsilon \Delta^2 - \Delta + V$ and $H_\epsilon = \Delta^2 - \Delta + V_\epsilon$. 
Consider the scaling  unitary  transformation $U_\epsilon$ defined by
$$
(U_\epsilon u)(x) = \epsilon^{-\frac{n}{4}} u(\epsilon^{-\frac{1}{2}}x), \ \ u\in L^2(\R^n).
$$
A direct computation shows that 
$
U_\epsilon^{-1} \mathcal{H}_\epsilon U_\epsilon = \epsilon^{-1} H_\epsilon,
$
establishing the unitary equivalence between $\mathcal{H}_\epsilon$ and $\epsilon^{-1} H_\epsilon$. Then their corresponding wave operators admit the following equivalent relationship:
\begin{align}\label{equivelence}
W_\pm(\mathcal{H}_\epsilon, \epsilon \Delta^2 - \Delta) 
= \operatorname*{s-lim}_{t\to\pm\infty} e^{it\mathcal{H}_\epsilon} e^{-it(\epsilon \Delta^2 - \Delta)} & = U_\epsilon \left( \operatorname*{s-lim}_{t\to\pm\infty} e^{it\epsilon^{-1}H_\epsilon} e^{-it\epsilon^{-1}(\Delta^2 - \Delta)} \right) U_\epsilon^{-1}\nonumber\\
&=U_\epsilon \ W_\pm(H_\epsilon, \Delta^2-\Delta)\  U_\epsilon^{-1}.
\end{align}
Assuming that  $V_\epsilon$ and $H_\epsilon=\Delta^2 - \Delta + V_\epsilon$ satisfy the same conditions  in Theorem~\ref{main_theorem} for each $\epsilon>0$,
then it follows from Theorem~\ref{main_theorem} and \eqref{equivelence}  that the wave operators $W_\pm(\mathcal{H}_\epsilon, \epsilon\Delta^2 - \Delta)$ are  bounded on $L^p(\mathbb{R}^n)$ for all $1 \leq p \leq \infty$ and dimensions $n\geq5,$ and 
$$\|W_\pm(\mathcal{H}_\epsilon, \epsilon\Delta^2 - \Delta)\|_{L^p(\R^n)\rightarrow L^{p}(\R^n)}=\|W_\pm(H_\epsilon, \Delta^2 - \Delta)\|_{L^p(\R^n)\rightarrow L^{p}(\R^n)}.$$
\end{remark}
    
    \begin{remark}
For the homogeneous fourth-order operator \( H=\Delta^2 + V \) under the same condition \eqref{V assumption}, Erdo\u{g}an and Green \cite{Erdogan-Green21, Erdogan-Green23} established that \( W_\pm(H, \Delta^2) \in \mathbb{B}(L^p(\mathbb{R}^n)) \) for all \( 1 \leq p \leq \infty \) in the regular case when \( n \geq 5 \). It is important to note that the dimensional restriction \( n \geq 5 \) is necessary, as \( W_\pm(H, \Delta^2) \) is generally unbounded for \( p = 1, \infty \) even when \( V \) is compactly supported, if \( n \leq 4 \) (see, e.g., Mizutani--Wan--Yao \cite{MiWYa, MWY23} for counterexamples in dimensions \( n = 1,3 \), respectively).  

However, since the low-energy behavior of the non-homogeneous fourth-order operator \( H = \Delta^2 - \Delta + V \) resembles that of \( -\Delta + V \), the dimensional threshold in Theorem \ref{main_theorem} may be reduced to \( n \geq 3 \) with suitable modifications to certain proofs. Furthermore, we expect that the wave operators for \( H = \Delta^2 - \Delta + V \) remain bounded on \( L^p(\mathbb{R}^n) \) for all \( 1 \leq p < \frac{n}{2} \) if zero is an eigenvalue of \( H \) (see, e.g., \cite{Goldberg-Green-Advance, Yajima_2016} for analogous results of \( -\Delta + V \)).  
  \end{remark}
 \begin{remark}
For \( n \geq 7 \), since the index \( \frac{n-1}{n-4+\sigma} < 2 \), the potential condition \eqref{V assumption} essentially imposes a regularity requirement on \( V \). In fact, Erdo\u{g}an, Goldberg, and Green \cite{EGG23} showed that the wave operators for \( \Delta^2 + V \) are unbounded on \( L^p(\mathbb{R}^n) \) for \( \frac{2n}{n-7} < p \leq \infty \) if \( V \) lacks sufficient smoothness, even when \( V \) is compactly supported. Therefore, we believe that an appropriate regularity condition on the potential \( V \) is necessary for the results in Theorem \ref{main_theorem}.
    \end{remark}

    \begin{remark} 
In contrast to \( \Delta^2 - \Delta \), which has a single threshold at \( 0 \), the non-homogeneous fourth-order operator \( \Delta^2 + \Delta \) presents two distinct thresholds: \( 0 \) and \( -\frac{1}{4} \) (the latter being degenerate), which complicates the analysis. Furthermore, in the proof of Theorem \ref{main_theorem}, we need the symbol function \( \phi_k(\xi) = \left( |\xi-k|^2 + |\xi|^2 + 1 \right)^{-1} > 0 \) as defined in \eqref{phi}. However, for \( \Delta^2 + \Delta \), the corresponding  function \( \left( |\xi-k|^2 + |\xi|^2 - 1 \right)^{-1} \) changes sign, which leads to an additional obstacle of technics if we study the ``+" case. 
		\end{remark}

	\subsection{Applications}  As applications of Theorem \ref{main_theorem}, we will establish $L^p$-$L^{p'}$ estimates for the Schr\"odinger group $e^{-itH}$ and the solutions of beam equation. Firstly,  when $V=0$, 
we have the following $L^p$-$L^{p'}$  estimate (see {\it e.g.} \cite{CMY}):
	\begin{align}\label{free decay}
		\big\|e^{-it(\Delta^2-\Delta)}\big\|_{L^p(\R^n)\rightarrow L^{p'}(\R^n)}\lesssim
		\begin{cases}
			|t|^{-\frac{n}{2}
				(\frac{1}{p}-\frac{1}{2})}, & {\rm if}\ 0<|t|<1, \\
			|t|^{-n\big(\frac{1}{p}-\frac{1}{2}\big)},	 & {\rm if}\ |t|\geq1,
		\end{cases}
	\end{align}
	where $\frac{1}{p}+\frac{1}{p'}=1$ and  $1\leq p\le 2.$
	
	As a consequence of Theorem \ref{main_theorem}, by virtue of \eqref{Lp-bound of f(H)} and  \eqref{free decay}, one can immediately obtain the following Corollary \ref{perturation decay}.

\begin{corollary} \label{perturation decay}
	Let $n\geq5$ and $P_{ac}(H)$ denote the projection onto the absolutely continuous spectrum of $H$. Suppose that $H$ and $V$ satisfy the same conditions as stated  in Theorem \ref{main_theorem}.  Then for any $1\leq p\le 2,$
	\begin{align*}
		\big\|e^{-itH}P_{ac}(H)\big\|_{L^p(\R^n)\rightarrow L^{p'}(\R^n)}\lesssim
		\begin{cases}
			|t|^{-\frac{n}{2}
				(\frac{1}{p}-\frac{1}{2})}, & {\rm if}\ 0<|t|<1, \\
			|t|^{-n\big(\frac{1}{p}-\frac{1}{2}\big)},	 & {\rm if}\ |t|\geq1.
		\end{cases}
	\end{align*}
\end{corollary}
\begin{remark}
Under the same conditions as stated in Remark \ref{W epsilon}, combine with Corollary \ref{perturation decay} and $e^{-it\mathcal{H}_\epsilon}P_{ac}(\mathcal{H}_\epsilon)=U_\epsilon \big(e^{-it\epsilon^{-1}H_\epsilon}P_{ac}(H_\epsilon)\big)U_\epsilon ^{-1},$ then for any $1\leq p\le 2$ and $n\geq5,$
\begin{align*}
		\big\|e^{-it\mathcal{H}_\epsilon}P_{ac}(\mathcal{H}_\epsilon)\big\|_{L^p(\R^n)\rightarrow L^{p'}(\R^n)}\lesssim_\epsilon
		\begin{cases}
			|t|^{-\frac{n}{2}
				(\frac{1}{p}-\frac{1}{2})}, & {\rm if}\ 0<|t|<1, \\
			|t|^{-n\big(\frac{1}{p}-\frac{1}{2}\big)},	 & {\rm if}\ |t|\geq1
		\end{cases}
	\end{align*}
where $\mathcal{H}_\epsilon = \epsilon \Delta^2 - \Delta + V$ and $H_\epsilon = \Delta^2 - \Delta + V_\epsilon$ for each $\epsilon>0$.
\end{remark}

Next, we turn to establish  the time decay estimates of the solutions  to  the following beam equation:
\begin{equation}\label{beam_equation}  
	\begin{cases}
		\partial_t^2 u + \left( \Delta^2 -\Delta+ V(x) \right) u = 0, & \\
		u(0, x) = f(x), \quad \partial_t u(0, x) = g(x),& (t, x) \in \mathbb{R} \times \mathbb{R}^n,\  n\geq5.
	\end{cases}
\end{equation}
Under this assumption  that $H$ has no embedded positive eigenvalues and zero is a regular point of $H,$ we denote the negative eigenvalues of
$H$ by $\lambda_j(j \geq 1)$,
counting multiplicities of every $\lambda_j$, where
$H \phi_j=\lambda_j \phi_j ( j \geq 1 )$ for $\phi_j \in L^2\left(\mathbb{R}\right)$, and let $P_{a c}(H)$ denote the projection onto the absolutely continuous spectrum space of $H$.
Given these notations, the solution of the Beam equation
\eqref{beam_equation}  can then be expressed as
$$
u_m(t,x):=u_{m,d}(t, x)+u_{m,c}(t, x),
$$
where
$$
\begin{aligned}
	& u_{m,d}(t,x)=\sum_{j} \cosh( t\sqrt{-\lambda_j})(f,\phi_j)\phi_j(x) +\frac{\sinh (t\sqrt{-\lambda_j})}{\sqrt{-\lambda_j}}(g,\phi_j)\phi_j(x), \\
	& u_{m,c}(t,x)=\cos (t \sqrt{H})P_{ac}(H)f(x) + \frac{\sin (t \sqrt{H})}{\sqrt{H}}P_{ac}(H)g(x).
\end{aligned}
$$
The presence of negative eigenvalues in the spectrum of
$H$ leads to exponential growth in
$u_{m,d}(t,x)$ as  $t$ becomes large. As a result, it is crucial to remove the discrete component associated with these negative eigenvalues and focus on the decay estimates for the continuous part $u_{m,c}(t, x)$.

For the free case (i.e., $V=0$), note that 
$$
\cos (t \sqrt{\Delta^2-\Delta})=\frac{e^{i t \sqrt{\Delta^2-\Delta}}+e^{-i t \sqrt{\Delta^2-\Delta}}}{2}\ \ \text{and}\ \ 	\frac{\sin (t \sqrt{\Delta^2-\Delta})}{\sqrt{\Delta^2-\Delta}} = \frac{1}{2} \int_{-t}^{t} \cos(s \sqrt{\Delta^2-\Delta})ds.
$$
  Using the decay estimate of  $e^{-it\sqrt{\Delta^2-\Delta}}$  
 (cf. \textit{e.g.}, \cite{CMY}), we can derive
\begin{equation}\label{free estimate}
	\big\|\cos \big(t \sqrt{\Delta^2-\Delta}\big)\big\|_{L^p(\R^n)\rightarrow L^{p'}(\R^n)}+\Big\|\frac{\sin (t \sqrt{\Delta^2-\Delta})}{t \sqrt{\Delta^2-\Delta}}\Big \|_{L^p(\R^n)\rightarrow L^{p'}(\R^n)} \lesssim|t|^{-n\big(\frac{1}{p}-\frac{1}{2}\big)},
\end{equation}
where $\frac{1}{p}+\frac{1}{p'}=1$ and  $1\leq p\le 2.$

Thus,  as an application of  Theorem \ref{main_theorem}, taking into account  \eqref{Lp-bound of f(H)} and  \eqref{free estimate}, we establish  the following  time decay estimates for  the beam equation \eqref{beam_equation}.     
\begin{corollary}\label{theorem beam}
		Let $n\geq5$ and $P_{ac}(H)$ denote the projection onto the absolutely continuous spectrum of $H$. Suppose that $H$ and $V$ satisfy the same conditions as stated  in Theorem \ref{main_theorem}.  Then for any  $1\leq p\le 2,$
		      \begin{equation*}
		      	\big\|\cos \big(t \sqrt{H}\big)P_{ac}(H)\big\|_{L^p(\R^n)\rightarrow L^{p'}(\R^n)}+\Big\|\frac{\sin (t \sqrt{H})}{t \sqrt{H}}P_{ac}(H)\Big \|_{L^p(\R^n)\rightarrow L^{p'}(\R^n)} \lesssim|t|^{-n\big(\frac{1}{p}-\frac{1}{2}\big)}.
		      \end{equation*}	   
	\end{corollary}                                           \begin{remark}
Under the same conditions as stated in Remark \ref{W epsilon},  combine with Corollary \ref{theorem beam}  and $e^{\pm it\sqrt{\mathcal{H}_\epsilon}}P_{ac}(\mathcal{H}_\epsilon)=U_\epsilon \big(e^{\pm it\sqrt{\epsilon^{-1}H_\epsilon}}P_{ac}(H_\epsilon)\big)U_\epsilon ^{-1}$, then for any $1\leq p\le 2$ and $n\geq5,$
\begin{equation*}
		      	\big\|\cos \big(t \sqrt{\mathcal{H}_\epsilon}\big)P_{ac}(\mathcal{H}_\epsilon)\big\|_{L^p(\R^n)\rightarrow L^{p'}(\R^n)}+\Big\|\frac{\sin (t \sqrt{\mathcal{H}_\epsilon})}{t \sqrt{\mathcal{H}_\epsilon}}P_{ac}(\mathcal{H}_\epsilon)\Big \|_{L^p(\R^n)\rightarrow L^{p'}(\R^n)} \lesssim_\epsilon|t|^{-n\big(\frac{1}{p}-\frac{1}{2}\big)}.
		      \end{equation*}	
\end{remark}
                                                                                                                                 \subsection{Further related results}
	\label{subsection_known_result}
	The $L^p$-boundedness of wave operators for homogeneous Schr\"odinger type operators  $H=(-\Delta)^m+V(x)$ on $\R^n$ with $m\ge1$, has been extensively studied in the literature. More specifically:
    
	For the classical Schr\"odinger operator $-\Delta+V(x)$  (i.e. $m=1$), since the seminal work by Yajima in \cite{Yajima-JMSJ-95}, there has been a significant body of works focused on the $L^p$-boundedness of the wave operators $W_\pm$.  In particular, in the one-dimensional space ($n=1$), the wave operators $W_\pm$ are bounded on $L^p(\mathbb{R})$ for $1<p<\infty$ in both regular and zero resonance cases. However, they are generally unbounded for $p=1$ and $p=\infty$ (see, for example, \cite{ArYa,DaFa,Weder}).
	In the regular case, for dimension $n=2$, the wave operators $W_\pm$ are bounded on $L^p(\mathbb{R}^2)$ for $1<p<\infty$, but the results at the endpoints are yet unknown (see \cite{Yajima-CMP-99, Jensen_Yajima_2D}). For dimensions $n\geq3$, the wave operators $W_\pm$ are bounded on $L^p(\mathbb{R}^n)$ for $1\leq p\leq\infty$ in the regular case (see, for example, \cite{BeSc,Yajima-1995,Yajima-JMSJ-95}). However, the presence of threshold resonances narrows down the range of values for $p$, depending on the dimension $n$ and asymptotic behaviors of zero resonant states and zero eigenfunctions (as discussed in \cite{EGG,Finco_Yajima_II,Goldberg-Green-Advance,Goldberg-Green-Poincare,Jensen_Yajima_4D, Yajima_2006,Yajima_2016,Yajima_2018,Yajima_2021arxiv,Yajima_2022arxiv}).
	
		For the higher order Schr\"odinger operator $(-\Delta)^m+V(x)$ with $m\ge 2$,    the first result was obtained by Goldberg and Green \cite{GoGr21} for the case $(m, n) = (2, 3)$, where the $L^p$-boundedness of wave operators was established for $1 < p < \infty$ if the zero energy is a regular point.  Additionally, Mizutani--Wan--Yao  extended the analysis to the endpoint $p=1, \infty$ and resonance cases (see  \cite{MWY23, MWY_ArXiv23_2}). 	For the case of $(m,n)=(2,1)$, Mizutani--Wan--Yao \cite{MiWYa} have obtained the $L^p$-boundedness for all $1<p<\infty$ in the cases of all kinds of resonances and counterexamples  at the endpoint $p=1,\infty$. 
		For $n>2m\ge4$, Erdo\u{g}an--Green \cite{Erdogan-Green21,Erdogan-Green23} proved the $L^p$-boundedness for all $1\le p\le \infty$ if the zero energy is a regular point and the potential $V(x)$ is sufficiently smooth. For further results in the cases $n > 4m - 1$ and $n > 4m$, see \cite{EGG23, EGL}, respectively. 
	 Moreover,  the case $n=2m=4$ was considered by Galtbayar--Yajima \cite{Yajima_2024JST} where the $L^p$-boundedness was proved for $1<p<p_0$ with suitable $p_0$ depending on the type of the singularity at the zero energy.
	
	\subsection{The outline of proof}
	Here, we provide a brief outline of the main ideas involved in the proof of  Theorem \ref{main_theorem}.
	The starting point is the following stationary formula:
		\begin{align}
		\label{stationary}
		W_\pm=\Id-\frac{1}{2\pi i}\int_0^\infty  R_V^\mp(\lambda)V\left(R_0^+(\lambda)-R_0^-(\lambda)\right)d\lambda.
	\end{align}
	To explain the formula \eqref{stationary}, we need to introduce some notations.  Let
	\begin{align*}
	R_\Delta(z)=(-\Delta-z)^{-1},\ \ 	R_0(z)=(\Delta^2-\Delta-z)^{-1},\ \ 
		R_V(z)=(H-z)^{-1},\quad z\in \C\setminus[0,\infty),
	\end{align*}
	be the resolvents of operators $-\Delta,$  $H_0=\Delta^2-\Delta$ and $H=\Delta^2-\Delta+V(x)$, respectively. Then $ R^\pm_\Delta(\lambda),$  $R^\pm_0(\lambda)$ and $R^\pm_V(\lambda)$  are defined as their proper boundary values (limiting resolvents) on $(0,\infty)$, namely for $\lambda>0,$
	\begin{align}\label{boundary}
		R_\Delta^\pm(\lambda)=\lim_{\ep \searrow 0}R_\Delta(\lambda\pm i\ep),\ \ 
	R^\pm_0(\lambda)=\lim_{\epsilon \searrow 0}R_0(\lambda\pm i\epsilon ),\ \  R_V^\pm(\lambda)=\lim_{\epsilon  \searrow 0}R_V(\lambda\pm i\epsilon ).
	\end{align}
	More specifically, the existence of $R^\pm_0(\lambda)$ as bounded operators from $L^2_s(\R^3)$ to $L^2_{-s}(\R^3)$ with $s>1/2$ follows from the limiting absorption principle for the resolvent $R_\Delta(z)$  (see e.g.  Agmon \cite{Agmon}) and   the following splitting identity  for $z\in\C\setminus[0, +\infty)$ with $0< \arg z<2\pi:$
	\begin{align}\label{the split}
		R_0(z)
		=\frac{1}{\sqrt{1+4z}}\left(R_\Delta\Big(-\frac{1}{2}+\frac{1}{2}\sqrt{1+4z}\Big)-R_{\Delta}\Big(-\frac{1}{2}-\frac{1}{2}\sqrt{1+4z}\Big)\right).
	\end{align}
	The existence of $R^\pm_V(\lambda)$ for $\lambda>0$  with  certain decay conditions on the potential has been also already shown (see  e.g. \cite{Agmon, Kuroda}).
	
	Combining with  \eqref{stationary} and  iterations of the resolvent identity, 
	$$R_V^-(\lambda)=\sum_{\ell=0}^{2k-1}(-1)^\ell (R_0^-(\lambda)V)^\ell R_0^-(\lambda)+(R_0^-(\lambda)V)^kR_V^-(\lambda)
	(VR_0^-(\lambda))^k,
	$$
	it follows that
  $  \Id-W_\pm=\sum_{j=1}^{2k}(-1)^{J-1}W_J^\pm +\Omega^\pm,$
    where
    \begin{align}\label{W_J}
W_J^{\pm}&=\frac{1}{2\pi i}\int_0^\infty(R_0^{\mp}(\lambda)V)^J\big(R_0^+(\lambda)-R_0^-(\lambda)\big)d\lambda,  \ \ \  \text{for}\  J=1, 2, \cdots, 2k.\\
\label{W^LH}
\Omega^\pm&=\frac{1}{2\pi i}\int_0^\infty(R_0^\mp(\lambda)V)^kR_V^\mp(\lambda)
(VR_0^\mp(\lambda))^kV\big(R_0^+(\lambda)-R_0^-(\lambda)\big)d\lambda.
	\end{align}
	By  changing the variable $\lambda\mapsto \eta^4+\eta^2$ and inserting  the identity $1=\chi(\eta)+\widetilde{\chi}(\eta),$  the operator $\Omega^\pm$ is further  decomposed the sum of the low energy  part  $\Omega_L^\pm$ with $\{0\le \eta\ll1\}$  and the high energy part $\Omega_H^\pm$ with $\{\eta\gtrsim1\}.$

	 Combining with  $W_J^-f=\overline{ W_J^+\bar{f}}$ and $\Omega^+f=\overline{ \Omega^-\bar{f}},$   all the problems boil down to demonstrate $W_J^+$ ($J\in\N^+$),  $\Omega_L^-$ and $\Omega_H^-\in\mathbb{B}(L^p(\R^n))$ for all $1\leq p\leq\infty$ and $n\geq5.$
	 \vskip0.2cm
	 $\bullet$\ \textit{\underline{The $L^p$-boundedness of $W_J^+$.}} For simplicity, we consider the following adjoint operator $W_J^*$,  instead of the operator $W_J^+$:
	 	$$
	 	W_J^* =- \frac{1}{2\pi i} \int_\mathbb{R} R_0^-(\lambda) \big(VR_0^+(\lambda)\big)^J \, d\lambda.
	 	$$
 In order to derive the desired result for $W_J^*$, by density, we  may assume  that  $V_1, \cdots, V_J\in\S(\R^n)$  such that their Fourier transform $\widehat{V_1}, \cdots, \widehat{V_J}\in C_0^\infty(\R^n),$ then   it suffices to show for $f\in\S(\R^n),$  the 
  limit 
	  \begin{equation} \label{ZJ1}
	  	Z_Jf := \lim_{\epsilon_1 \downarrow0} \cdots \lim_{\epsilon_J \downarrow 0} \lim_{\epsilon_0 \downarrow 0}\  \frac{1}{2\pi i} \int_ \R R_0(\lambda-i\epsilon_0)\ \prod_{j=1}^J\big(V_jR_0(\lambda+i\epsilon_j)\big)f d\lambda,
	  \end{equation}
	   exists in the strong topology of $L^p(\R^n)$ and  $Z_J$ can be extended to a bounded operator on $L^p$ for all $1 \leq p \leq \infty$ and $J\in\mathbb{N}^+$ with the proper bounds.
	    
	  Set $(\omega_j, t_j, y_j)\in \big(S^{n-1}\times(0, +\infty)\times\R^n\big)$ for $j=1,2,\cdots, J-1,$ and $(\omega_J, t_J, y_J)\in \big(S^{n-1}\times(0, \rho_J)\times\R^n\big)$ with  $\rho_J:=-\omega_J\big( x-y_J-\sum_{\ell=1}^{J-1}(y_\ell-2t_\ell\omega_\ell)
	\big),$  then
	 $Z_{J}f(x)$ can be expressed as 
	  \begin{align*}
	  	Z_{J }f(x)=\int_{\Sigma_{J-1}}\int_{S^{n-1}}\int_{-\infty}^{\rho_J}\int_{\R^n}F(\omega_1, t_1, y_1,\cdots,\omega_J, t_J, y_J)f(\bar{x}-\gamma_J)dy_Jdt_Jd\omega_J\cdots dy_1dt_1d\omega_1,
	  \end{align*}
	    where   $\bar{x}:=x-2(x\cdot\omega_J)$ and  $\Sigma_{J-1}:=\big(S^{n-1}\times(0, +\infty)\times\R^n\big)^{J-1}, $
	    $\gamma_J:=\bar{y}_J+2t_J\omega_J+\sum_{\ell=1}^{J-1}(\overline{y_\ell-2t_\ell\omega_\ell}).$  Notably,   $F(\omega_1, t_1, y_1,\cdots,\omega_J, t_J, y_J)\in L^1\big(\left(S^{n-1}\times \R\times\R^n\right)^{J}\big)$ satisfies 
	     \begin{align*}
	    	\left\|F\right\|_{L^1\big(\left(S^{n-1}\times \R\times\R^n\right)^{J}\big)}\lesssim_{n, \sigma, J}
	    	\begin{cases}
	    		\prod_{j=1}^J \ \left\|\langle \cdot\rangle^{\frac{9-n}{2}+\sigma+} V_j\right\|_{L^2} & {\rm if}\ \ n=5,6,\\ 
	    		\prod_{j=1}^J\left\| \F\left(\langle\cdot\rangle^{\frac{2(n-4+\sigma)}{n-1}+2\sigma+} V_j\right)\right\|_{L^{\frac{n-1}{n-4+\sigma}}} & {\rm if}\ \ n\geq7.
	    	\end{cases}
	    \end{align*}
	    Apply  Minkowski's inequality and note that $x\mapsto\bar{x}$ is an isometry,  then we obtain that 
	    $$\left\|Z_{J }f\right\|_{L^p}
	    \lesssim\left\|F\right\|_{L^1\big(\left(S^{n-1}\times \R\times\R^n\right)^{J}\big)}\left\|f\right\|_{L^p}$$ for  all $1\leq p\leq\infty,$
	     which yields the desired results.
	     \vskip0.2cm
	      $\bullet$\ \textit{\underline{The $L^p$-boundedness of   $\Omega_L^-$.}}	
	      		Setting $v(x)=\sqrt{|V(x)|}$, $U(x)=\sgn V(x)$ and $M(\eta)=U+vR_0^+(\eta^4+\eta^2)v$, one has the  symmetric second resolvent equation:
	      		$$
	      		R_V^+(\eta^4+\eta^2)V=R_0^+(\eta^4+\eta^2)v M^{-1}(\eta)v,
	      		$$
	      		where $M^{-1}(\eta):=[M(\eta)]^{-1}$ as long as it exists.

	    Denote	$$
	      A(\eta):=\big(R_0^+(\eta^4+\eta^2)V\big)^{k-1}R_0^+(\eta^4+\eta^2),\ \text{and}\  \Gamma(\eta):=wA(\eta)v M^{-1}(\eta)vA(\eta)w$$ with $w=V^{\frac{1}{2}}$,  then  the low energy part $\Omega_L^-$ can be expressed as 
	      \begin{align}\label{out1}
	      	\Omega_L^-:=\frac{1}{\pi i}\int_0^\infty\eta(2\eta^2+1)\chi(\eta)R_0^+(\eta^4+\eta^2)v\Gamma(\eta)v \Big(R_0^+-R_0^-\Big)(\eta^4+\eta^2)d\eta.
	      \end{align}
	Let $|V(x)|\lesssim\langle x \rangle^{-n-4-}.$  Then we can show that
	      \begin{align}\label{out2} 
	      	\widetilde{\Gamma}(x, y)=\max_{\ell\in\big\{0,1,\cdots,\lceil\frac{n}{2}\rceil +1\big\}}\sup_{0<\eta\ll1}\Big | \eta^{\max(0, \ell-1)}\partial_\eta^{\ell}\Gamma(\eta)(x,y) \Big|\lesssim \langle x\rangle^{-\frac{n}{2}-}\langle y\rangle^{-\frac{n}{2}-},
	      \end{align}
	      where $\Gamma(\eta)(x,y)$ is the kernel of the operator $\Gamma(\eta).$
	       Applying  the expression \eqref{out1} and the condition \eqref{out2}, it can be shown that  $\Omega_L^-$ is admissible, i.e. its kernel $\Omega_L^-(x, y)$ satisfies
	       $ \Omega_L^-(x, y)\in L^\infty_yL^1_x\cap L^\infty_xL^1_y, $
	       which yields the desired results.
	\vskip0.2cm
	       $\bullet$\ \textit{\underline{The $L^p$-boundedness of   $\Omega_H^-$.}}  
By iterations and  integrating by parts  $N=\lceil\frac{n}{2}\rceil+1$ times with respect to  $\eta$ for the kernel $\Omega_H^-(x, y)$, provided $\beta >2N+2,$  it can be proved that    
	       \begin{align*}
	       	\Big|\Omega_H^-(x, y)\Big|\lesssim\big\langle x\big\rangle^{-\frac{n-1}{2}}\big\langle y\big\rangle^{-\frac{n-1}{2}}\big\langle|x|\pm|y|\big\rangle^{-N}\in L^\infty_yL^1_x\cap L^\infty_xL^1_y,
	       \end{align*}
	    which establishes  the desired results.
	      \subsection{The organizations of  paper}  In Section \ref{section2}, we first decompose $\Id - W_\pm$ into the sum of a sequence of $W_J^\pm$ and $\Omega^\pm$.  Then for $n \geq 5$, we show that $W_J^+ \in \mathbb{B}(L^p(\mathbb{R}^n))$ for all $1 \leq p \leq \infty$ and $J \in \mathbb{N}^+$.
In Section \ref{section3}, we further decompose $\Omega^\pm$ into the low energy part $\Omega_L^\pm$ and the high energy part $\Omega_H^\pm$. Subsequently, we prove that $\Omega_L^-\in\mathbb{B}(L^p(\mathbb{R}^n))$ for all $1 \leq p \leq \infty$ when  zero is a regular point of $H$.
	      
In Section \ref{section4}, we establish  the high energy part $\Omega_H^- \in \mathbb{B}(L^p(\mathbb{R}^n))$ for all $1 \leq p \leq \infty$.
	   
Finally, Appendix \ref{section5} provides some  technical lemmas needed in the paper.


\section{$L^p$ boundedness of the $J$-th iteration  $W_J^\pm$}\label{section2}
In this section, we show that each $W_J^\pm$ defined by \eqref{W_J} is bounded on $L^p(\R^n)$ for all  $n\geq5$ and $1\leq p\leq\infty.$
It suffices to prove $W_J^+\in\mathbb{B}(L^p(\R^n))$ since $W_J^-f=\overline{ W_J^+\bar{f}}.$  Here,  we adapt some arguments from Yajima in \cite{Yajima-JMSJ-95}   and Erdo\u{g}an-Greeen in \cite{Erdogan-Green21}.
\begin{theorem}\label{theorem WJ}
Let $n\geq5$ and fix $0<\sigma\ll 1$. 
Then $W_J$ defined by \eqref{W_J}  can be extended to a bounded operator on $L^p(\R^n)$ for all $1 \leq p \leq\infty$ and $J\in\mathbb{N}^+$. In particular, 
\begin{align*}
	\left\|W_J^{+}\right\|_{\mathbb{B}(L^p)}\lesssim_{n, \sigma, J}
	\begin{cases}
		\left\|\langle \cdot\rangle^{\frac{9-n}{2}+\sigma+} V\right\| _{L^2}^J\left\|f\right\| _{L^p}& {\rm if}\ \ n=5,6, \\
		\left\| \F\left(\langle\cdot\rangle^{\frac{2(n-4+\sigma)}{n-1}+2\sigma+} V\right)\right\|_{L^{\frac{n-1}{n-4+\sigma}}}^J\left\|f\right\| _{L^p}  & {\rm if}\ \ n\geq7.
	\end{cases}
\end{align*}
\end{theorem}
To streamline the notation somewhat, it is enough to demonstrate the results for the adjoint operator $W_J^*$ of $W_J^+$. 
 
 Note that $W_J^*$ is bounded on $L^2$ and fixed $f\in\S,$
 \begin{align}\label{W_J^*}
 W_J^* f = \lim_{\epsilon_1 \downarrow0} \cdots \lim_{\epsilon_J \downarrow 0} \lim_{\epsilon_0 \downarrow 0}\  -\frac{1}{2\pi i} \int_ \R R_0(\lambda-i\epsilon_0)\ \prod_{j=1}^J\big(VR_0(\lambda+i\epsilon_j)\big)f d\lambda,
 \end{align}
  as a weak limit. Let $(\mu, \nu)=(\frac{9-n}{2}+\sigma+, 2)$ or $(\frac{2(n-4+\sigma)}{n-1}+2\sigma+, \frac{n-1}{n-4+\sigma}).$ Consider that $\mathcal{F}(\langle x \rangle^\mu V) \in L^\nu(\R^n)$ is equivalent to $\widehat{V}\in H^{\mu}_\nu(\R^n),$ the generalized Sobolev space of order $\mu,$ and  $C^\infty_0(\R^n)$ is dense in $H^{\mu}_\nu(\R^n)$. Then in order to prove that $W_J^*\in\mathbb{B}(L^p(\R^n))$ for all $1\leq p\leq\infty$
 and $J\in\mathbb{N}^+$, by the density, it suffices to show  the following Proposition \ref{proposition Z_J}: 
 \begin{proposition}\label{proposition Z_J}
 	Let $n\geq5$ and $V_1, \cdots, V_J\in\S(\R^n)$  such that their Fourier transform $\widehat{V_1}, \cdots, \widehat{V_J}\in C_0^\infty(\R^n)$.  Fixed $f\in \mathcal{S}(\R^n),$ the operator $Z_J$ is defined  by
 	\begin{equation} \label{Z_J}
 		 Z_J f = \lim_{\epsilon_1 \downarrow0} \cdots \lim_{\epsilon_J \downarrow 0} \lim_{\epsilon_0 \downarrow 0}\  \frac{1}{2\pi i} \int_ \R R_0(\lambda-i\epsilon_0)\ \prod_{j=1}^J\big(V_jR_0(\lambda+i\epsilon_j)\big)f d\lambda,
 	\end{equation}
 	where the product should be taken from the left to the right. Then,   $Z_J$ can be extended to a bounded operator on $L^p(\R^n)$ for all $1 \leq p \leq \infty$ and $J\in\mathbb{N}^+$. In particular, fixed $0<\sigma\ll 1,$
 	\begin{align*}
 	\left\|Z_J\right\|_{\mathbb{B}(L^p)}\lesssim_{n, \sigma, J}
 	\begin{cases}
 		\prod_{j=1}^J \ \left\|\langle \cdot\rangle^{\frac{9-n}{2}+\sigma+} V_j\right\|_{L^2} & {\rm if}\ \ n=5,6,\\ 
 		\prod_{j=1}^J\left\| \F\left(\langle\cdot\rangle^{\frac{2(n-4+\sigma)}{n-1}+2\sigma+} V_j\right)\right\|_{L^{\frac{n-1}{n-4+\sigma}}} & {\rm if}\ \ n\geq7.
 	\end{cases}
 		\end{align*}
 \end{proposition}
  Next, we show the proof of Proposition \ref{proposition Z_J}, which is divided into the following two subsections:
 \subsection{The existence of $Z_Jf$ as $L^p$-limit  }In this subsection, we  show that the limit \eqref{Z_J3}
  	exists in the strong topology of $L^p(\R^n),$ i.e.  $Z_Jf$ can be written as the expression \eqref{Z_J3} in the sense of $L^p$  for all 
   $1 \leq p \leq \infty.$

 Let  $\vec{\epsilon}=(\epsilon_1, \epsilon_2, \cdots, \epsilon_J )$ and denote
\begin{align}\label{Z_J,epsilon}
	Z_{J, \vec{\epsilon}, \epsilon_0}f(x)= \frac{1}{2\pi i} \int_ \R\Big[ R_0(\lambda-i\epsilon_0)\ 
	\prod_{j=1}^J\big(V_jR_0(\lambda+i\epsilon_j)\big)f\Big](x) d\lambda,   \ \ \ \  \forall f\in\S.
\end{align}

 Take the Fourier transform for $Z_{J, \vec{\epsilon}, \epsilon_0}f(x)$ and set $\vec{k}=(k_1, k_2, \cdots, k_J ),$ then
 \begin{align}\label{FZ}
 &	\F(Z_{J, \vec{\epsilon},  \epsilon_0}f)(\xi)=\frac{-i}{(2\pi)^{\frac{n}{2}J+1}}\int_\R
 	\frac{1}{|\xi|^4+|\xi|^2-\lambda+i\epsilon_0}\times\nonumber\\
 		&\ \ \ \ \ \ \ \ \ \int_{\R^{Jn}}
 	\left(\prod_{j=1}^J
 	\frac{\widehat{V_j}(k_j)}{|\xi-\sum_{\ell=1}^{j}k_{\ell}|^4+|\xi-\sum_{\ell=1}^{j}k_{\ell}|^2-\lambda-i\epsilon_j}\right)\hat{f}(\xi-\sum_{\ell=1}^{J}k_{\ell})d\vec{k}d\lambda.
 \end{align}
  Perform the $\lambda$ integration by the residue theorem and denote $Z_{J, \vec{\epsilon} }f(x)	:=\lim_{\epsilon_0\rightarrow0^+}Z_{J, \vec{\epsilon}, \epsilon_0}f(x),$  then taking the limit $\epsilon_0\to 0^+$ for \eqref{FZ}, one has
  \small
      \begin{align*}
     \F(Z_{J, \vec{\epsilon} }f)(\xi)=\frac{1}{(2\pi)^{\frac{n}{2}J}}\int_{\R^{Jn}}&
     		\left(\prod_{j=1}^J	\frac{\widehat{V_j}(k_j)}{|\xi-\sum_{\ell=1}^{j}k_{\ell}|^4+|\xi-\sum_{\ell=1}^{j}k_{\ell}|^2-|\xi|^4-|\xi|^2-i\epsilon_j}\right)\hat{f}(\xi-\sum_{\ell=1}^{J}k_{\ell})d\vec{k}.
     	\end{align*}
 Change the variables $(k_1, k_2, \cdots, k_J )$  by 
  $(k_1-k_0, k_2-k_1, \cdots, k_J -k_{J-1})$ with $k_0=0,$  then
  \begin{align}
  \label{F(Z_J,1)}
  	\F(Z_{J, \vec{\epsilon} }f)(\xi)=\frac{1}{(2\pi)^{\frac{n}{2}J}}\int_{\R^{Jn}}
  	\left(\prod_{j=1}^J	\frac{\widehat{V_j}(k_j-k_{j-1})}{|\xi-k_{j}|^4+|\xi-k_{j}|^2-|\xi|^4-|\xi|^2-i\epsilon_j}\right)
  	\hat{f}(\xi-k_{J})d\vec{k}.
  \end{align}
   Define  the multiplier operator $T_{k, \epsilon}$  depending on the parameters $k\in\R^n\backslash{0}$ and $\epsilon>0$ by 
   \begin{align}\label{T_k,1}
   	T_{k, \epsilon}f(x)=\frac{1}{(2\pi)^{\frac{n}{2}}}\F^{-1}
   	\left(\frac{\hat{f}(\xi)}{|\xi-k|^4+|\xi-k|^2-|\xi|^4-|\xi|^2-i\epsilon}\right)(x).
   	\end{align}
  Then  $Z_{J}f(x)$ can be expressed as 
  \begin{align}\label{Z_J3}
  	Z_{J }f(x)=\lim_{\epsilon_1 \downarrow0} \cdots \lim_{\epsilon_J \downarrow 0}\int_{\R^n} T_{k_1, \epsilon_1}\int_{\R^n}T_{k_2, \epsilon_2}\int_{\R^n}\dots\int_{\R^n}K_J(k_1,\cdots, k_J)T_{k_J, \epsilon_J}f_{k_J} (x)d\vec{k},
  	\end{align}
    as a weak limit, 
  	where 
  	\begin{align}\label{KJ}
  \vec{k}=(k_1, k_2, \cdots, k_J ),\ \ f_{k_J} (x)=e^{ik_Jx}f(x),\ \ f\in\S, \ \text{and}\ 	K_J(k_1,\cdots, k_J)=\prod_{j=1}^J	\widehat{V_j}(k_j-k_{j-1}).
  	\end{align}
     As a result,   establishing Proposition \ref{proposition Z_J} reduces to proving that the limit \eqref{Z_J3} exists in the strong topology of $L^p(\mathbb{R}^n)$ for all $1 \leq p \leq \infty$ and satisfies the bounds in Proposition~\ref{proposition Z_J}. This subsection addresses the former, which is given by the following Theorem \ref{theorem Lp limit}.
  \begin{theorem} \label{theorem Lp limit}
  Let  $T_{k, \epsilon}$ and $Z_J$ be defined by \eqref{T_k,1} and \eqref{Z_J3}, respectively. It holds that
  \begin{itemize}
      \item[(i)] Fix $f(k,x)\in\S(\R^n_k, \S(\R^n_x))$ and set $h_{k}(x):
      =i(2\pi)^{-n}\F^{-1}\big((|\cdot-k|^2+|\cdot|^2+1)^{-1}\big)(x),$ then
   \begin{align}\label{T4}
 	\lim_{\epsilon\to 0^+}	T_{k, \epsilon}f(k, x)
 	=\frac{1}{|k|}\int_0^\infty e^{-i|k|t}\int_{\R^n}h_k(y)f(k, x+2t\omega-y)dydt\ \     a.e. \ \text{for}\  x\in\R^n.
 \end{align} 
 \item[(ii)] The limit \eqref{Z_J3} exists in the strong topology of $L^p(\mathbb{R}^n)$ for all $1 \leq p \leq \infty,$ namely
\begin{align}\label{Z8}
 	Z_{J }f(x)=\int_{\R^n} T_{k_1 }\int_{\R^n}T_{k_2 }\int_{\R^n}\cdots\int_{\R^n}K_J(k_1,\dots, k_J)T_{k_J}f_{k_J} (x)d\vec{k},\ \ J\in\N^+,
 \end{align}
 in the sense of $L^p$ for all $1 \leq p \leq \infty,$
where $T_{k}f(k, x):=	\lim_{\epsilon\to 0^+}	T_{k, \epsilon}f(k, x)$ and $\vec{k},$ $f_{k_J},$ $K_J$  are defined by   \eqref{KJ}.
 \end{itemize}
  \end{theorem}
  We first deal with 
 {\bf the proof of Theorem \ref{theorem Lp limit} (i),} which is divided into the next three steps:
\vskip 0.1cm
    {\bf Step I.} \underline {Express the multiplier operator $T_{k, \epsilon}$ by \eqref{T_k3}.}  
   \vskip 0.1cm
    Note that 
   $$|\xi-k|^4+|\xi-k|^2-|\xi|^4-|\xi|^2-i\epsilon
   =\frac{1}{\phi_k(\xi)}\left(|\xi-k|^2-|\xi|^2-i\epsilon\phi_k(\xi)\right),
   $$
  	where 
  \begin{align}\label{phi}
  	\phi_k(\xi)	=\left(|\xi-k|^2+|\xi|^2+1\right)^{-1}.
  \end{align}
  By $	\phi_k(\xi)	>0$, one has 
  $$\frac{1}{|\xi-k|^2-|\xi|^2-i\epsilon\phi_k(\xi)}=i\int_0^\infty e^{-\epsilon\phi_k(\xi)t} e^{-i|\xi-k|^2t+i|\xi|^2t}dt.$$
  Thus $T_{k, \epsilon}$ can be expressed as
  \begin{align}
  	T_{k, \epsilon}f(x)
  	=\frac{i}{(2\pi)^{\frac{n}{2}}}\int_0^\infty
  	e^{-i|k|^2t}\F^{-1}\Big(\phi_k(\xi)e^{-\epsilon\phi_k(\xi)t}\F \big[f(\cdot+2t\omega )\big](\xi)\Big)(x)dt.
  	\end{align}
  We utilize the change of variable $|k|t\longmapsto t$
  and take  $k=s\omega,\  (s, \omega)\in \R^+\times S^{n-1},$  where $S^{n-1}$ denotes the unit sphere,  then by 
  the formula for the Fourier transform of a product, 
  \begin{align}\label{T_k3}
  	T_{k, \epsilon}f(x)
  	=s^{-1}\int_0^\infty e^{-ist} \left(h_{s\omega, \epsilon t}\ast f\right)(x+2t\omega)dt,
  	\end{align}
  where   $\ast $ denotes convolution and 
  \begin{align}\label{h}
  	h_{s\omega, \epsilon t}(x)=\frac{i}{(2\pi)^{n}}\F^{-1}\Big(\phi_{s\omega}(\xi)e^{-\frac{\phi_{s\omega}(\xi)}{s} \epsilon t}\Big).
  	\end{align}
   \vskip 0.1cm
    {\bf Step II.} \underline {Analyze $h_{s\omega, \epsilon t}$ for all $\epsilon\geq0$ by Lemma \ref{lemma h}.}  
   \vskip 0.1cm 

 \begin{lemma}\label{lemma h} 
 	Fix $0<\sigma\ll 1$ and $h_{s\omega, \epsilon t}(x)$ is defined by \eqref{h}. Then
 	\begin{align}\label{estimate h}
 		\sup_{\epsilon\geq0}\left| \partial_s^jh_{s\omega, \epsilon t}(x)\right|\lesssim s^{n-2-j}
 		\min\big(\frac{1}{|sx|^{n-\sigma}},  \frac{1}{|sx|^{n+\sigma}}\big), \  \text{for}\ j=0,1,2.
 		\end{align}
 		In particular,  $\partial_s^jh_{s\omega, \epsilon t}$ converges to $\partial_s^jh_{s\omega}$ $a.e.$ and in $L^{1}(\R^n)$ as $\epsilon\rightarrow0,$  for $j=0,1,2,$ where $h_{s\omega}:=h_{s\omega, 0}$.
 \end{lemma}
 \begin{proof}
 We  first claim the estimate \eqref{estimate h} for $h_{s\omega}$ with $j=0.$ Note that 
 \begin{align}\label{lemma h1}
 	h_{s\omega}(x)&=\frac{i}{(2\pi)^{n}}\F^{-1}\Big(\frac{1}{|\xi-s\omega|^2+|\xi|^2+1}\Big)(x)\nonumber\\
 	&=\frac{i}{(2\pi)^{n}}s^{n-2}\F^{-1}\Big(\frac{1}{|\xi-\omega|^2+|\xi|^2+s^{-2}}\Big)(sx):=\frac{i}{(2\pi)^{n}}s^{n-2}\F^{-1}\Big(\psi_{s\omega}(\xi)\Big)(sx).
 	\end{align}
 Since  $\left|\nabla_\xi^N \psi_{s\omega}(\xi)\right|\lesssim\langle\xi\rangle^{-2-N},$ then
 $\nabla_\xi^N \psi_{s\omega}(\xi)\in L^1(\R^n)$ provided $N>n-2.$ 
 Furthermore, its inverse Fourier transform is a bounded function, namely:
 $$\left|\F^{-1}\Big(\nabla_\xi^N\psi_{s\omega}(\xi)\Big)(x)\right|=|x|^N\left|\F^{-1}\Big(\psi_{s\omega}(\xi)\Big)(x)\right| \lesssim1.$$
 Thus, by choosing $N=n\pm\sigma $ with $0<\sigma\ll 1,$ we obtain that
 \begin{align*}
 	\left| h_{s\omega}(x)\right|=\eta+s^{n-2}O\left(\min\big(\frac{1}{|sx|^{n-\sigma}},  \frac{1}{|sx|^{n+\sigma}}\big)\right),
 	\end{align*}
 	where $\eta$ is a distribution supported at zero. Observe that $\F\left(h_{s\omega}\right)=\frac{i}{(2\pi)^{n}}\left(|\xi-k|^2+|\xi|^2+1\right)^{-1}$ tends to zero as $|\xi|\longrightarrow\infty,$ then $\eta=0,$ which implies that we complete the proof of  the estimate \eqref{estimate h} for $h_{s\omega}$ with $j=0.$
 
 As for $\partial_s^jh_{s\omega}(x)$ with $j=1,2,$  we can check that 
 \begin{align*}
 &	\partial_sh_{s\omega}(x)=\frac{i}{(2\pi)^{n}}s^{n-3}\F^{-1}\Big(2(\xi\omega-1)\psi_{s\omega}^2\Big)(sx),\\
 &\partial_s^2h_{s\omega}(x)=\frac{i}{(2\pi)^{n}}s^{n-4}\F^{-1}\Big(8(\xi\omega-1)^2\psi_{s\omega}^3-2\psi_{s\omega}^2\Big)(sx).
 	\end{align*}
 By virtue of 
  \begin{align*}
 		\nabla_\xi^N\Big(2(\xi\omega-1)\psi_{s\omega}^2\Big)\lesssim\langle\xi\rangle^{-3-N},\ \text{and}\ \ 
 	\nabla_\xi^N\Big(8(\xi\omega-1)^2\psi_{s\omega}^3-2\psi_{s\omega}^2\Big)\lesssim\langle\xi\rangle^{-4-N},
 \end{align*}
 and analogous  to the case  for $j=0,$ we obtain the desired conclusion for  $\partial_s^jh_{s\omega}(x)$ with $j=1,2.$

 Next, we start to deal with $\sup_{\epsilon>0}\left| \partial_s^jh_{s\omega, \epsilon t}(x)\right|$ for $j=0,1,2.$ Note that
  \begin{align*}
 	&h_{s\omega,\epsilon t}(x)=\frac{i}{(2\pi)^{n}}s^{n-2}\F^{-1}\Big(\psi_{s\omega}e^{-\frac{\epsilon t}{s^3}\psi_{s\omega}}\Big)(sx),\\
 		&	\partial_sh_{s\omega,\epsilon t}(x)=\frac{i}{(2\pi)^{n}}s^{n-3}\F^{-1}\left(2(\xi\omega-1)\psi_{s\omega}^2e^{-\frac{\epsilon t}{s^3}\psi_{s\omega}}+\Big(\psi_{s\omega}-2(\xi\omega-1)\psi_{s\omega}^2\Big)\frac{\epsilon t}{s^3}\psi_{s\omega}e^{-\frac{\epsilon t}{s^3}\psi_{s\omega}}\right)(sx),\\
  &	\partial_s^2h_{s\omega, \epsilon t}(x)=\frac{i}{(2\pi)^{n}}s^{n-4}\F^{-1}\bigg[\Big(8(\xi\omega-1)^2\psi_{s\omega}^3-2\psi_{s\omega}^2\Big)e^{-\frac{\epsilon t}{s^3}\psi_{s\omega}}+\Big(-2\psi_{s\omega}+(6\xi\omega-4)\psi_{s\omega}^2\\
  	&-16(\xi\omega-1)^2\psi_{s\omega}^3\Big)\frac{\epsilon t}{s^3}\psi_{s\omega}e^{-\frac{\epsilon t}{s^3}\psi_{s\omega}}
  	+\Big(\psi_{s\omega}-4(\xi\omega-1)\psi_{s\omega}^2+4(\xi\omega-1)^2\psi_{s\omega}^3\Big)\big(\frac{\epsilon t}{s^3}\psi_{s\omega}\big)^2 e^{-\frac{\epsilon t}{s^3}\psi_{s\omega}}\bigg](sx).
  \end{align*}
 Consider that $\sup_{\alpha>0} \alpha^{r}e^{-\alpha}\lesssim 1$ for a fixed $r>0$, which yields that the bound in the bracket is independent of  $\epsilon,$   and combine with 
 $\left|\nabla_\xi^N \psi_{s\omega}(\xi)\right|\lesssim\langle\xi\rangle^{-2-N},$  analogous to the proof  of the case for $\partial_s^jh_{s\omega}(x)$ with $j=0,1,2,$ we obtain the estimate \eqref{estimate h}.
 
 In the following, we prove that $\partial_s^jh_{s\omega, \epsilon t}$ converges to $\partial_s^jh_{s\omega}$ $a.e.$ and in $L^{1}(\R^n)$ as $\epsilon\rightarrow0,$  for $j=0,1,2.$ According to $\left|\nabla_\xi^N \psi_{s\omega}(\xi)\right|\lesssim\langle\xi\rangle^{-2-N}$ and 
 \begin{align*}
 &\left|h_{s\omega, \epsilon t}(x)-h_{s\omega}(x)\right|=\frac{i}{(2\pi)^{n}}s^{n-2}\F^{-1}\left(\big (e^{-\frac{\epsilon t}{s^3}\psi_{s\omega}}-1\big)\psi_{s\omega}\right)(sx),\\ &e^{-\frac{\epsilon t}{s^3}\psi_{s\omega}}-1=-\frac{\epsilon t}{s^3}\psi_{s\omega}\int_0^1 e^{-\theta\frac{\epsilon t}{s^3}\psi_{s\omega}}d\theta,
 \end{align*}
 it's not hard to obtain that 
 $$\left| \nabla_\xi^N\left(\psi_{s\omega}\big(e^{-\frac{\epsilon t}{s^3}\psi_{s\omega}}-1\big)\right)\right|
 \lesssim\frac{\epsilon t}{s^3}\langle\xi\rangle^{-4-N},$$
 which gives that 
 $$\left|h_{s\omega, \epsilon t}(x)-h_{s\omega}(x)\right|\lesssim \epsilon t s^{n-5}	\min\big(\frac{1}{|sx|^{n-\sigma}},  \frac{1}{|sx|^{n+\sigma}}\big),$$
 followed  from the way to get the estimate \eqref{estimate h}.
 Hence,  $h_{s\omega, \epsilon t}$ converges to $h_{s\omega}$ $a.e.$ as $\epsilon\rightarrow0,$  and by Lebesgue's dominated convergence theorem,  $h_{s\omega, \epsilon t}\longrightarrow h_{s\omega}$ in $L^{1}(\R^n)$ as $\epsilon\rightarrow0.$ 
 Similarly, 
 \begin{align*}
  \left|\partial_s^jh_{s\omega, \epsilon t}(x)-\partial_s^jh_{s\omega}(x)\right|\lesssim \epsilon t s^{n-5-j}	\min\big(\frac{1}{|sx|^{n-\sigma}},  \frac{1}{|sx|^{n+\sigma}}\big),  \ \ \text{for}\ j=1,2, 
 \end{align*}
 which gives that 
  $\partial_s^jh_{s\omega, \epsilon t}$ converges to $\partial_s^jh_{s\omega}$ $a.e.$ and in $L^{1}(\R^n)$ as $\epsilon\rightarrow0,$  for $j=1,2.$ 
  Thus we complete the proof.
 \end{proof}
    {\bf Step III.} \underline {Compelte the proof of Theorem \ref{theorem Lp limit} (i).}  
   \vskip 0.2cm 
Let $f(k,x)\in\S(\R^n_k, \S(\R^n_x))$ and take into account Lemma \ref{lemma h}, then by \eqref{T_k3} and  Lebesgue's dominated convergence theorem, 
 \begin{align*}
 	\lim_{\epsilon\to 0^+}	T_{k, \epsilon}f(k, x)
 	=\frac{1}{|k|}\int_0^\infty e^{-i|k|t}\int_{\R^n}h_k(y)f(k, x+2t\omega-y)dydt\ \     a.e. \ \text{for}\  x\in\R^n.
 \end{align*}
 Thus we finish the proof of Theorem \ref{theorem Lp limit} (i).
 \vskip0.2cm
 Next,  we come to show {\bf the proof of Theorem \ref{theorem Lp limit} (ii),}  which is divided into the following two claims:
 \vskip 0.1cm 
  {\bf Claim A.} \underline {Show the proof of Theorem \ref{theorem Lp limit} (ii) for $J=1$ by Lemma \ref{lemma G}.} 
    \vskip 0.1cm 
 Set  $T_{k}f(k, x):=	\lim_{\epsilon\to 0^+}	T_{k, \epsilon}f(k, x)$, $f(k,x)\in\S(\R^n_k, \S(\R^n_x))$ and define 
 \begin{align}\label{G1}
 &	G_\epsilon f (x):=\int_{\R^n}T_{k, \epsilon}f(k, x)dk=\int_0^\infty\int_{S^{n-1}}\int_0^\infty e^{-ist}s^{n-2}\big(h_{s\omega, \epsilon t}\ast f(s\omega, \cdot)\big)(x+2t\omega)dsd\omega dt,\\
\label{G0}
 &G_0f (x):=\int_{\R^n}T_{k}f(k, x)dk=\int_0^\infty\int_{S^{n-1}}\int_0^\infty e^{-ist}s^{n-2}\big(h_{s\omega}\ast f(s\omega, \cdot)\big)(x+2t\omega)dsd\omega dt.
 \end{align}
 Actually,  let $f(k_1 , x):=\widehat{V_1}(k_1)f_{k_1}(x) \in \S(\R^{n}_{k_1}, \S(\R^n_x))$, then by \eqref{Z_J3},  we find that for $J=1,$  
 	$$Z_{1}f(x)=\lim_{\epsilon_1 \downarrow0}\int_{\R^n}
 	\widehat{V_1}(k_1)	T_{k_1, \epsilon_1}f_{k_1} (x)dk_1=\lim_{\epsilon_1 \downarrow0}G_{\epsilon_1}f(x),$$
where $f$ appeared in $G_{\epsilon_1}f$ denotes $f(k_1 , x)$ for short.

 Next, we demonstrate that $G_\epsilon f$ converges to $G_0 f$ in $L^p$ as $\epsilon \to 0$ for all $1 \leq p \leq\infty$, as stated in Lemma \ref{lemma G}. We first present the  Littlewood-Paley decomposition.
 
  Let $\varphi $ be a smooth function such that $\varphi(s)=1$ for
 $|s| \leq \frac{1}{2}$ and $\varphi(s)=0$ for $|s| \geq 1$. $\varphi_N(s):=\varphi\left(2^{-N} s\right)-\varphi\left(2^{-N+1} s\right), N \in \mathbb{Z}$. Then $\varphi_N(s)=\varphi_0\left(2^{-N} s\right),    \operatorname{supp} \varphi_0 \subset$ $\left[\frac{1}{4}, 1\right]$ and
 \begin{equation}\label{varphi_0}
 	\sum_{N=-\infty}^{\infty} \varphi_0\left(2^{-N} s\right)=1,\quad  s \in \mathbb{R} \backslash\{0\}.
 \end{equation}
 Additionally, we give some notations. For multi-indices $\alpha=(\alpha_1,\cdots, \alpha_n),$ ($\alpha_1,\cdots, \alpha_n \in\N^+\cup\{0\}$), we denote
 $D_x^\alpha:=\partial^{\alpha_1}_{ x_1}\partial^{\alpha_2}_{ x_2}\cdots\partial^{\alpha_n}_{ x_n},$  and $|\alpha|:=\alpha_1+\alpha_2+\cdots+\alpha_n,$ where  $$\partial^{\alpha_j}_{ x_j}=\frac{\partial^{\alpha_j}}{\partial x_j^{\alpha_j}},\  \ j=1,2,\cdots, n,\ \text{and}\  x=(x_1, x_2, \cdots, x_n)\in\R^n.$$
 \begin{lemma}\label{lemma G}
 	Let $\epsilon>0$ and $f(k,x)\in\S(\R^n_k,\S (\R^n_x))$. Then for all $n\geq5$ and $1\leq p\leq\infty,$
 	\begin{align} \label{estimate G}
 		\left\| G_\epsilon f\right\|_{L^p}\lesssim\int_{\R^n}
 	|k|^{-4-}\langle k \rangle^{1+}\sum_{|\alpha|\leqslant2}
 	\left\| D_k^\alpha f(k, \cdot) \right\|_{L^p}dk.
 \end{align}
 Moreover, $ G_\epsilon f$ converges to $ G_0 f$ in $L^p(\R^n)$ as $\epsilon$ tends to zero for all $1\leq p\leq\infty.$
 \end{lemma}
 \begin{proof}
 	By  \eqref{G1} and  Littlewood-Paley decomposition \eqref{varphi_0}, 
 it follows  that 
 	\begin{align}\label{G2}
 		\left| G_\epsilon f(x)\right|\leq
 		&\sum_{N=-\infty}^{\infty} \int_0^\infty\int_{S^{n-1}}|K_{N, \epsilon}(x, \omega, t)|
 		d\omega dt,
 		\end{align}
 		where  
 			\begin{align}\label{K}
 				K_{N, \epsilon}(x, \omega, t)=\int_0^\infty e^{-ist}\varphi_0\left(2^{-N} s\right) s^{n-2}\big(h_{s\omega, \epsilon t}\ast f(s\omega, \cdot)\big)(x+2t\omega)ds.
 					\end{align}
 					Note that $\supp \varphi_0\left(2^{-N} s\right)\subset [2^{N-2}, 2^N]$ and \eqref{estimate h}:
 					$$
 					\sup_{\epsilon\geq0}\left| \partial_s^jh_{s\omega, \epsilon t}(x)\right|\lesssim s^{n-2-j}H(sx), \ \  
 				H(sx):=	\min\big(\frac{1}{|sx|^{n-\delta}},  \frac{1}{|sx|^{n+\delta}}\big), \ \   j=0,1,2.$$
 	Then combining with  $s\sim2^N$  such that $s^{n-4}=(2^N)^{n-4},$
 	\begin{align}\label{K1}			
 		|K_{N, \epsilon}(x, \omega, t)|
 		&\lesssim(2^N)^{n-4}\int_{2^{N-2}}^{2^N} \Big[\big(s^nH(s\cdot)\big)\ast
 		\big(\big|\partial_s^j f(s\omega, \cdot)\big|\big)\Big](x+2t\omega)ds.
 	\end{align}
 Integrating by parts twice, 
 $$|K_{N, \epsilon}(x, \omega, t)|=\left|\frac{1}{t^2}\int_{2^{N-2}}^{2^N} e^{-ist}\partial_s^2\Big(\varphi_0\left(2^{-N} s\right) s^{n-2}\big(h_{s\omega, \epsilon t}\ast f(s\omega, \cdot)\big)(x+2t\omega)\Big)ds\right|.$$	
Take into account that $\big|\partial_s^2  \big(\varphi_0\left(2^{-N} s\right) s^{n-2}\big(h_{s\omega, \epsilon t}\ast f(s\omega, \cdot)\big)(x+2t\omega)\big)\big|$ is dominated by 	$$\lesssim s^{n-6}\langle s \rangle^2\Big[\big(s^nH(s\cdot)\big)\ast
	\Big(\sum_{j=0}^{2}\big|\partial_s^j f(s\omega, \cdot)\big|\Big)\Big](x+2t\omega).$$
Moreover, according to $s\sim2^N$  such that $s^{n-6}\langle s \rangle^2=(2^N)^{n-6}
\langle 2^N \rangle^2,$
\begin{align}\label{K2}			
	|K_{N, \epsilon}(x, \omega, t)|\lesssim\frac{1}{t^2}(2^N)^{n-6}
	\langle 2^N \rangle^2
	\int_{2^{N-2}}^{2^N} \Big[\big(s^nH(s\cdot)\big)\ast
	\Big(\sum_{j=0}^{2}\big|\partial_s^j f(s\omega, \cdot)\big|\Big)\Big](x+2t\omega)ds.
\end{align}
 Combining with \eqref{K1} and \eqref{K2},  it follows that 
 	\begin{align}\label{K3}			
 		K_{N, \epsilon}(x, \omega, t)
 		&\lesssim \frac{1}{\langle t \rangle^2}(2^N)^{n-6}
 		\langle 2^N \rangle^2\int_{2^{N-2}}^{2^N} \Big[\big(s^nH(s\cdot)\big)\ast
 		\Big(\sum_{j=0}^{2}\big|\partial_s^j f(s\omega, \cdot)\big|\Big)\Big](x+2t\omega)ds.
 	\end{align}					
By \eqref{K1} and \eqref{K3}, take $0<\theta<1$ and note that $2^N\sim s,$ then $K_{N, \epsilon}(x, \omega, t)=K_{N, \epsilon}^{\theta}K_{N, \epsilon}^{1-\theta}$ satisfies
 \begin{align*}
 	&\lesssim\frac{1}{\langle t \rangle^{2\theta}}\int_{2^{N-2}}^{2^N}
 	s^{n-4-2\theta}
 	\langle s \rangle^{2\theta}
 	 \Big[\big(s^nH(s\cdot)\big)\ast
 	\Big(\sum_{j=0}^{2}\big|\partial_s^j f(s\omega, \cdot)\big|\Big)\Big](x+2t\omega)ds.
 	\end{align*}
 Choose $\theta=\frac{1}{2}+,$ then $|K_{N, \epsilon}(x, \omega, t)|$ is controlled by 
 \begin{align}\label{K5}
 \lesssim	\frac{1}{\langle t \rangle^{1+}}
 	\int_{2^{N-2}}^{2^N} s^{n-5-}\langle s \rangle^{1+}\Big[\big(s^nH(s\cdot)\big)\ast
 	\Big(\sum_{j=0}^{2}\big|\partial_s^j f(s\omega, \cdot)\big|\Big)\Big](x+2t\omega)ds=:K_{N}(x, \omega, t).
 \end{align}
 Consider  \eqref{G2} and   Lebesgue's dominated convergence theorem, then it suffices to proved that 
 $$\varPi(x):=\sum_{N=-\infty}^{\infty} \int_0^\infty\int_{S^{n-1}}K_{N}(x, \omega, t)
 dtd\omega \in L^p(\R^n),$$
 for all $1\leq p\leq\infty$ and its $L^p$ norm  satisfies the bound \eqref{estimate G}.
  
 Indeed,  in virtue of  Minkowski's inequality, Young inequality and 
 performing  the  integral calculation   on $t$,  we conclude that
 $$\big\|  \varPi(\cdot)\big\|_{L^p} \lesssim\int_0^\infty\int_{S^{n-1}}s^{n-5-}\langle s \rangle^{1+}\left\| s^nH(s\cdot)\right\|_{L^1}
 \Big\|\sum_{j=0}^{2}\big|\partial_s^j f(s\omega, \cdot)\big|\Big\|_{L^p}dsd\omega.$$
 Let $k=s\omega$ and 
 note that $\left\| s^nH(s\cdot)\right\|_{L^1}\lesssim1,$ then we finally get the desired result. 
 \end{proof}
 As a result,  
  for $J=1,$
 	$$Z_{1}f(x)=\lim_{\epsilon_1 \downarrow0}\int_{\R^n}
 	\widehat{V_1}(k_1)	T_{k_1, \epsilon_1}f_{k_1} (x)dk_1=\int_{\R^n}
 	\widehat{V_1}(k_1)	T_{k_1}f_{k_1} (x)dk_1,$$
    in the sense of  $L^p$  for all $1\leq p\leq\infty.$
 \vskip 0.1cm 
  {\bf Claim B.} \underline {Complete the proof of Theorem \ref{theorem Lp limit} (ii).} 
    \vskip 0.1cm 
 Return to $Z_J$  defined by \eqref{Z_J3}   for $\forall J\in\N^+,$ and let $f(\vec{k}, x)=K_J(k_1,\cdots, k_J)f_{k_J}\in \S(\R^{nJ}_{\vec{k}}, \S(\R^n_x))$, by the repeated use of Lemma \ref{lemma G} and taking $\epsilon_J,\cdots, \epsilon_1\to 0^+,$ then we conclude that $Z_Jf\in L^p(\R^n)$ for all $1\leq p\leq\infty$ and can be expressed as 
 \begin{align*}
 	Z_{J }f(x)=\int_{\R^n} T_{k_1 }\int_{\R^n}T_{k_2 }\int_{\R^n}\cdots\int_{\R^n}K_J(k_1,\dots, k_J)T_{k_J}f_{k_J} (x)d\vec{k},
 \end{align*}
 in the sense of $L^p,$
where  
$\vec{k},$ $f_{k_J},$ and $K_J$  are defined by   \eqref{KJ}.
\subsection{The $L^p$-bound of $Z_J$}
After completing the proof of Theorem \ref{theorem Lp limit}, it remains to study the $L^p$-bound for operator $Z_J$ for all $1\leq p\leq \infty,$ which is given by the the following Theorem \ref{theorem bound}. 
\begin{theorem}\label{theorem bound}
Let   $Z_J$ be defined by  \eqref{Z8}, then $Z_J$ can be extended to a bounded operator on $L^p(\R^n)$ for all $1 \leq p \leq \infty$ and $J\in\mathbb{N}^+$.  In particular, fixed $0<\sigma\ll 1,$
 	\begin{align*}
 	\left\|Z_J\right\|_{\mathbb{B}(L^p)}\lesssim_{n, \sigma, J}
 	\begin{cases}
 		\prod_{j=1}^J \ \left\|\langle \cdot\rangle^{\frac{9-n}{2}+\sigma+} V_j\right\|_{L^2} & {\rm if}\ \ n=5,6,\\ 
 		\prod_{j=1}^J\left\| \F\left(\langle\cdot\rangle^{\frac{2(n-4+\sigma)}{n-1}+2\sigma+} V_j\right)\right\|_{L^{\frac{n-1}{n-4+\sigma}}} & {\rm if}\ \ n\geq7.
 	\end{cases}
 		\end{align*}  
\end{theorem}
Next, we provide {\bf the proof of Theorem \ref{theorem bound}} by dividing into the following  two steps:
\vskip 0.1cm
 {\bf Step (i).} \underline{Establish Lemma \ref{lemma Z} }, which reduces the study of the $L^p$-bound of the operator $Z_J$ to the $L^1((S^{n-1} \times \mathbb{R} \times \mathbb{R}^n)^J)$-bound of the function $F$ defined by \eqref{F}.
 \begin{lemma}\label{lemma Z}
 	Let $(s_j, \omega_j, t_j, y_j)\in \R^+\times S^{n-1}\times \R \times\R^n$ for $j=1,\cdots,J (J\in \N^+),$ and denote $\bar{x}:=x-2(x\cdot\omega_J).$   Then 
 	\begin{align*}
 			Z_{J }f(x)=\int_{\Sigma_{J-1}}\int_{S^{n-1}}\int_{-\infty}^{\rho_J}\int_{\R^n}F(\omega_1, t_1, y_1,\cdots,\omega_J, t_J, y_J)f(\bar{x}-\gamma_J)dy_Jdt_Jd\omega_J\cdots dy_1dt_1d\omega_1,
 		\end{align*}
 where $\Sigma_m:=\big(S^{n-1}\times(0, +\infty)\times\R^n\big)^{m}, \forall m\in \N^+,$ 
 $\rho_J:=-\omega_J\big( x-y_J-\sum_{\ell=1}^{J-1}(y_\ell-2t_\ell\omega_\ell)
 \big),$ $\gamma_J:=\bar{y}_J+2t_J\omega_J+\sum_{\ell=1}^{J-1}(\overline{y_\ell-2t_\ell\omega_\ell}),$  and 
  $F(\omega_1, t_1, y_1,\cdots,\omega_J, t_J, y_J)$ is given by 
\begin{align}\label{F}
	F:=\int_{(0,  +\infty)^J}\bigg(\prod_{j=1}^J e^{-is_jt_j}s_j^{n-2}h_{s_j\omega_j}(y_j)\bigg)K_J(s_1\omega_1,\cdots, s_J\omega_J)ds_J\cdots s_1.
\end{align}
 In particular,  $\left\|Z_{J }f\right\|_{L^p}
 \lesssim\left\|F\right\|_{L^1\big(\left(S^{n-1}\times \R\times\R^n\right)^{J}\big)}\left\|f\right\|_{L^p}$ for  all $1\leq p\leq\infty.$
 	\end{lemma}
 	\begin{proof}
 		 Observing that $G_0f (x)$ and $Z_{J }f(x)$ are given by \eqref{G0} and \eqref{Z8}, respectively,   $Z_J$ is essentially the $J$-fold composition of the operator $G_0.$ 
 		
 	Let   $\Theta_i=(\omega_j, t_j, y_j)$ and $k_j=s_j\omega_j$	for  $j=1,\cdots,J. $  We firstly  use  \eqref{G0} to rewrite  the most inner integral in \eqref{Z8} to derive that $	\big(G_0\left(K_Jf_{k_J}\right)\big)(x, k_1, \cdots, k_{J-1})$ can be written as
 		\begin{align*}
 			\int_{\Sigma_1}\left(\int_0^{+\infty}e^{is_J\big(t_J+\omega_J(x-y_J)\big)}s_J^{n-2}h_{s_J\omega_J}(y_J)K_J(s_1\omega_1,\cdots,  s_J\omega_J)ds_J\right)f(x-y_J+2t_J\omega_J)d\Theta_J.
 		\end{align*}
 	By  changing the variable $t_J+\omega_J(x-y_J)\mapsto -t_J,$ we rewrite
 		$	\big(G_0\left(K_Jf_{k_J}\right)\big)(x, k_1, \cdots, k_{J-1})$  as
 			\begin{align*}
 			\int_{S^{n-1}}\int_{-\infty}^{\rho_1}\int_{\R^n}\left(\int_0^{+\infty}e^{-is_Jt_J}s_J^{n-2}h_{s_J\omega_J}(y_J)K_J(s_1\omega_1,\cdots, s_J\omega_J)ds_J\right)f(\bar{x}-\gamma_1)d\Theta_J,
 		\end{align*}
 	where $\rho_1=-\omega_J(x-y_J),  \gamma_1=\bar{y}_J+2t_J\omega_J,$	and  $\bar{x}:=x-2(x\cdot\omega_J)$ is the reflection of $x$ along the $\omega_J$ axis.  
 	
 	Set 		$\rho_j:=-\omega_J\big( x-y_J-\sum_{\ell=J-j+1}^{J-1}(y_\ell-2t_\ell\omega_\ell)
 \big)$ and $\gamma_j:=\bar{y}_J+2t_J\omega_J+\sum_{\ell=J-j+1}^{J-1}(\overline{y_\ell-2t_\ell\omega_\ell}),$  for $j=2,3, \cdots, J.$ Next, continue to use \eqref{G0} to rewrite  the integral of $k_{J-1}$ in \eqref{Z8} to get that $\big(G_0 G_0\left(K_Jf_{k_J}\right)\big)(x, k_1, \cdots, k_{J-2})$ is equal to 
 		\begin{align*}
 		\int_{\Sigma_1}\int_{S^{n-1}}\int_{-\infty}^{\rho_2}\int_{\R^n} & \bigg(\int_{(0, +\infty)^2}\Big(\prod_{j=J-1}^J e^{-is_jt_j}s_j^{n-2}h_{s_j\omega_j}(y_j)\Big) K_J(s_1\omega_1,\cdots,  s_J\omega_J)ds_J s_{J-1}\bigg)\\
 	&\ \ \ \ \ \ \ \ \  \times 	f(\bar{x}-\gamma_2)d\Theta_Jd\Theta_{J-1}.
 		\end{align*}
 		
 		Go on this way  to rewrite the integral of $k_{J-j+1}$
 		for $j=1,\cdots, J$ in \eqref{Z8} to obtain  that $\big(G_0^j\left(K_Jf_{k_J}\right)\big)(x, k_1, \cdots, k_{J-j})$ is expressed as 
 			\begin{align*}
 			\int_{\Sigma_{j-1}}\int_{S^{n-1}}\int_{-\infty}^{\rho_j}\int_{\R^n} & \bigg(\int_{(0, +\infty)^j}\Big(\prod_{j=J-j+1}^J e^{-is_jt_j}s_j^{n-2}h_{s_j\omega_j}(y_j)\Big)K_J(s_1\omega_1,\cdots,  s_J\omega_J)ds_J\cdots ds_{J-j+1}\bigg)\\
 			&\ \ \ \ \ \ \ \ \  \times 	f(\bar{x}-\gamma_j)d\Theta_J\cdots d\Theta_{J-j+1},
 		\end{align*}
 		where	$G_0^j= G_0 \circ G_0 \circ \cdots \circ G_0,  (j \text{ times}).$	Hence, we  conclude that 
 			\begin{align*}
 			Z_{J }f(x)=\int_{\Sigma_{J-1}}\int_{S^{n-1}}\int_{-\infty}^{\rho_J}\int_{\R^n}F(\omega_1, t_1, y_1,\cdots,\omega_J, t_J, y_J)f(\bar{x}-\gamma_J)dy_Jdt_Jd\omega_J\cdots dy_1dt_1d\omega_1,
 		\end{align*}
 		where  $F(\omega_1, t_1, y_1,\cdots,\omega_J, t_J, y_J)$ is given by \eqref{F}.
 		 Furthermore, it's easy to find that 
 		 $$\big| 	Z_{J }f\big|\leq\int_{\left(S^{n-1}\times\R\times\R^n\right)^{J}}\big|F(\omega_1, t_1, y_1,\cdots,\omega_J, t_J, y_J)\big|\big|f(\bar{x}-\gamma_J)\big|d\Theta_J\cdots\varTheta_1.$$
 		 Applying  Minkowski's inequality and considering that $x\mapsto\bar{x}$ is an isometry,   we finally derive that 
 		  $\left\|Z_{J }f\right\|_{L^p}
 		 \lesssim\left\|F\right\|_{L^1\big(\left(S^{n-1}\times \R\times\R^n\right)^{J}\big)}\left\|f\right\|_{L^p}$ for  all $1\leq p\leq\infty.$
 		 
 		   Thus we complete the proof.
 	\end{proof}
    	\begin{remark}
 	In Lemma \ref{lemma Z} with $J=1,$ 	it is not hard to check that $\bar{x}:=x-2(x\cdot\omega_1),$  
 	 $\rho_1:=-\omega_1( x-y_1),$ $\gamma_1:=\bar{y}_1+2t_1\omega_1,$ and 
 	$$
 		Z_{1 }f(x)=\int_{S^{n-1}}\int_{-\infty}^{\rho_1}\int_{\R^n}F(\omega_1, t_1, y_1)f(\bar{x}-\gamma_1) dy_1dt_1d\omega_1. $$
  	\end{remark}
 	 {\bf Step (ii).} \underline{Compelte the proof of Therorem \ref{theorem bound}.}
 	\vskip0.1cm
     By Lemma \ref{lemma Z},  it suffices to show that the bound for $\|F\|_{L^1((S^{n-1} \times \mathbb{R} \times \mathbb{R}^n)^J)}$ coincides with the bound for $\|Z_J\|_{\mathbb{B}(L^p)}$ as stated in Therorem \ref{theorem bound}.
 	\begin{lemma}\label{lemma F}
 		Let $F(\omega_1, t_1, y_1,\cdots,\omega_J, t_J, y_J)$ be  defined by \eqref{F} and  fix $0<\sigma\ll 1.$ Then
 		\begin{align*}
 			\left\|F\right\|_{L^1\big(\left(S^{n-1}\times \R\times\R^n\right)^{J}\big)}\lesssim_{n, \sigma, J}
 			\begin{cases}
 				\prod_{j=1}^J \ \left\|\langle \cdot\rangle^{\frac{9-n}{2}+\sigma+} V_j\right\|_{L^2} & {\rm if}\ \ n=5,6,\\ 
 				\prod_{j=1}^J\left\| \F\left(\langle\cdot\rangle^{\frac{2(n-4+\sigma)}{n-1}+2\sigma+} V_j\right)\right\|_{L^{\frac{n-1}{n-4+\sigma}}} & {\rm if}\ \ n\geq7.
 			\end{cases}
 		\end{align*}
 		\end{lemma}
 		\begin{proof}
 			Insert the identity
 			$$1=\prod_{j=1}^{J}\Big(\chi(|y_j|)+\widetilde{\chi}(|y_j|)\Big)=\sum_{\Lambda\subset E}\Big[\Big(\prod_{j\in\Lambda}\chi(|y_j|)\Big)\Big(\prod_{j\notin\Lambda}\
 			\widetilde{\chi}(|y_j|)\Big)\Big],$$
 			into \eqref{F}, where $E=\{1,2, \cdots, J\}.$ Then $F(\omega_1, t_1, y_1,\cdots,\omega_J, t_J, y_J)$ can be expressed as 
 			$$\sum_{\Lambda\subset E}F(\omega_1, t_1, y_1,\cdots,\omega_J, t_J, y_J)\Big(\prod_{j\in\Lambda}\chi(|y_j|)\Big)\Big(\prod_{j\notin\Lambda}\
 			\widetilde{\chi}(|y_j|)\Big):=\sum_{\Lambda\subset E}F_\Lambda(\omega_1, t_1, y_1,\cdots,\omega_J, t_J, y_J).$$
 		Thus it suffices to claim that every $\left\|F_\Lambda\right\|_{L^1\big(\left(S^{n-1}\times \R\times\R^n\right)^{J}\big)}$ satisfies  the desired bound. 
 			
 			Fixed $r\geq2$,  note that $\left\|F_\Lambda\right\|_{L^1\big(\left(S^{n-1}\times \R^n\right)^{J}, \ L_{\vec{t}}^r(\R^J)\big)}$ can be expressed as
 			\begin{align*}
 				\int_{\left(S^{n-1}\times \R^n\right)^{J}}\bigg[\int_{\R^J} &
 				\bigg|\int_{\R^J}\bigg(\prod_{j=1}^J e^{-is_jt_j}\chi_{(0, +\infty)}(s_j)s_j^{n-2}h_{s_j\omega_j}(y_j)\bigg)K_J(s_1\omega_1,\cdots, s_J\omega_J)d\vec{s} \   
 				\bigg|^{r} d\vec{t} \ \bigg]^{\frac{1}{r}}\\
 				&\times\Big(\prod_{j\in\Lambda}\chi(|y_j|)\Big)\Big(\prod_{j\notin\Lambda}\
 				\widetilde{\chi}(|y_j|)\Big)d\vec{y}d\vec{\omega},
 				\end{align*}
 				where $\vec{t}=(t_1, \cdots, t_J), \ \vec{s}=(s_1, \cdots, s_J),$ and $ \vec{\omega}=(\omega_1, \cdots, \omega_J).$ Then utilize Hausdorff-Young inequality with $\frac{1}{r'}+\frac{1}{r}=1$ to obtain that $\left\|F_\Lambda\right\|_{L^1\big(\left(S^{n-1}\times \R^n\right)^{J},\  L_{\vec{t}}^r(\R^J)\big)}$  is dominated by
 					\begin{align*}
 				\lesssim\int_{\left(S^{n-1}\times \R^n\right)^{J}}\bigg[\int_{(0, +\infty)^J} &
 					\bigg| \Big(\prod_{j=1}^J s_j^{n-2}h_{s_j\omega_j}(y_j)\Big)K_J\bigg|^{r'}d\vec{s} \  \bigg]^{\frac{1}{r'}}\Big(\prod_{j\in\Lambda}\chi(|y_j|)\Big)\Big(\prod_{j\notin\Lambda}\
 					\widetilde{\chi}(|y_j|)\Big)d\vec{y}d\vec{\omega}.
 				\end{align*}
 				Take the identity $1=\chi(|y_j|)+\widetilde{\chi}(|y_j|)$ into the bound \eqref{estimate h},  fixed $0<\sigma\ll 1,$ then
 				\begin{align}\label{estimate h10}
 				\big|\partial_{s_j}^\ell h_{s_j\omega_j}(y_j)\big|\lesssim s_j^{-2-\ell}\left(s_j^\sigma\frac{1}{|y_j|^{n-\sigma}}\chi(|y_j|)+s_j^{-\sigma}\frac{1}{|y_j|^{n+\sigma}}\widetilde{\chi}(|y_j|)\right),
 				\end{align}
 				for  $\ell=0,1,2, $ and $j=1, \cdots, J.$
 				Hence, it can be checked that 
 					\begin{align*}
 			\left\|F_\Lambda\right\|_{L^1\big(\left(S^{n-1}\times \R^n\right)^{J},\  L_{\vec{t}}^r(\R^J)\big)}\lesssim&\int_{\left(S^{n-1} \times \R^n\right)^{J}} \bigg[\int_{(0, +\infty)^J} 
 					\bigg| \Big(\prod_{j\in\Lambda}s_j^{n-4+\sigma}\Big)\Big(\prod_{j\notin\Lambda}s_j^{n-4-\sigma}\Big)K_J\bigg|^{r'}d\vec{s} \  \bigg]^{\frac{1}{r'}}\\
 					&\times
 					\Big(\prod_{j\in\Lambda}\frac{1}{|y_j|^{n-\sigma}}\chi(|y_j|)\Big)\Big(\prod_{j\notin\Lambda}
 				\frac{1}{|y_j|^{n+\sigma}}	\widetilde{\chi}(|y_j|)\Big)d\vec{y}d\vec{\omega}.
 				\end{align*}
 				Next, perform the integrals for $y_J, y_{J-1}, \cdots, y_1$  and use  H\"older's  inequality to derive
 					\begin{align*}
 					\left\|F_\Lambda\right\|_{L^1\big(\left(S^{n-1}\times \R^n\right)^{J},\  L_{\vec{t}}^r(\R^J)\big)} \lesssim	\bigg(\int_{\left(S^{n-1}\right)^{J}}\int_{(0, +\infty)^J} 
 					\bigg| \Big(\prod_{j\in\Lambda}s_j^{n-4+\sigma}\Big)\Big(\prod_{j\notin\Lambda}s_j^{n-4-\sigma}\Big)K_J\bigg|^{r'}d\vec{s} 
 					d\vec{\omega}\bigg)^{\frac{1}{r'}},
 				\end{align*}
 				which yields that $\left\|F_\Lambda\right\|_{L^1\big(\left(S^{n-1}\times \R^n\right)^{J}, \ L_{\vec{t}}^r(\R^J)\big)}$ is controlled by 
 					\begin{align}\label{Ft0}
 					 \lesssim	\bigg(\int_{\R^{nJ}} 
 					\bigg| \Big(\prod_{j\in\Lambda}|k_j|^{n-4+\sigma+\frac{1-n}{r'}}\Big)\Big(\prod_{j\notin\Lambda}|k_j|^{n-4-\sigma+\frac{1-n}{r'}}\Big)K_J\bigg|^{r'}d\vec{k}\bigg)^{\frac{1}{r'}},
 				\end{align}
 				where $\vec{k}=(k_1, \cdots, k_J)\in\R^{nJ}.$ 
 				
 			Fix $\alpha_j\in\{0,1\}$ for $j=1, \cdots, J.$	Similarly, by Hausdorff-Young inequality, we derive that $	\left\|t_1^{\alpha_1}\cdots t_J^{\alpha_J} F_\Lambda\right\|_{L^1\big(\left(S^{n-1}\times \R^n\right)^{J},\  L_{\vec{t}}^r(\R^J)\big)}$  satisfies the bound
 				\begin{align}\label{tF}
 			\lesssim	\int_{\left(S^{n-1}\times \R^n\right)^{J}}  \bigg[\int_{(0, +\infty)^J} &
 					\bigg| \partial_{s_1}^{\alpha_1}\cdots\partial_{s_J}^{\alpha_J}  \bigg(\Big(\prod_{j=1}^J s_j^{n-2}h_{s_j\omega_j}(y_j)\Big)K_J\bigg)\bigg|^{r'}d\vec{s} \ \bigg]^{\frac{1}{r'}}\nonumber\\
 				& \ \ \ \ \ \ \ \  \ \ \ \  \ \ \ \  	\times\Big(\prod_{j\in\Lambda}\chi(|y_j|)\Big)\Big(\prod_{j\notin\Lambda}\
 					\widetilde{\chi}(|y_j|)\Big)d\vec{y}d\vec{\omega}.
 				\end{align}
 				Note that $\Big| \partial_{s_1}^{\alpha_1}\cdots\partial_{s_J}^{\alpha_J}\Big(\big(\prod_{j=1}^J s_j^{n-2}h_{s_j\omega_j}(y_j)\big)K_J\Big)\Big|$ is dominated by
 					\begin{align}\label{E1}
 				\lesssim\sum_{\tilde{\Lambda}\subset E}\ 
 					\bigg|\bigg[\prod_{\ell\in\tilde{\Lambda}}\partial_{s_\ell}^{\alpha_{\ell}}\Big(\prod_{j=1}^J s_j^{n-2}h_{s_j\omega_j}(y_j)\Big)\bigg]\bigg(\Big(\prod_{\ell\notin\tilde{\Lambda}}\partial_{s_\ell}^{\alpha_{\ell}}\Big)K_J\bigg)\bigg|.
 				\end{align}
 				Combining with \eqref{estimate h10}, it follows that
 					\begin{align}\label{E2}
 					&\bigg|\bigg[\prod_{\ell\in\tilde{\Lambda}}\partial_{s_\ell}^{\alpha_{\ell}}\Big(\prod_{j=1}^J s_j^{n-2}h_{s_j\omega_j}(y_j)\Big)\bigg]\bigg(\Big(\prod_{\ell\notin\tilde{\Lambda}}\partial_{s_\ell}^{\alpha_{\ell}}\Big)K_J\bigg)\bigg|\Big(\prod_{j\in\Lambda}\chi(|y_j|)\Big)\Big(\prod_{j\notin\Lambda}\
 					\widetilde{\chi}(|y_j|)\Big)\nonumber\\
 					&\ \ \ \ \ \ \ \ \ \ \  \ \  \lesssim\bigg(\prod_{j\in\Lambda}
 					s_j^{n-4+\sigma}	\frac{\chi(|y_j|)}{|y_j|^{n-\sigma}}	\bigg)\bigg(\prod_{j\notin\Lambda}
 				s_j^{n-4-\sigma}	\frac{\widetilde{\chi}(|y_j|)}{|y_j|^{n+\sigma}}	\bigg)\Big(\prod_{\ell\in\tilde{\Lambda}}s_\ell^{-\alpha_\ell}\Big)\bigg|\Big(\prod_{\ell\notin\tilde{\Lambda}}\partial_{s_\ell}^{\alpha_{\ell}}\Big)K_J\bigg|.
 				\end{align}
 				Consider \eqref{E1} and \eqref{E2} into \eqref{tF},  then 
 we firstly take the integrals for $\vec{y}$ and later utilize H\"older's  inequality to get $	\left\|t_1^{\alpha_1}\cdots t_J^{\alpha_J} F_\Lambda\right\|_{L^1\big(\left(S^{n-1}\times \R^n\right)^{J},\  L_{\vec{t}}^r(\R^J)\big)}$  is bounded by
 					\begin{align*}
 					\lesssim\sum_{\tilde{\Lambda}\in E}		\bigg(\int_{\R^{nJ}} 
 					\bigg| \Big(\prod_{j\in\Lambda}|k_j|^{n-4+\sigma+\frac{1-n}{r'}}\Big) \Big(\prod_{j\notin\Lambda}|k_j|^{n-4-\sigma+\frac{1-n}{r'}}\Big)    \Big(\prod_{\ell\in\tilde{\Lambda}}|k_\ell|^{-\alpha_\ell}\Big)\bigg(\Big(\prod_{\ell\notin\tilde{\Lambda}} \nabla_{k_\ell}^{\alpha_{\ell}}\Big)K_J\bigg)\bigg|^{r'}d\vec{k}\bigg)^{\frac{1}{r'}}.
 				\end{align*}
 	By  Hardy's inequality (see Lemma \ref{Appendix 5}), $\left\|t_1^{\alpha_1}\cdots t_J^{\alpha_J} F_\Lambda\right\|_{L^1\big(\left(S^{n-1}\times \R^n\right)^{J},\  L_{\vec{t}}^r(\R^J)\big)}$ is dominated by 
 	\begin{align}\label{Ft1}
	\lesssim\bigg(\int_{\R^{nJ}} 
 		\bigg| \Big(\prod_{j\in\Lambda}|k_j|^{n-4+\sigma+\frac{1-n}{r'}}\Big)\Big(\prod_{j\notin\Lambda}|k_j|^{n-4-\sigma+\frac{1-n}{r'}}\Big) \bigg(\Big(\prod_{\ell=1}^J \nabla_{k_\ell}^{\alpha_{\ell}}\Big)K_J\bigg)\bigg|^{r'}d\vec{k}\bigg)^{\frac{1}{r'}}.
 	\end{align}			
 Take into account  the bound above and 
 $\prod_{j=1}^J\big(1+|t_j|\big)=\sum_{\alpha_1, \dots, \alpha_J\in\{0,1 \}}\big|t_1^{\alpha_1}t_2^{\alpha_2}\cdots t_J^{\alpha_J}\big|,$	then 			
we get  that $	\big\|\big(\prod_{j=1}^J\langle t_j\rangle\big)  F_\Lambda\big\|_{L^1\big(\left(S^{n-1}\times \R^n\right)^{J},\  L_{\vec{t}}^r(\R^J)\big)}$ satisfies the bound \eqref{Ft1}.
 \vskip0.2cm	
 $\bullet$	\underline{When $n=5, 6$,}  we choose $r'=2 \ (r=2)$ to make $n-4\pm\sigma+\frac{1-n}{r'}=\frac{n-7}{2}\pm\sigma<0$.

 	Applying  Hardy's inequality (see Lemma \ref{Appendix 5}) to the integral for $k_J$ in \eqref{Ft0}, the multiplier theorem and  the definition $K_J$ given by \eqref{KJ},  we derive that 
 	\small
 	\begin{align*}
 		\int_{\R^n}	\Big||k_J|^{\frac{n-7}{2}\pm\sigma}	\widehat{V}_J(k_J-k_{J-1})\Big|^2dk_J&\lesssim	\int_{\R^n}	\Big|\nabla_{k_J}^{\frac{7-n}{2}\mp\sigma}	\widehat{V}_J(k_J-k_{J-1})\Big|^2dk_J\\
 		&\lesssim\left\| \F\left(\langle \cdot \rangle^{\frac{7-n}{2}+\sigma}V_J(\cdot)\right) \right\|_{L^2}^2
 		\lesssim\left\| \langle \cdot \rangle^{\frac{7-n}{2}+\sigma}V_J(\cdot) \right\|_{L^2}^2,
 	\end{align*}
 	here the last inequality is followed from the Plancherel's theorem.			
 Go on this way to deal with the integrals for $k_{J-1}, \cdots,  k_1,$ in \eqref{Ft0}, then	
 \begin{align}	\label{Ft00}		
 		\left\|F_\Lambda\right\|_{L^1\big(\left(S^{n-1}\times \R^n\right)^{J},\  L_{\vec{t}}^2(\R^J)\big)}\lesssim\prod_{j=1}^J\left\| \langle \cdot \rangle^{\frac{7-n}{2}+\sigma}V_j(\cdot) \right\|_{L^2}.
  	\end{align}
  	Furthermore,  utilize	 Hardy's inequality  (see Lemma \ref{Appendix 5}) to the integral for $k_J$ in \eqref{Ft1} with $r'=2,$
  	\small
  	\begin{align}\label{k_J}
  		&\int_{\R^n}	\Big||k_J|^{\frac{n-7}{2}\pm\sigma}	\nabla_{k_J}^{\alpha_J}\widehat{V}_J(k_J-k_{J-1})\Big|^2dk_J+ \int_{\R^n}	\Big||k_J|^{\frac{n-7}{2}\pm\sigma}	\nabla_{k_{J-1}}^{\alpha_{J-1}}\nabla_{k_J}^{\alpha_J}\widehat{V}_J(k_J-k_{J-1})\Big|^2dk_J\nonumber\\
  	&\lesssim\int_{\R^n}	\Big|\nabla_{k_J}^{\frac{7-n}{2}\mp\sigma}	\nabla_{k_J}^{\alpha_J}\widehat{V}_J(k_J-k_{J-1})\Big|^2dk_J+ \int_{\R^n}	\Big|\nabla_{k_J}^{\frac{7-n}{2}\mp\sigma}	\nabla_{k_{J-1}}^{\alpha_{J-1}}\nabla_{k_J}^{\alpha_J}\widehat{V}_J(k_J-k_{J-1})\Big|^2dk_J.
  	\end{align}					
  	By the multiplier theorem,  \eqref{k_J} satisfies the bound $\big\| \langle \cdot \rangle^{\frac{7-n}{2}+2+\sigma}V_J(\cdot) \big\|_{L^2}^2.$ Continuing this process to deal with the integrals for $k_{J-1}, \cdots,  k_1,$ in \eqref{Ft1}, then	
  	\begin{align}\label{Ft11}		
  			\Big\|\Big(\prod_{j=1}^J\langle t_j\rangle\Big)  F_\Lambda\Big\|_{L^1\big(\left(S^{n-1}\times \R^n\right)^{J},\  L_{\vec{t}}^2(\R^J)\big)}\lesssim\prod_{j=1}^J\left\| \langle \cdot \rangle^{\frac{7-n}{2}+2+\sigma}V_j(\cdot) \right\|_{L^2}.
  	\end{align}
  	Take into account \eqref{Ft00} and \eqref{Ft11}, then by multilinear complex interpolation,
  	\begin{align*}
  		\Big\|\Big(\prod_{j=1}^J
  		\langle t_j\rangle^{\frac{1}{2}+}\Big)  F_\Lambda\Big\|_{L^1\big(\left(S^{n-1}\times \R^n\right)^{J},\  L_{\vec{t}}^2(\R^J)\big)}\lesssim\prod_{j=1}^J\left\| \langle \cdot \rangle^{\frac{9-n}{2}+\sigma+}V_j(\cdot) \right\|_{L^2}.
  	\end{align*}
  	Apply Cauchy-Schwarz inequality get that	$\big\|F_\Lambda\big\|_{L^1\big(\left(S^{n-1}\times \R\times\R^n\right)^{J}\big)}$ is controlled by 
  \begin{align}\label{FT1}
  	\lesssim
  	\Big\| \prod_{j=1}^J
  	\langle t_j\rangle^{-\frac{1}{2}-}\Big\|_{L_{\vec{t}}^2(\R^J)}	\Big\|\Big(\prod_{j=1}^J
  	\langle t_j\rangle^{\frac{1}{2}+}\Big)  F_\Lambda\Big\|_{L^1\big(\left(S^{n-1}\times \R^n\right)^{J},\  L_{\vec{t}}^2(\R^J)\big)}
  \lesssim\prod_{j=1}^J\left\| \langle \cdot \rangle^{\frac{9-n}{2}+\sigma+}V_j(\cdot) \right\|_{L^2}.
  \end{align}
 \vskip0.2cm 
$\bullet$  	\underline{When $ n\geq7$, } we choose $r'=\frac{n-1}{n-4+\sigma} \ (r=\frac{n-1}{3-\sigma}>2)$ to make $n-4+\sigma+\frac{1-n}{r'}=0$ and $n-4-\sigma+\frac{1-n}{r'}=-2\sigma<0$. 

By the same way to get \eqref{Ft00},
 we apply  Hardy's inequality (see Lemma \ref{Appendix 5}) to the integrals for $k_J, \cdots,k_1$ in \eqref{Ft0} and  the multiplier theorem to get 
  \begin{align}	\label{Ft000}		
 	\left\|F_\Lambda\right\|_{L^1\big(\left(S^{n-1}\times \R^n\right)^{J},\  L_{\vec{t}}^r(\R^J)\big)}\lesssim\Bigg(\int_{\R^{nJ}} 
 	\bigg| \Big(\prod_{j\notin\Lambda}|k_j|^{-2\sigma}\Big)K_J\bigg|^{r'}d\vec{k}\Bigg)^{\frac{1}{r'}}\lesssim\prod_{j=1}^J\left\|\F\Big(\langle \cdot \rangle^{2\sigma}V_j(\cdot)\Big) \right\|_{L^{r'}}.
 \end{align}
 Similar to derive \eqref{Ft11},		
we apply  Hardy's inequality  (see Lemma \ref{Appendix 5}) to the integrals for $k_J, \cdots,k_1$ in \eqref{Ft1}  and the multiplier theorem to get that $	\big\|\Big(\prod_{j=1}^J\langle t_j\rangle\Big)  F_\Lambda\big\|_{L^1\big(\left(S^{n-1}\times \R^n\right)^{J},\  L_{\vec{t}}^r(\R^J)\big)}$ satisfies
 	\begin{align}\label{Ft111}
 &\lesssim\bigg(\int_{\R^{nJ}} 
 	\bigg| \Big(\prod_{j\notin\Lambda}|k_j|^{-2\sigma}\Big) \bigg(\Big(\prod_{\ell=1}^J \nabla_{k_\ell}^{\alpha_{\ell}}\Big)K_J\bigg)\bigg|^{r'}d\vec{k}\bigg)^{\frac{1}{r'}}
 	\lesssim\prod_{j=1}^J\left\|\F\Big(\langle \cdot \rangle^{2+2\sigma}V_j(\cdot)\Big) \right\|_{L^{r'}}.
 \end{align}			
Utilize  H\"older's  inequality  to the  integral with respect to $\vec{t}$, then it follows that  \begin{align}\label{FT2}
	\big\|F_\Lambda\big\|_{L^1\big(\left(S^{n-1}\times \R\times\R^n\right)^{J}\big)}&\lesssim
	\Big\| \prod_{j=1}^J
	\langle t_j\rangle^{-\frac{1}{r'}-}\Big\|_{L_{\vec{t}}^
		{r'}(\R^J)}	\Big\|\Big(\prod_{j=1}^J
	\langle t_j\rangle^{\frac{1}{r'}+}\Big)  F_\Lambda\Big\|_{L^1\big(\left(S^{n-1}\times \R^n\right)^{J},\  L_{\vec{t}}^r(\R^J)\big)}
	\nonumber\\
	&\lesssim\prod_{j=1}^J\Big\|\F\Big(\langle \cdot \rangle^{\frac{2(n-4+\sigma)}{n-1}+2\sigma+}V_j(\cdot)\Big) \Big\|_{L^{\frac{n-1}{n-4+\sigma}}},
\end{align}
 where the last inequality is established by  multilinear complex interpolation since  \eqref{Ft000} and \eqref{Ft111}.
 
Combining with the bounds \eqref{FT1} and \eqref{FT2}, while carefully tracking their dependence on $n, \sigma$ and $J,$  we obtain the desired result.
Thus  we complete the proof.
 		\end{proof}
 	By uniting Theorems \ref{theorem Lp limit} and \ref{theorem bound}, we complete the proof of Proposition \ref{proposition Z_J}, which, in turn, concludes the proof of Theorem \ref{theorem WJ}.
 	\section{$L^p$ boundedness of the low energy part $\Omega_L^\pm$}\label{section3}
 	In this section,  we aim to demonstrate that the low energy  part $\Omega_L^\pm$ defined by \eqref{WL} is bounded on $L^p$ for all $1\leq p\leq\infty$ when dimensions $n\geq5.$ 
 	
 		By  changing the variable $\lambda\mapsto \eta^4+\eta^2,$ the operator $\Omega^\pm$ defined by 
 \eqref{W^LH} is expressed  as 
 \small
 \begin{align}\label{WLH1}
 	\frac{1}{\pi i}\int_0^\infty\eta(2\eta^2+1)(R_0^\mp V)^kR_V^\mp
 	(VR_0^\mp)^kV\big(R_0^+-R_0^-\big)d\eta,
 	\end{align}
  where $R_0^\mp:=R_0^\mp(\eta^4+\eta^2)$ and $R_V^\mp:=R_V^\mp(\eta^4+\eta^2)$ for short.
 	Insert the identity $1=\chi(\eta)+\widetilde{\chi}(\eta)$ into \eqref{WLH1}, then $\Omega^\pm=\Omega_L^\pm+\Omega_H^\pm,$ where $\Omega_L^\pm$ denotes the low energy  part with $0< \eta\ll1$:
 	 \begin{align}\label{WL}
 	\Omega_L^\pm:=\frac{1}{\pi i}\int_0^\infty\eta(2\eta^2+1)\chi(\eta)(R_0^\mp V)^kR_V^\mp(VR_0^\mp)^kV\big(R_0^+-R_0^-\big)d\eta,
 	\end{align}
 	and  $\Omega_H^\pm$ denotes the high energy part  with $\{\eta\gtrsim1\}$:
 	 \begin{align*}
 		\Omega_H^\pm:=\frac{1}{\pi i}\int_0^\infty\eta(2\eta^2+1)\widetilde{\chi}(\eta)(R_0^\mp V)^kR_V^\mp(VR_0^\mp)^kV\big(R_0^+-R_0^-\big)d\eta.
 	\end{align*}
 	
 	Next, we focus on the $L^p$ boundedness for $\Omega_L^\pm.$ Since $\Omega_L^+f=\overline{ \Omega_L^-\bar{f}},$ it reduces to deal with $\Omega_L^-,$ which is given by the following Theorem \ref{theorem WL}.
 	\begin{theorem}\label{theorem WL}
 		Let  $n\geq5$ and $V(x)\lesssim\langle x \rangle^{-n-4-}$. Then provided $k$ enough large depending on
 		$n$, the operator $\Omega_L^-$ defined by \eqref{WL} 
 		extends to a bounded operator on $L^p(\R^n)$ for all $1\leq p\leq\infty.$
 	\end{theorem} 
 \subsection{The low energy analyses of $R_0^\pm$ and  $R_V^\pm$}
Before the proof of Theorem \ref{theorem WL}, we need to research  the boundary operator \texorpdfstring{$R_0^\pm(\eta^4+\eta^2)$}{$R_0^\pm$} and \texorpdfstring{$R_V^+(\eta^4+\eta^2)$}{$R_V^+$} for $0<\eta\ll1$.

 	According to  \eqref{boundary}-\eqref{the split} and  the  limiting absorption principle (see e.g. \cite{Agmon}),  it concludes that
 		\begin{align}\label{split}
 	R_0^{\pm}(\eta^4+\eta^2)
 			=\frac{1}{1+2\eta^2}\Big(R^{\pm}_\Delta\big(\eta^2\big)-R_{\Delta}\big(-1-\eta^2\big)\Big), \ \ \eta>0.
 		\end{align}
 		
 		Recall the kernel of 
 		the second order  free Schr\"odinger  resolvent $R_\Delta(z)(x,y)$(see {\it e.g.} \cite{Jensen-80, Jensen-Kato-79}	):
 			\begin{itemize}
 			\item If $n\geq5$  and $n$ is odd, then for $\Im(z)>0,$
 			\begin{align}\label{2-odd}
 				R_\Delta(z^2)(x,y)
 				=\frac{e^{iz|x-y|}}{|x-y|^{n-2}}\sum_{j=0}^
 				{\frac{n-3}{2}}c_j|x-y|^jz^j.
 			\end{align}
 			\item If $n\geq5$  and $n$ is even, then for $\Im(z)>0$ and $\big|z|x-y|\big|\ll1,$
 			 	\begin{align}\label{2-even<}
 			 	R_\Delta(z^2)(x,y)
 			 	=\frac{1}{|x-y|^{n-2}} \bigg(\sum_{j=0}^
 			 	{\frac{n}{2}-2}a_j\big(z|x-y|\big)^{2j}
 			 	+\sum_{j=\frac{n}{2}-1}^
 			 	{\infty}\Big(a_j+d_j\ln \big(z|x-y|\big) \Big)\big(z|x-y|\big)^{2j}\bigg).
 			 \end{align}
 			 	\item If $n\geq5$  and $n$ is even, then for $\Im(z)>0$ and $\big|z|x-y|\big|\gtrsim1,$
 			 \begin{align}\label{2-even>}
 			 R_\Delta(z^2)(x,y)
 			 =\frac{e^{iz|x-y|}}{|x-y|^{n-2}}\big(z|x-y|\big)^
 			 {\frac{n-2}{2}}w\big(z|x-y|\big).
 			 \end{align}
 		\end{itemize}
 		Here, the constants $a_j, c_j, d_j$ dependent on $n$ can be computed and $\big|\frac{d^\ell}{dz^\ell}w(z)\big|\lesssim\big(1+|z|\big)^{-\frac{1}{2}-\ell}.$
 		\vskip0.2cm
 		Observe that   $0<\arg\big(\sqrt{1+4z}\big)<\pi$ and $\pi<\arg\big(-\sqrt{1+4z}\big)<2\pi$ for  $z\in\C\setminus[0, +\infty)$ with $0< \arg z<2\pi,$ which gives that
 		$\Im\Big(-\frac{1}{2}+\frac{1}{2}\sqrt{1+4z}\Big)^{\frac{1}{2}}>0$ and 	$\Im\Big(-\frac{1}{2}-\frac{1}{2}\sqrt{1+4z}\Big)^{\frac{1}{2}}>0.$ 
 			Hence combine with  the splitting identity \eqref{split} and  the kernel $R_\Delta(z)(x,y)$ \eqref{2-odd}-\eqref{2-even>}, then  by   the  limiting absorption principle and denoting $r=|x-y|$, it concludes that for $\eta>0,$ we have
 			\begin{itemize}
 			\item If $n\geq5$  and $n$ is odd, then 
 			\begin{align}\label{4-odd}
 				R_0^\pm(\eta^4+\eta^2)(x,y)
 				=\frac{1}{(1+2\eta^2)r^{n-2}}\bigg[e^{\pm i\eta r}\sum_{j=0}^
 				{\frac{n-3}{2}}c_j\big(\pm\eta r\big)^j-e^{ -r\sqrt{1+\eta^2}}\sum_{j=0}^
 				{\frac{n-3}{2}}c_j\big(ir\sqrt{1+\eta^2}\big)^j\bigg].
 			\end{align}
 			\item If $n\geq5$  and $n$ is even, then for  $\eta r\ll1$ and $r\ll1,$
 			\begin{align}\label{4-even<}
 				&R_0^\pm(\eta^4+\eta^2)(x,y)
 				=\frac{1}{(1+2\eta^2)r^{n-2}}\bigg[ \bigg(\sum_{j=0}^
 				{\frac{n}{2}-2}a_j\big(\pm\eta r\big)^{2j}
 				+\sum_{j=\frac{n}{2}-1}^
 				{\infty}\Big(a_j+d_j\ln\big (\pm\eta r\big) \Big)\big(\pm\eta r\big)^{2j}\bigg)\nonumber\\
 			&	\ \ \ \ \ -\bigg(\sum_{j=0}^
 			{\frac{n}{2}-2}a_j\big(ir\sqrt{1+\eta^2}\big)^{2j}
 			+\sum_{j=\frac{n}{2}-1}^
 			{\infty}\big(ir\sqrt{1+\eta^2}\big)^{2j}\Big(a_j+\frac{ \pi i}{2}d_j+d_j\ln \big( r\sqrt{1+\eta^2}\big) \Big)\bigg)\bigg].
 			\end{align}
 			\item If $n\geq5$  and $n$ is even, then for  $\eta r\ll1$ and $r\gtrsim1,$
 			\begin{align}\label{4-even<>}
 				R_0^\pm(\eta^4+\eta^2)(x,y)
 				=\frac{1}{(1+2\eta^2)r^{n-2}}\bigg[ &\bigg(\sum_{j=0}^
 				{\frac{n}{2}-2}a_j\big(\pm\eta r\big)^{2j}
 				+\sum_{j=\frac{n}{2}-1}^
 				{\infty}\Big(a_j+d_j\ln\big (\pm\eta r\big) \Big)\big(\pm\eta r\big)^{2j}\bigg)\nonumber\\
 				&	\ \ \ \ \ -e^{ -r\sqrt{1+\eta^2}}\big(ir\sqrt{1+\eta^2}\big)^{\frac{n-2}{2}}w\big(ir\sqrt{1+\eta^2}\big)\bigg].
 			\end{align}
 			\item If $n\geq5$  and $n$ is even, then for  $\eta r\gtrsim1,$ 
 				\begin{align}\label{4-even>}
 				R_0^\pm(\eta^4+\eta^2)(x,y)
 				=\frac{1}{(1+2\eta^2)r^{n-2}}\bigg(&e^{\pm i\eta r}\big(\pm\eta r\big)^{\frac{n-2}{2}}w_{\pm}\big(\eta r\big)\nonumber\\
 				&-e^{ -r\sqrt{1+\eta^2}}\big(ir\sqrt{1+\eta^2}\big)^{\frac{n-2}{2}}w\big(ir\sqrt{1+\eta^2}\big)\bigg).
 			\end{align}
 		Here, $w_{\pm}(z)\in C^\infty(\R)$ such that  $\big|\frac{d^\ell}{dz^\ell}w_{\pm}(z)\big|\lesssim\big(1+|z|\big)^{-\frac{1}{2}-\ell}.$
 	\end{itemize}
 	Next, we will use  \eqref{4-odd}-\eqref{4-even>} to establish some estimates about  the kernel $	R_0^\pm(\eta^4+\eta^2)(x,y).$
 	Denote
 			\begin{align}\label{R0F}
 			R_0^+(\eta^4+\eta^2)(x,y)
 			:=\frac{e^{i\eta|x-y|}}{(1+2\eta^2)r^{n-2}}F(\eta, |x-y|).
 		\end{align}
 		Then we  discuss the function $F(\eta, |x-y|)$ for $0<\eta\ll1$, classified  into two cases i.e. $n$ is odd and $n$ is even, which is given by the following  two Lemmas \ref{lemma F odd}-\ref{lemma F even}, respectively.
 		\begin{lemma}\label{lemma F odd}
 			Let $n\geq5$ be odd and $r:=|x-y|$. Then for $0<\eta\ll1,$
 				\begin{align}\label{F1 odd}
 			\Big|\partial_\eta^{\ell}F(\eta, r)\Big|\lesssim
 				 r^\ell\big\langle\eta r \big\rangle^{\frac{n-3}{2}-\ell}+
 				\big\langle r \big\rangle^{\frac{n-3}{2}+\ell}e^{-r},
 				\ \ \   \ \ell=0,1,2,\cdots.
 			\end{align}
 			In particular,  for $0<\eta\ll1,$
 				\begin{align}\label{F2 odd}
 				\Big|\partial_\eta^{\ell}F(\eta, r)\Big|\lesssim \eta^{-\ell}\big\langle\eta r \big\rangle^{\frac{n-3}{2}}, \ \ \ \ell=0,1,2,\cdots.
 			\end{align}
 		\end{lemma}
 		\begin{proof}
 			Take into account \eqref{R0F} and \eqref{4-odd} to obtain that
 				\begin{align}\label{F3 odd}
 				F(\eta, r)
 				=\sum_{j=0}^
 				{\frac{n-3}{2}}c_j\big(\eta r\big)^j-e^{ \big(-\sqrt{1+\eta^2}-i\eta\big)r}\sum_{j=0}^
 				{\frac{n-3}{2}}c_j\big(ir\sqrt{1+\eta^2}\big)^j:=F_1(\eta r)-F_2(\eta, r).
 			\end{align}
 			Since $F_1(\eta r)$ is a polynomial of degree $\frac{n-3}{2}$ about $\eta r$, it follows that 
 			\begin{align*}
 			 	\Big|\partial_\eta^{\ell}F_1(\eta  r)\Big|\lesssim r^\ell\big\langle\eta r \big\rangle^{\frac{n-3}{2}-\ell}, \ \ \ \ell=0,1,2,\cdots.
 			 \end{align*}
 			 Since $0<\eta\ll1,$ then $\sqrt{1+\eta^2}\approx 1,$ and  $\big|\frac{d^\ell}{d\eta^\ell}\big(\sqrt{1+\eta^2}\big)\big|\lesssim1$.   Moreover,
 			 	\begin{align*}
 			 	\Big|\partial_\eta^{\ell}F_2(\eta,   r)\Big|\lesssim \big\langle r \big\rangle^{\frac{n-3}{2}+\ell}e^{-r\sqrt{1+\eta^2}}\lesssim \big\langle r \big\rangle^{\frac{n-3}{2}+\ell}e^{-r},  \ \ \ell=0,1,2,\cdots.
 			 \end{align*}
 			 
 			 It is obvious that 
 			 \begin{align}\label{F10 odd}
 			 	r^\ell\big\langle\eta r \big\rangle^{\frac{n-3}{2}-\ell}=\eta^{-\ell}(\eta r)^{\ell}\big\langle\eta r \big\rangle^{\frac{n-3}{2}-\ell}\leq\eta^{-\ell}\big\langle\eta r \big\rangle^{\frac{n-3}{2}}, \ \ \ \ell=0,1,2,\cdots.
 			 	\end{align}
 			 And consider $0<\eta\ll1$ to  derive that  
 			 $$ \big\langle r \big\rangle^{\frac{n-3}{2}+\ell}e^{-r}\lesssim1\lesssim\eta^{-\ell}\lesssim\eta^{-\ell}\big\langle\eta r \big\rangle^{\frac{n-3}{2}},\ \ \ \ \ \ell=0,1,2,\cdots.$$
 			  			Thus we complete the proof. 
 		\end{proof}	
 		\begin{lemma}\label{lemma F even}
 			Let $n\geq6$ be even and $r:=|x-y|$. Then for $0<\eta\ll1$ and $\eta r\gtrsim1,$
 			\begin{align}\label{F1 even}
 				\Big|\partial_\eta^{\ell}F(\eta, r)\Big|\lesssim r^\ell\big\langle\eta r \big\rangle^{\frac{n-3}{2}-\ell}, \ \ \ \ell=0,1,2,\cdots.
 			\end{align}
 			Moreover, for $0<\eta\ll1$ and $\eta r\ll1,$
 			\begin{align}\label{F2 even}
 			\Big|\partial_\eta^{\ell}F(\eta, r)\Big|\lesssim
 				\begin{cases}
 					r^\ell & {\rm if}\ \ 0\leq \ell\leq3,\\ 
 					r^4\big|\ln(\eta r)\big| & {\rm if}\ \ \ell=4,\\
 					r^\ell(\eta r)^{-(\ell-4)}& {\rm if}\ \ \ell>4.
 				\end{cases}
 			\end{align}
 			In particular,  for $0<\eta\ll1,$
 			\begin{align}\label{F3 even}
 				\Big|\partial_\eta^{\ell}F(\eta, r)\Big|\lesssim \eta^{-\ell}\big\langle\eta r \big\rangle^{\frac{n-3}{2}}, \ \ \ \ell=0,1,2,\cdots.
 			\end{align}
 		\end{lemma}
 		\begin{proof}
 			Let 
 			\begin{align}\label{F4 even}
 			F_1(\eta r):=
 			\begin{cases}
 				e^{-i\eta r}\bigg(\sum_{j=0}^
 				{\frac{n}{2}-2}a_j\big(\eta r\big)^{2j}
 				+\sum_{j=\frac{n}{2}-1}^
 				{\infty}\Big(a_j+d_j\ln\big (\eta r\big) \Big)\big(\eta r\big)^{2j}\bigg),& {\rm if}\ \ \eta r\ll 1,\\ 
 				\big(\eta r\big)^{\frac{n-2}{2}}w_{+}\big(\eta r\big),& {\rm if}\ \ \eta r\gtrsim1.
 			\end{cases}
 		\end{align}
 			 When $\langle \eta \rangle r\ll 1,$ we denote
 				\begin{align}\label{F5 even}
 				F_2(\eta, r)=
 				e^{-i\eta r}	\bigg[\sum_{j=0}^
 					{\frac{n}{2}-2}a_j\big(ir\sqrt{1+\eta^2}\big)^{2j}
 					+\sum_{j=\frac{n}{2}-1}^
 					{\infty}\big(ir\sqrt{1+\eta^2}\big)^{2j}\Big(a_j+\frac{ \pi i}{2}d_j+d_j\ln \big( r\sqrt{1+\eta^2}\big) \Big)\bigg].
 			\end{align}
 			When  $\langle \eta \rangle r\gtrsim 1,$ we  denote
 			\begin{align}\label{F6 even}
 				F_2(\eta, r)=
 					e^{ \big(-\sqrt{1+\eta^2}-i\eta\big)r}\big(ir\sqrt{1+\eta^2}\big)^{\frac{n-2}{2}}w\big(ir\sqrt{1+\eta^2}\big).
 				\end{align}
 				
 			By \eqref{4-even<}-\eqref{4-even>} and \eqref{R0F},  we derive that $F(\eta, r)=F_1(\eta r)-F_2(\eta, r).$ Next, we  start to estimate $\partial_\eta^{\ell}F_1(\eta r)$ and $\partial_\eta^{\ell}F_2(\eta, r)$, respectively, for
 				$\ell\in\N^+\cup\{0\}.$ 
 				
 			$\bullet$ \underline{	Firstly, we deal with $\partial_\eta^{\ell}F_1(\eta r).$}
 			For $\eta r\ll 1$ and noting that $n\geq6$ such that $\frac{n}{2}-1\geq2,$  the parenthesis in \eqref{F4 even} i.e. $e^{i\eta r}F_1(\eta r)$ has the following estimate:
 					\begin{align}\label{F7 even}
 					\Big|\partial_\eta^{\ell}\Big(e^{i\eta r}F_1(\eta r)\Big)\Big|\lesssim
 					\begin{cases}
 						r^\ell & {\rm if}\ \ 0\leq \ell\leq3,\\ 
 						r^4\big|\ln(\eta r)\big| & {\rm if}\ \ \ell=4,\\
 						r^\ell(\eta r)^{-(\ell-4)}& {\rm if}\ \ \ell>4.
 					\end{cases}
 				\end{align}
 		Then by Leibniz formula, it is 	easy to find that  $\Big|\partial_\eta^{\ell}F_1(\eta r)\Big|= \Big|\partial_\eta^{\ell}\Big[e^{-i\eta r}\Big(e^{i\eta r}F_1(\eta r)\Big)\Big]\Big|$ has the same results as above for $\eta r\ll 1$ and $\ell=0,1,2,\cdots.$
 		
 	For	$\eta r\gtrsim1$, consider that $\big|\partial_\eta^\ell w_+(\eta r)\big|\lesssim r^\ell(\eta r)^{-\frac{1}{2}-\ell}$  $(\ell=0,1,2,\cdots)$
 	to obtain that
 					\begin{align}\label{F8 even}
 					\Big|\partial_\eta^{\ell}F_1(\eta r)\Big|\lesssim
 						r^\ell\big\langle\eta r \big\rangle^{\frac{n-3}{2}-\ell}, \ \  \ell=0,1,2,\cdots.
 				\end{align}
 				
$\bullet$ \underline{Next, we focus on  $\partial_\eta^{\ell}F_2(\eta, r).$} 
 	For $\langle\eta \rangle r\ll 1$ namely $r\ll 1$ ( in the case, it's obvious that $\eta r\ll 1$ by $0<\eta\ll1$),  considering that $n\geq6$ such that $\frac{n}{2}-1\geq2,$  the bracket in \eqref{F5 even} i.e. $e^{i\eta r}F_2(\eta, r)$ satisfies that
 \begin{align}\label{F9 even}
 	\Big|\partial_\eta^{\ell}\Big(e^{i\eta r}F_2(\eta, r)\Big)\Big|\lesssim
 	\begin{cases}
 		r^\ell & {\rm if}\ \ 0\leq \ell\leq3,\\ 
 		r^4\big|\ln(\langle\eta \rangle r)\big| & {\rm if}\ \ \ell=4,\\
 		r^4& {\rm if}\ \ \ell>4,
 	\end{cases}
 \end{align}
 		which is dominated by \eqref{F7 even} since $r\ll 1$ and $\eta r\ll 1$ with $0<\eta\ll1$.
 		Note that  $\big|\partial_\eta^{\ell}F_2(\eta, r)\big|$ has the same results as above for $\langle \eta \rangle r\ll 1$ and $\ell=0,1,2,\cdots.$
 			
 				For $\langle\eta \rangle r\gtrsim 1$ with  $0<\eta\ll1$, namely $r\gtrsim1$ but $\eta r\ll1$ or $r\gtrsim1$ and $\eta r\gtrsim1$. In virtue of \eqref{F6 even}, it follows that
 			\begin{align}\label{F10 even}
 				\Big|\partial_\eta^{\ell}F_2(\eta, r)\Big|\lesssim \big\langle r \big\rangle^{\frac{n-3}{2}}e^{-r}\lesssim r^\ell\big\langle\eta r \big\rangle^{\frac{n-3}{2}-\ell}, \ \  \ell=0,1,2,\cdots.
 			\end{align}
 		Observe  that  the bound above (i.e. $r^\ell\big\langle\eta r \big\rangle^{\frac{n-3}{2}-\ell}$) is controlled by \eqref{F7 even} when $r\gtrsim1$ but $\eta r\ll1$.
 			 
 			$\bullet$ \underline{Return to $F(\eta, r).$}
 				For $\eta r\ll 1$ and $0<\eta\ll1$,  since $F(\eta, r)=F_1(\eta r)-F_2(\eta, r),$ then   combining with \eqref{F7 even}, \eqref{F9 even} and \eqref{F10 even}, $\big|\partial_\eta^{\ell}F(\eta, r)\big|$ satisfies the bound \eqref{F2 even}.
 				
 			For $\eta r\gtrsim 1$ and $0<\eta\ll1$,  by \eqref{F8 even} and \eqref{F10 even}, $\big|\partial_\eta^{\ell}F(\eta, r)\big|$ satisfies the bound \eqref{F1 even}.
 				
 				Note that for $\eta r\ll 1$ and $\ell\in\N^+\cup\{0\},$  one has $r^\ell \lesssim\eta^{-\ell}\lesssim\eta^{-\ell}\big\langle\eta r \big\rangle^{\frac{n-3}{2}}$ and 
 					\begin{align*}
 						&r^4\big|\ln(\eta r)\big| =\eta^{-4}(\eta r)^4\big|\ln(\eta  r)\big|\lesssim\eta^{-4}\big\langle\eta r \big\rangle^{\frac{n-3}{2}}\ \
 						r^\ell(\eta r)^{-(\ell-4)}=\eta^{-\ell}(\eta r)^4\lesssim\eta^{-\ell}\big\langle\eta r \big\rangle^{\frac{n-3}{2}}.
 					\end{align*}
 				Furthermore,	combining with \eqref{F10 odd}, we derive \eqref{F3 even}.
 				
 				Thus we complete the proof. 
 		\end{proof}
 Next  we will establish  the estimate of the following kernel  for fixed $k\in\N^+$ enough large depending on $n$:
 		\begin{align}\label{A}
 		 A(\eta)(z_1, z_2):=\big(R_0^+(\eta^4+\eta^2)V\big)^{k-1}R_0^+(\eta^4+\eta^2)(z_1, z_2), 
 	\end{align}
 	which plays an important role in studying the low energy part $\Omega_L^+$ defined by \eqref{WL}. 
 	\begin{lemma}\label{lemma A}
 		Let $n\geq 5$ and $V(x)\lesssim\langle x \rangle^{-\beta}.$ Then for fixed $k\in\N^+$ enough large depending on $n$ and $
 		  \ell=0,1,\cdots, \Big\lceil\frac{n}{2}\Big\rceil +1,$ provided $\beta>\frac{n+1}{2}+2,$ one has
 		\begin{align}\label{A1}
 			\sup_{0<\eta\ll 1}\Big| \eta^{\max(0, \ell-1)}\partial_\eta^\ell A(\eta)(z_1, z_2) \Big|\lesssim	\begin{cases}
 				\langle z_1\rangle^2\langle z_2\rangle^2  & {\rm if}\  n\  {\rm is\  odd},\\
 			\langle z_1\rangle^{\frac{3}{2}}\langle z_2\rangle^{\frac{3}{2}}  & {\rm if}\  n\  {\rm is \ even}.
 			\end{cases}
 			\end{align}
 		\end{lemma}
 		\begin{proof} 
 			By virtue of the bounds  \eqref{F1 odd}, \eqref{F1 even}, and \eqref{F2 even}, we can check that
 			\begin{align}\label{Fodd even}
 				\Big| \eta^{\max(0, \ell-1)}\partial_\eta^\ell F(\eta, |x-y|) \Big|\lesssim\langle x-y\rangle^\ell\big\langle\eta r \big\rangle^{\frac{n-3}{2}} \ \ \ \  \ell=0,1,\cdots, \Big\lceil\frac{n}{2}\Big\rceil +1.
 			\end{align}
 		Then for  $0<\eta\ll 1,$   considering  \eqref{R0F}, \eqref{Fodd even} and  Leibniz formula,  it follows that
 				\begin{align*}
 				\Big| \eta^{\max(0, \ell-1)}\partial_\eta^\ell A(\eta)(z_1, z_2) \Big|\lesssim \sum_{t=1}^{k}\sum_{s=0}^{\ell}&\int_{\R^{(k-1)n}}
 			|u_{t-1}-u_t|^s\bigg(\prod_{j=1}^{k-1}
 				\frac{\big \langle u_{j-1}-u_j\big\rangle^{\frac{n+1}{2}-2}}{|u_{j-1}-u_j|^{n-2}}\big|V(u_j)\big|\bigg)	\\
 				&\times	\frac{\big \langle u_{k-1}-u_k\big\rangle^{\frac{n+1}{2}-2}}{|u_{k-1}-u_k|^{n-2}}d\vec{u}:=\sum_{t=1}^{k}\sum_{s=0}^{\ell}A_{s,t }(z_1, z_2),
 			\end{align*}
 				where $u_0=z_1,$ $u_k=z_2,$ and $\vec{u}=(u_1, u_2,\dots, u_{k-1}).$
 	We only focus on the following $A_{\ell,k}(z_1, z_2)$ for $\ell = 0, 1, \cdots, \Big\lceil \frac{n}{2} \Big\rceil + 1$, as the remaining cases can be deduced accordingly.
 		\begin{align*}
 		 A_{\ell,k}(z_1, z_2)= \int_{\R^{(k-1)n}}
 		\bigg(\prod_{j=1}^{k-1}
 		\frac{\big \langle u_{j-1}-u_j\big\rangle^{\frac{n+1}{2}-2}}{|u_{j-1}-u_j|^{n-2}}\big|V(u_j)\big|\bigg)
 		\frac{\big \langle u_{k-1}-u_k\big\rangle^{\frac{n+1}{2}-2}}{|u_{k-1}-u_k|^{n-2-\ell}}d\vec{u}.
 	\end{align*}
 For simplicity, 	set $$\Pi_A:=	\int_{A} 
 	\frac{\big \langle u_0-u_j\big\rangle^{\frac{n+1}{2}-2a}}{|u_0-u_j|^{n-2a}}\big|V(u_j)\big|
 	\frac{\big \langle u_j-u_{j+1}\big\rangle^{\frac{n+1}{2}-2}}{|u_j-u_{j+1}|^{n-2}}du_j,$$
 	where $A\subseteq\R^n$ denotes the integration region.
 	
 \underline{ we  first claim that for $a\in \N^+$ and fixed  $0<\delta\ll1,$ }
 	 \begin{align*}
\Pi_{\R^n}  \lesssim
 		\begin{cases}
 		{\big \langle u_0-u_{j+1}\big\rangle^{\frac{n+1}{2}-2(a+1)}}{\big|u_0-u_{j+1}\big|^{-n+2(a+1)}},& \ \  {\rm if}\ 1\leq a\leq
 		\lceil\frac{n}{2}\rceil-2,\\ 
 	|u_0-u_{j+1}|^{-\delta} \big \langle u_0-u_{j+1}\big\rangle^{-\frac{n-1}{2}+\delta},& \ \  {\rm if}\ a=\lceil\frac{n}{2}\rceil-1,
 		\end{cases}
 	\end{align*}
 	and  if $a>\lceil\frac{n}{2}\rceil-1$, then  $\Pi_{\R^n} \lesssim \big \langle u_0-u_{j+1}\big\rangle^{-\frac{n-1}{2}}.$ 
 	
 	Indeed, we decompose $\mathbb{R}^n$ into the four parts:
 		\begin{align*}
 &	E_1:=\big\{u_j\in\R^n\big| |u_0-u_j|\geq1,\ |u_j-u_{j+1}|\geq1 \big\},\ \ 	E_2:=\big\{u_j\in\R^n\big| |u_0-u_j|\geq1,\ |u_j-u_{j+1}|<1 \big\},\\
 &	E_3:=\big\{u_j\in\R^n\big| |u_0-u_j|<1,\ |u_j-u_{j+1}|\geq1 \big\},\ \ 	E_4:=\big\{u_j\in\R^n\big| |u_0-u_j|<1,\ |u_j-u_{j+1}|<1 \big\}.
 	\end{align*}
 	Note that for $x\in\R^n,$ $\langle x\rangle\approx 1$ if $|x|\leq1$ and
$\langle x\rangle\approx |x|$ if $|x|>1.$   Then $	\Pi_{E_1}$ satisfies
 	 	 \begin{align}\label{E1-1}
 		\overset{(1)}\lesssim 	\int_{E_1}
 		\frac{\big|V(u_j)\big|}{|u_0-u_j|^{\frac{n-1}{2}}|u_j-u_{j+1}|^{\frac{n-1}{2}}}du_j
 \lesssim 	\int_{E_1}\Big(
 		\frac{1}{|u_0-u_j|^{n-1}}+\frac{1}{|u_j-u_{j+1}|^{n-1}}\Big)\big|V(u_j)\big|du_j,
 		\end{align}
 		which is dominated  by $\lesssim1,$ followed from
 		Lemma 
 		\ref {Appendix 1} provided $\beta>1.$ Additionally, observe that
 		 \begin{align*}
 			\frac{1}{|u_0-u_j|^{\frac{n-1}{2}}|u_j-u_{j+1}|^{\frac{n-1}{2}} } \lesssim\frac{1}{|u_0-u_{j+1}|^{\frac{n-1}{2}}}
 			\left(\frac{1}{|u_0-u_j|^{\frac{n-1}{2}} }+\frac{1}{|u_j-u_{j+1}|^{\frac{n-1}{2}} }\right),
 		\end{align*}
 		where the inequality is established by $$|u_0-u_{j+1}|^{\frac{n-1}{2}}\lesssim|u_0-u_j|^{\frac{n-1}{2}}+|u_j-u_{j+1}|^{\frac{n-1}{2}}.$$ 
 		Putting the inequality above  into the first bound in  \eqref{E1-1}, by Lemma \ref{Appendix 1}, we derive $	\Pi_{E_1}\lesssim |u_0-u_{j+1}|^{-\frac{n-1}{2}},$ provided $\beta>\frac{n+1}{2}.$ Considering that $	\Pi_{E_1}\lesssim1$, we conclude that 
 		\begin{align*}
 			\Pi_{E_1}
 			\lesssim\big \langle u_0-u_{j+1}\big\rangle^{-\frac{n-1}{2}}, \ \  a\in\N^+.
 		\end{align*}
 	By the same way,  we can check that $\Pi_{E_i}$ satisfies $\lesssim\big \langle u_0-u_{j+1}\big\rangle^{-\frac{n-1}{2}}$ for $i=2,3$ and $a\in\N^+.$
 	Additionally, it is not hard to find that 
 		$$
 	\big \langle u_0-u_{j+1}\big\rangle^{-\frac{n-1}{2}}	\lesssim	{\big \langle u_0-u_{j+1}\big\rangle^{\frac{n+1}{2}-2(a+1)}}{\big|u_0-u_{j+1}\big|^{-n+2(a+1)}},\  \  \  a=1,2, \cdots, 
 		\Big\lceil\frac{n}{2}\Big\rceil-2. $$
 		As for the case of the integration region $E_4,$ 
 	we derive for $a=1,2, \dots,
 	\lceil\frac{n}{2}\rceil-2,$
 			 \begin{align*}
 				\Pi_{E_4}\lesssim 	\int_{E_4}
 			\frac{\big|V(u_j)\big|}{|u_0-u_j|^{n-2a}|u_j-u_{j+1}|^{n-2}}du_j\lesssim 
 			\frac{1}{|u_0-u_{j+1}|^{n-2(a+1)}},
 		\end{align*}
 		 where the last inequality follows from Lemma \ref{Appendix 2} since $n-2a+n-2-n>0$ for $a\leq
 		 \lceil\frac{n}{2}\rceil-2$.
 		 
 		 If $a=
 		 \lceil\frac{n}{2}\rceil-1$  and $n$ is even (i.e. $a=\frac{n}{2}-1$ and $n$ is even), then $0<\delta\ll1,$
 		  \begin{align*}
 		 	\int_{E_4}
 		 	\frac{\big|V(u_j)\big|}{|u_0-u_j|^{2}|u_j-u_{j+1}|^{n-2}}du_j
 		 	\lesssim \frac{1}{	|u_0-u_{j+1}|^{\delta}}\int_{E_4}
 		 	\frac{|u_0-u_j|^{\delta}+|u_j-u_{j+1}|^{\delta}}{|u_0-u_j|^{2}|u_j-u_{j+1}|^{n-2}}du_j\lesssim \frac{1}{	|u_0-u_{j+1}|^{\delta}},
 		 \end{align*}
 		 where the finally inequality is deduced by the  integration region $E_4$ and $n-\delta <n.$
 		 
 		 If $a\notin\big\{1,2, \cdots,
 		\lceil\frac{n}{2}\rceil-2\big\}\bigcup\big\{\lceil\frac{n}{2}\rceil-1\big|\  n\ \text{ is even}\big\},$ then $n-2a+n-2<n,$ which yields that
 		 $$	\int_{E_4}
 		 \frac{\big|V(u_j)\big|}{|u_0-u_j|^{n-2a}|u_j-u_{j+1}|^{n-2}}du_j \lesssim 1.
 		$$
 		Hence we complete the first claim.
 		
 Additionally,	note that $\big \langle u_0-u_{j}\big\rangle^{-\frac{n-1}{2}}\lesssim	|u_0-u_j|^{-\delta}\big \langle u_0-u_j\big\rangle^{-\frac{n-1}{2}+\delta}$ and 
 	\begin{align*}
 			&\int_{\R^n}
 				|u_0-u_j|^{-\delta}\big \langle u_0-u_j\big\rangle^{-\frac{n-1}{2}+\delta}\big|V(u_j)\big|
 			\frac{\big \langle u_j-u_{j+1}\big\rangle^{\frac{n+1}{2}-2}}{|u_j-u_{j+1}|^{n-2}}du_j
 			\lesssim \big \langle u_0-u_{j+1}\big\rangle^{-\frac{n-1}{2}},
 		\end{align*}
 		by Lemma \ref{Appendix 1} with $\beta>\frac{n+1}{2}.$
 	Thus, choose $k$ enough large $ (e.g. \  k\geq\lceil\frac{n}{2}\rceil+1)$ to conclude  that 
 	\begin{align*}
 	 A_{\ell,k}(z_1, z_2)\lesssim	\int_{\R^n}\big \langle u_0-u_{k-1}\big\rangle^{-\frac{n-1}{2}}
 	 \big|V(u_{k-1})\big|
 	 \frac{\big \langle u_{k-1}-u_k\big\rangle^{\frac{n+1}{2}-2}}{|u_{k-1}-u_k|^{n-2-\ell}}du_{k-1}.
 	 \end{align*}
 	 For $|u_{k-1}-u_k|\leq1,$ it is easy to get $ A_{\ell,k}(z_1, z_2)\lesssim1.$ For $|u_{k-1}-u_k|\geq1$ and  $\ell=0,1,\cdots, \Big\lceil\frac{n}{2}\Big\rceil +1,$  we derive that
 	 \begin{align*}
 		\int_{\R^n}\big \langle u_0-u_{k-1}\big\rangle^{-\frac{n-1}{2}}
 	\frac{	\big|V(u_{k-1})\big|}{|u_{k-1}-u_k|^
 		{\frac{n-1}{2}-\ell}}du_{k-1}
 	& \lesssim
 	 \int_{\R^n}\frac{\big|V(u_{k-1})\big|}{\big \langle u_0-u_{k-1}\big\rangle^{\frac{n-1}{2}}} \big|u_{k-1}-u_k\big|^{\lceil\frac{n}{2}\rceil-\frac{n}{2}+\frac{3}{2}}du_{k-1},
 	 \end{align*}
 	which is bounded by $\lesssim\langle u_k \rangle^{\lceil\frac{n}{2}\rceil-\frac{n}{2}+\frac{3}{2}},$  deduced by Lemma \ref{Appendix 1}, provided $\beta>\frac{n+1}{2}+2.$
 	 
 	 Similarly, by choosing $k$ sufficiently large depending on $n$, we obtain $A_{\ell,1}(z_1, z_2) \lesssim \langle u_0 \rangle^{\lceil\frac{n}{2}\rceil-\frac{n}{2}+\frac{3}{2}}$. Furthermore, if $t \neq 1, k$, we integrate starting from the farthest end up to the $t$-th term, yielding the bound $\lesssim 1$. 	 Thus we complete the proof. 
 		\end{proof}
 		
 	Next, setting $v(x)=|V(x)|^{1/2},$  we claim that  the operator 
 $\mathcal{R}_\ell$ with kernel:    
 	\begin{align}\label{Rl}
 		\mathcal{R}_\ell(x,y):=v(x)v(y)\sup_{0<\eta\ll1}
 		\Big|\eta^{\max(0, \ell-1)}\partial_\eta^{\ell}R_0^+(\eta^4+\eta^2)(x,y) \Big|
 		\end{align}	    
 	is bounded on $L^2,$ under the condition that 
 	 $|V(x)|\lesssim\langle x \rangle^{-\beta} $ for some $\beta>0$.
 	\begin{lemma}\label{lemma Rj}
 		Let $n\geq5$ and $r=|x-y|.$ Then for  $\ell\in\N^+\cup\{0\}, $ 
 		\begin{align*}
 			\sup_{0<\eta\ll1}
 			\Big|\eta^{\max(0, \ell-1)}\partial_\eta^{\ell}R_0^+(\eta^4+\eta^2)(x,y) \Big|\lesssim	
 			r^{-n+2} + r^{-\frac{n-1}{2}+\ell}.
 			\end{align*}
 			Specifically, for  $\ell\in\N^+\cup\{0\}, $ 
 				\begin{align*}
 				\sup_{0<\eta\ll1}
 				\Big|\eta^{\max(0, \ell-1)}\partial_\eta^{\ell}R_0^+(\eta^4+\eta^2)(x,y) \Big|\lesssim
 				\begin{cases}
 					r^{-n+2}+r^{-n+3} & {\rm if}\ \eta r\ll1,\\ 
 					r^{-\frac{n-1}{2}+\ell} & {\rm if}\ \eta r\gtrsim1.
 				\end{cases}
 			\end{align*}
 			Moreover, $\mathcal{R}_\ell\in\mathbb{B}(L^2)$ for $\ell=0,1,2,\cdots, \lceil\frac{n}{2}\rceil +1,$  provided  
 			$|V(x)|\lesssim\langle x \rangle^{- n-4-}.$
 	\end{lemma}
 	\begin{proof}
 	At beginning, 	for $n\geq5,$ taking into account  \eqref{R0F},  \eqref{F1 odd}, \eqref{F1 even} and \eqref{F2 even}, we derive the desired result for $\ell=0.$
 		
 	Additionally,	by Leibniz formula and \eqref{R0F} ,  for $\ell\in\N^+,$
 		\begin{align}\label{D-F}
 			\Big|\eta^{ \ell-1}\partial_\eta^{\ell}R_0^+(\eta^4+\eta^2)(x,y) \Big|&=
 			\bigg| 	r^{-n+2}\ \eta^{ \ell-1}\sum_{\ell_1+\ell_2+\ell_3=\ell} 
 			\Big(\frac{d^{\ell_1}}{d\eta^{\ell_1}}
 			\big(1+2\eta^2\big)^{-1}\Big)
 			\Big(\partial_\eta^{\ell_2}e^{i\eta r}\Big)
 			\partial_\eta^{\ell_3}F(\eta, r) \bigg|\nonumber\\
 			&\lesssim r^{-n+2}\ \eta^{ \ell-1}\sum_{\ell_1+\ell_2+\ell_3=\ell} 
 			r^{\ell_2}\big|\partial_\eta^{\ell_3}F(\eta, r)\big|.
 		\end{align}
 		
 When $n\geq5$ and $n$ is odd, for $\ell\in\N^+,$ if $ \eta r\gtrsim1$ (noting that $r\gtrsim1$ by $0<\eta\ll1$ ), then followed from \eqref{F1 odd},
 		 $$\Big|\partial_\eta^{\ell_3}F(\eta, r)\Big|\lesssim r^{\ell_3}\big(\eta r\big)^{\frac{n-3}{2}-\ell_3}, \ \  \ell_3=0,1,2,\cdots.$$
 		 Hence, by virtue of \eqref{D-F}, we derive for $\ell\in\N^+,$  $ \eta r\gtrsim1$ and  $0<\eta\ll1$,
 		       	\begin{align}\label{D-F-O1}
 		       	\Big|\eta^{ \ell-1}\partial_\eta^{\ell}R_0^+(\eta^4+\eta^2)(x,y) \Big|
 		       	&\lesssim r^{-n+2}\ \eta^{ \ell-1}\sum_{\ell_1+\ell_2+\ell_3=\ell} 
 		       	r^{\ell_2+\ell_3}\big(\eta r\big)^{\frac{n-3}{2}}
 		       	\lesssim  r^{-\frac{n-1}{2}+\ell} .
 		       \end{align}
 		       As for $  \eta r\ll1$, combining with \eqref{F1 odd} and \eqref{D-F}, it follows that
 		       	\begin{align*}
 		       	\Big|\eta^{ \ell-1}\partial_\eta^{\ell}R_0^+(\eta^4+\eta^2)(x,y) \Big|
 		       	&\lesssim r^{-n+2}\ \eta^{ \ell-1}\sum_{\ell_1+\ell_2+\ell_3=\ell} 
 		       	r^{\ell_2}\Big(r^{\ell_3}+	\big\langle r \big\rangle^{\frac{n-3}{2}+\ell_3}e^{-r}\Big).
 		       \end{align*}
 		 Since    for $\ell\in\N^+,$  $ \eta r\ll1$ and $0<\eta\ll1$,    
 		       	\begin{align*}
 		       & r^{-n+2}\ \eta^{ \ell-1}\sum_{\ell_1+\ell_2+\ell_3=\ell} 
 		       	r^{\ell_2+\ell_3}= r^{-n+2}\sum_{0\leq\ell_1\leq\ell-1} 
 		       r\big(\eta r\big)^{\ell-\ell_1-1}\eta^{\ell_1}+r^{-n+2}\ \eta^{ \ell-1}\lesssim r^{-n+3}+r^{-n+2},\\
 		       & r^{-n+2}\ \eta^{ \ell-1}\sum_{\ell_1+\ell_2+\ell_3=\ell} 	r^{\ell_2}\big\langle r \big\rangle^{\frac{n-3}{2}+\ell_3}e^{-r}\lesssim r^{-n+2}.
 		       \end{align*}
 		       Then  it follows that for $\ell\in\N^+,$  $ \eta r\ll1$ and $0<\eta\ll1$,
 		       	\begin{align}\label{low R}
 		       	\Big|\eta^{ \ell-1}\partial_\eta^{\ell}R_0^+(\eta^4+\eta^2)(x,y) \Big|
 		       	\lesssim r^{-n+3}+r^{-n+2}.
 		       \end{align}
 		       
 		  We consider the case  where $n\geq5$ and $n$ is even  by employing a similar decomposition based on $ \eta r\gtrsim1$ and $ \eta r\ll1$.  
 		  
 		  Next, we omit the details since the argument proceeds in the same manner as in the case when $n$ is odd.
 		     For $ \eta r\gtrsim1$ and $\ell\in \N^+$,  by virtue of   \eqref{F1 even} and  \eqref{D-F}, we also obtain  the bound \eqref{D-F-O1}.
 		      For $ \eta r\ll1,$ $0<\eta\ll1$ and $\ell\in \N^+$, considering \eqref{F2 even} and \eqref{D-F}, we also get the bound 
\eqref{low R}.
  
  As a result, it follows that for $ \ell\in\N^+\cup\{0\},$
  \begin{align*}
  	\sup_{0<\eta\ll1}
  	\Big|\eta^{\max(0, \ell-1)}\partial_\eta^{\ell}R_0^+(\eta^4+\eta^2)(x,y) \Big|\lesssim	
  	r^{-n+2} +	r^{-n+3}+ r^{-\frac{n-1}{2}+\ell}\lesssim	
  	r^{-n+2} + r^{-\frac{n-1}{2}+\ell}.
  \end{align*}   
  
  Finally, we establish that $\mathcal{R}_\ell \in \mathbb{B}(L^2) $ for $0 \leq\ell \leq \lceil n/2 \rceil + 1 $. By Lemma \ref{Appendix 1}, we deduce that 
  $$
  v(x)v(y)r^{-\frac{n-1}{2}+\ell} \in L^2(\mathbb{R}^{2n}),
  $$
  provided that 
  $
  |V(x)| \lesssim \langle x \rangle^{- 2\lceil n/2 \rceil - 3 -}
  $ (noting that $2\lceil n/2 \rceil +3 =n+3$, if $n$ is even and  $2\lceil n/2 \rceil +3 =n+4$, if $n$ is odd).
   This implies that the operator with kernel $ v(x)v(y)r^{-\frac{n-1}{2}} $  is a Hilbert--Schmidt operator. Furthermore, by \cite[Proposition 3.2]{Goldberg-Visan-CMP}, the operator with kernel $v(x)v(y)r^{-n+2}$ is bounded on $ L^2$, provided that 
  $
  |V(x)| \lesssim \langle x \rangle^{-2}.
  $
 Thus we complete the proof. 
	\end{proof}
 	
 	Now we give the expression of $(R_0^+-R_0^-)(\eta^4+\eta^2).$
 	
 	\begin{lemma}\label{lemma R0-R0}
 		Let $n\geq5$ and $r=|x-y|.$ Then 
 		\begin{align}\label{F+-}
 			\Big[(R_0^+-R_0^-)(\eta^4+\eta^2)\Big](x, y)=\frac{\eta^{n-2}}{1+2\eta^2}\Big(e^{i\eta r}\Phi_+(\eta r)+e^{-i\eta r}\Phi_-(\eta r)\Big),
 			\end{align}
 			where $\Phi_\pm\in C^{\infty}(\R)$ such that
 			\begin{align*}
 		\Big|\partial_\eta^{\ell}\Phi_\pm(\eta r)\Big|\lesssim\eta^{-\ell}\big\langle \eta r\big\rangle^{\frac{1-n}{2}},  \ \  \  \ell=0,1,2,\cdots.
 			\end{align*}
 	\end{lemma}
 	\begin{proof}
 		By the splitting identity \eqref{split}, it follows that
 			\begin{align*}
 			\Big[R_0^+(\eta^4+\eta^2)-R_0^-(\eta^4+\eta^2)\Big](x, y)=\frac{1}{1+2\eta^2}\Big[R_\Delta^+(\eta^2)-R_\Delta^-(\eta^2)\Big](x, y).
 		\end{align*}
 		Combining with  \cite[Lemma 2.4 and (6)]{Erdogan-Green23}, we  establish this lemma.
 	\end{proof}

 	In the following, we start to research  $R_V^+(\eta^4+\eta^2).$ Consider  that 
 	\begin{align*}
 		R_V^+(\eta^4+\eta^2)V=R_0^+(\eta^4+\eta^2)v M^{-1}(\eta)v,
 	\end{align*}
 	  where $M(\eta)=U+vR_0^+(\eta^4+\eta^2)v,$ $v(x)=|V(x)|^{1/2}$ and $U(x)=\sgn V(x)$, namely $U(x)=1$ if $V(x)\ge0$ and $U(x)=-1$ if $V(x)<0$. Here, denote $M^{-1}(\eta):=[M(\eta)]^{-1}$ as long as it exists.
 	  
 	  In the following, we show the existence of $ M^{-1}(\eta)$  and  analyze the operator $\partial_\eta^\ell M^{-1}(\eta)$  for  $\ell=0,1,\cdots, \big\lceil\frac{n}{2}\big\rceil +1$, under the condition that
 	  $|V(x)|\lesssim\langle x \rangle^{- n-4-}$ (noting that the decay rate $\beta>n+4$ is required by Lemma \ref{lemma Rj}).
 	  
 	  First, we define the operator:
 	  \begin{align*} 
 	  	T_0:=U+vR_0(0)v=U+v\Big((-\Delta)^{-1}-(-\Delta+1)^{-1}\Big)v\in \mathbb{AB}(L^2).
 	  \end{align*}
 	  Note that  the assumption that zero is a regular point of $H$ i.e.  the operator $T_0$ is invertible on $L^2$ (See Definition \ref{definition regular}) and we emphasize that $T_0^{-1}\in \mathbb{AB}(L^2).$
 	  
 	  Continue to denote the operator $E(\eta)$ with kernel:
 	   $$	E(\eta)(x,y)=v(x)\Big(R_0^+(\eta^4+\eta^2)-R_0(0)\Big)(x, y)v(y).$$
 	 By the mean value theorem, it follows that 
 	 \begin{align} \label{E-eta}
 	 	\big|E(\eta)(x,y)\big|=v(x)\Big|\Big(R_0^+(\eta^4+\eta^2)-R_0(0)\Big)(x, y)\Big|v(y)\lesssim\eta \mathcal{R}_1(x,y),
 	 \end{align}
 		where $\mathcal{R}_1(x,y)$  is defined by $\eqref{Rl}.$   Combine with Lemma \ref{lemma Rj} and \eqref{E-eta}, then  $E(\eta)\in \mathbb{AB}(L^2)$ and $\norm{E(\eta)}_{\mathbb{B}(L^2)}\to 0$ as $\eta\to 0.$ Furthermore, by a Neumann series expansion and the invertibility of $T_0,$ we obtain that $M^{-1}(\eta)$ is invertible  and 
 		\begin{align} \label{M-1}
 		M^{-1}(\eta)=\Big(T_0+E(\eta)\Big)^{-1}
 		=\sum_{j=0}^{\infty}(-1)^j\ T_0^{-1}\Big(E(\eta)T_0^{-1}\Big)^j\in \mathbb{AB}(L^2).
 		\end{align}
 		In particular, by \eqref{E-eta} and \eqref {M-1}, the operator $\widetilde{M}_0$ with kernel: 
 			\begin{align} \label{M0}
 			\widetilde{M}_0(x, y)=\sup_{0<\eta\ll1}\Big |  M^{-1}(\eta)(x, y) \Big|,
 		\end{align}
 		is bounded on $L^2.$ Since $M(\eta)=U+vR_0^+(\eta^4+\eta^2)v,$ then we can check that the operator $\eta^{\ell-1}\partial_\eta^{\ell}M^{-1}(\eta)$  (for  $\ell\in\N^+$)  is a linear combination of operators of the form 
 			\begin{align*} 
 	 M^{-1}(\eta)\eta^{-1}\prod_{j=1}^J\left[v\bigg(\eta^{\ell_j}\partial_\eta^{\ell_j}R_0^+\big(\eta^4+\eta^2\big)\bigg)
 	 v M^{-1}(\eta)\right],
 		\end{align*}
 		where $\sum_{j=1}^{J}\ell_j=\ell,$ and  $\ell_j\in\N^+$. Hence, utilize Lemma \ref{lemma Rj} and \eqref{M0} to establish that the operator $\widetilde{M}_\ell$ with kernel
 			\begin{align} \label{MM}
 			\widetilde{M}_\ell(x, y)=\sup_{0<\eta\ll1}\Big | \eta^{\max(0, \ell-1)}\partial_\eta^{\ell}M^{-1}(\eta)(x, y)  \Big|,
 		\end{align}
 		is bounded on $L^2$ for $\ell=0,1,2,\cdots, \lceil\frac{n}{2}\rceil +1,$  provided $|V(x)|\lesssim\langle x \rangle^{- n-4-}.$
 		
 		Now, we aim to  provide  another expression of $\Omega_L^-$ defined by \eqref{WL}.
 	 Let $w(x)= |V(x)|^{1/2}\sgn V(x).$  Note that
 			\begin{align*} 
 			\big(R_0^+ V\big)^kR_V^+\big(VR_0^+\big)^kV\big(R_0^+-R_0^-\big)
 			=R_0^+v\Big(wA(\eta)v M^{-1}(\eta)vA(\eta)w\Big)v \big(R_0^+-R_0^-\big).
 		\end{align*}
 		Define  the operator  $\Gamma(\eta)$: $$\Gamma(\eta):=wA(\eta)v M^{-1}(\eta)vA(\eta)w.$$
 	 Hence, the low energy part $\Omega_L^-$ defined by \eqref{WL} can be expressed as:
  \begin{align}\label{WL2}
 	\Omega_L^-:=\frac{1}{\pi i}\int_0^\infty\eta(2\eta^2+1)\chi(\eta)R_0^+(\eta^4+\eta^2)v\Gamma(\eta)v \Big(R_0^+(\eta^4+\eta^2)-R_0^-(\eta^4+\eta^2)\Big)d\eta.
 \end{align}
 
We conclude that provided $|V(x)|\lesssim\langle x \rangle^{- n-4-},$
 		    	\begin{align} \label{Gamma}
 		    	\widetilde{\Gamma}(x, y)=\max_{\ell\in\big\{0,1,\cdots,\lceil\frac{n}{2}\rceil +1\big\}}\sup_{0<\eta\ll1}\Big | \eta^{\max(0, \ell-1)}\partial_\eta^{\ell}\Gamma(\eta)(x,y) \Big|\lesssim \langle x\rangle^{-\frac{n}{2}-}\langle y\rangle^{-\frac{n}{2}-},
 		    \end{align}
 		    where $\Gamma(\eta)(x,y)$ is the kernel of the operator $\Gamma(\eta).$
 		    
 		Specifically, 	by  Lemma \ref{lemma A} and \eqref {MM},   $\big | \eta^{\max(0, \ell-1)}\partial_\eta^{\ell}\Gamma(\eta)(x,y) \big|$ is bounded by
 		\begin{align*}    
 	 &v(x)v(y)\sum_{\ell_1+\ell_2+\ell_3=\ell}\eta^{\max(0, \ell-1)}\int_{\R^{2n}}\Big|\partial_\eta^{\ell_1}A(\eta)(x, u_1)\Big| v(u_1)\Big|\partial_\eta^{\ell_2}M^{-1}(\eta)(x, y)  \Big|v(u_2)\Big|\partial_\eta^{\ell_3}A(\eta)( u_2, y)\Big|du_1du_2\\
 	&\lesssim\big\langle x \big\rangle^{-\frac{n}{2}-}
 	  \big\langle y \big\rangle^{-\frac{n}{2}-}\norm{\langle \cdot \rangle^{-\frac{n}{2}-}}_{L^2}^2\sum_{\ell_2=0}^{\lceil\frac{n}{2}\rceil +1} \norm{\widetilde{M}_{\ell_2}}_{\mathbb{B}(L^2)}\lesssim\big\langle x \big\rangle^{-\frac{n}{2}-}
 	  \big\langle y \big\rangle^{-\frac{n}{2}-},
 	 \end{align*}  
	for $0<\eta\ll1$ and $\ell=0,1,2,\cdots, \lceil\frac{n}{2}\rceil +1,$  provided $|V(x)|\lesssim\langle x \rangle^{- n-4-}.$  
 \subsection{The proof of Theorem \ref{theorem WL}} 	
Next, we employ \eqref{WL2} and \eqref{Gamma} to demonstrate the $L^p$ boundedness of the low energy part  $\Omega_L^-$ for $1 \leq p \leq \infty$. This result is covered in Proposition \ref{proposition-WL} below.
\begin{proposition}\label{proposition-WL}
 		    Let $n\geq5$ and  $\Gamma(\eta)\in \mathbb{AB}(L^2)$  be an operator-valued function with respect to the variable $\eta$ such that 
 		    	\begin{align}\label{Gamma2} 
 		    	\widetilde{\Gamma}(x, y):=\max_{\ell\in\big\{0,1,\cdots,\lceil\frac{n}{2}\rceil +1\big\}}\sup_{0<\eta\ll1}\Big | \eta^{\max(0, \ell-1)}\partial_\eta^{\ell}\Gamma(\eta)(x,y) \Big|\lesssim \langle x\rangle^{-\frac{n}{2}-}\langle y\rangle^{-\frac{n}{2}-}.
 		    \end{align}
 		    Then the operator $K$ with kernel:
 		    	\begin{align*} 
 K(x, y)=\frac{1}{\pi i}\int_0^\infty\eta(2\eta^2+1)\chi(\eta)\left[R_0^+(\eta^4+\eta^2)v\Gamma(\eta)v \Big(R_0^+(\eta^4+\eta^2)-R_0^-(\eta^4+\eta^2)\Big)\right](x, y)d\eta,
 		    \end{align*}
 		    is bounded on $L^p(\R^n)$ for all $1\leq p \leq \infty,$ provided $|V(x)|\lesssim\langle x \rangle^{- n-}.$
 		    \end{proposition}

 Before  the proof of Proposition \ref{proposition-WL}, we introduce the definition of an admissible operator. Specifically, we say that an operator $K$ is admissible if its kernel $K(x, y)$ satisfies the condition 
 \begin{align}\label{Schur test}
 	\sup_{x \in \mathbb{R}^n} \int_{\mathbb{R}^n} \big|K(x, y)\big| \, dy 
 	+ \sup_{y \in \mathbb{R}^n} \int_{\mathbb{R}^n} \big|K(x, y)\big| \, dx 
 	< +\infty.
 \end{align}
We note that, by the Schur test, every admissible operator is bounded on $L^p$ for all $1 \leq p \leq \infty$.
 \begin{proof}[Proof of Proposition \ref{proposition-WL}]	  
Utilize $\eqref{R0F}$  and $ \eqref{F+-}$ to  get  $K(x, y)=\mathcal{K}^+(x, y)+\mathcal{K}^-(x, y),$  where
\begin{align*}
	\mathcal{K}^{\pm}(x, y):=\frac{1}{\pi i}\int_{\R^{2n}}\frac {v_{12}}{r_1^{n-2}}&\int_0^\infty e^{i\eta(r_1\pm r_2)}\chi(\eta)\frac{\eta^{n-1}}{1+2\eta^2}F(\eta, r_1)\Phi_\pm(\eta r_2)\Gamma(\eta)(z_1, z_2)d\eta dz_1dz_2.
\end{align*}
Here $r_1:=|x-z_1|, r_2:=|y-z_2|,$ and $ v_{12}:=v(z_1)v(z_2)$ for short.

Note that
$\big|\partial_\eta^{\ell}F(\eta, r)\big|\lesssim \eta^{-\ell}\big\langle\eta r \big\rangle^{\frac{n-3}{2}}$ and $\big|\partial_\eta^{\ell}\Phi_\pm(\eta r)\big|\lesssim\eta^{-\ell}\big\langle \eta r\big\rangle^{\frac{1-n}{2}}$ for $\ell\in\N^+\cup\{0\}$
followed from Lemma \ref{lemma F odd}-\ref{lemma F even}
and Lemma \ref{lemma R0-R0}. These estimates are used frequently in the whole proof.

We decompose $K(x, y)$  into $$K(x,y)=\sum_{j=1}^{4}K_j(x, y),$$
where the integrand of $K_j(x, y)$ is defined as the restriction of the integrand of $K(x, y)$ to the following subsets:  
\begin{align*}
	&K_1(x, y): \quad r_1, r_2 \lesssim 1; \ \ \ 
	K_2(x, y): \quad r_1 \approx r_2 \gg 1; \\
	&K_3(x, y): \quad r_1 \gg \langle r_2 \rangle; \ \ \ 
	K_4(x, y): \quad r_2 \gg \langle r_1 \rangle.
\end{align*}
Similarly, we define $	\mathcal{K}_j^{\pm}(x, y)$ from $	\mathcal{K}^{\pm}(x, y)$ for $j=1,2,3,4.$

Next,  it suffices to  show that each operator $K_j$ is admissible  for $j=1,2,3,4.$  
 \vskip0.2cm
 $\bullet$ \underline{ First we deal with the operator $K_1$ } with kernel $K_1(x,y)=\mathcal{K}^+_1(x,y)+\mathcal{K}^-_1(x,y),$ where
\begin{align*}
	\mathcal{K}_1^{\pm}(x, y)=\frac{1}{\pi i}\int_{\R^{2n}}\frac {v_{12}\chi_{r_1, r_2\lesssim1}}{r_1^{n-2}}\int_0^\infty e^{i\eta(r_1\pm r_2)}\chi(\eta)\frac{\eta^{n-1}}{1+2\eta^2}F(\eta, r_1)\Phi_\pm(\eta r_2)\Gamma(\eta)(z_1, z_2)d\eta dz_1dz_2.
\end{align*}
Note that $r_1, r_2\lesssim1$  and such that $|F|$ and $|\Phi_\pm|$ has the bound $\lesssim1,$ then  
\begin{align*}
\Big|\mathcal{K}_1^{\pm}(x, y)\Big|\lesssim\int_{\R^{2n}}\frac {v_{12}\widetilde{\Gamma}(z_1, z_2)}{r_1^{n-2}}\chi_{r_1, r_2\lesssim1}\int_0^1 \eta^{n-1}d\eta dz_1dz_2.
\end{align*}
Furthermore, by  \eqref {Gamma2} and $|V(x)|\lesssim\langle x \rangle^{- n-},$ it follows that 
\begin{align*}
	\Big|\mathcal{K}_1^{\pm}(x, y)\Big|\lesssim\int_{\R^{2n}}\frac {\langle z_1\rangle^{-n-}\langle z_2\rangle^{-n-}}{r_1^{n-2}}\chi_{r_1, r_2\lesssim1}dz_1dz_2.
\end{align*} 
Since $r_1, r_2\lesssim1,$  it is easy to check that
 the bound above satisfies \eqref{Schur test}, (i.e. $K_1$ is admissible). 
 \vskip0.2cm
 $\bullet$\underline{ Next, we focus on the operator $K_2$ } with kernel $K_2(x,y)=\mathcal{K}^+_2(x,y)+\mathcal{K}^-_2(x,y),$ where
\begin{align*}
	\mathcal{K}_2^{\pm}(x, y)=\frac{1}{\pi i}\int_{\R^{2n}}\frac {v_{12}\chi_{r_1\approx r_2\gg1}}{r_1^{n-2}}\bigg(\int_0^\infty e^{i\eta(r_1\pm r_2)}\chi(\eta)\frac{\eta^{n-1}}{1+2\eta^2}F(\eta, r_1)\Phi_\pm(\eta r_2)\Gamma(\eta)(z_1, z_2)d\eta\bigg) dz_1dz_2.
\end{align*}
Take the identity $1=\chi\big(\eta |r_1\pm r_2|\big)+\widetilde{\chi}\big(\eta |r_1\pm r_2|\big)$ into the integrand  in  the parenthesis above to decompose the integral of $\eta$ into two terms. For the first term linked with $\chi\big(\eta |r_1\pm r_2|\big)$ , one has
\begin{align}\label{K31}
	\bigg|\int_0^\infty e^{i\eta(r_1\pm r_2)} & \chi(\eta)\chi\big(\eta |r_1\pm r_2|\big)\frac{\eta^{n-1}}{1+2\eta^2}F(\eta, r_1)\Phi_\pm(\eta r_2)\Gamma(\eta)(z_1, z_2)d\eta\bigg|\nonumber\\
	&\ \ \ \ \ \ \ \ \ \  \lesssim\widetilde{\Gamma}(z_1, z_2)\int_0^1 \chi\big(\eta |r_1\pm r_2|\big)\eta^{n-1}
	\big\langle \eta r_1 \big\rangle^{\frac{n-3}{2}} \big\langle \eta r_2 \big\rangle^{\frac{1-n}{2}} d\eta.
\end{align}
Additionally, integrate by parts twice for the second term, then  it  can be written as  
\begin{align}\label{K32}
	&\bigg|\int_0^\infty e^{i\eta(r_1\pm r_2)}
	\frac{1}{|r_1\pm r_2|^2}\partial_\eta^2\bigg( \chi(\eta)\widetilde{\chi}\big(\eta |r_1\pm r_2|\big)\frac{\eta^{n-1}}{1+2\eta^2}F(\eta, r_1)\Phi_\pm(\eta r_2)\Gamma(\eta)(z_1, z_2)\bigg)d\eta\bigg|\nonumber\\
	&\  \lesssim\widetilde{\Gamma}(z_1, z_2)\int_0^1 \big(\eta |r_1\pm r_2|\big)^{-2}\widetilde{\chi}\big(\eta |r_1\pm r_2|\big)\eta^{n-1}
	\big\langle \eta r_1 \big\rangle^{\frac{n-3}{2}} \big\langle \eta r_2 \big\rangle^{\frac{1-n}{2}} d\eta.
\end{align}
Combine with \eqref{K31} and \eqref{K32} to derive that
\begin{align*}
	\Big|\mathcal{K}_2^{\pm}(x, y)\Big|\lesssim\int_{\R^{2n}}\frac {v_{12}
		\widetilde{\Gamma}(z_1, z_2)}{r_1^{n-2}}\chi_{r_1\approx r_2\gg1}
	\int_0^1 \frac{\eta^{n-1}}{\big\langle \eta (r_1\pm r_2) \big\rangle^2}
	\big\langle \eta r_1 \big\rangle^{\frac{n-3}{2}} \big\langle \eta r_2 \big\rangle^{\frac{1-n}{2}} d\eta dz_1dz_2.
\end{align*}
We claim that $\mathcal{K}_2^{\pm}(x, y)$ satisfy \eqref{Schur test}, which yields that $K_2$ is admissible. In fact, noting that $r_1\approx r_2$ such that $\big\langle \eta r_1 \big\rangle\approx\big\langle \eta r_2 \big\rangle,$ it follows that
\begin{align*}
	\int_{\R^n}\Big|\mathcal{K}_2^{\pm}(x, y)\Big|dx \lesssim\int_{\R^{2n}}v_{12}
		\widetilde{\Gamma}(z_1, z_2)	\int_0^1 \eta^{n-1}
\bigg(\int_{\R^n}	\frac{\chi_{r_1\approx r_2\gg1}}{r_1^{n-2}\big\langle \eta r_1 \big\rangle\big\langle \eta (r_1\pm r_2)\big\rangle^2}dx\bigg)
		d\eta dz_1dz_2.
\end{align*}
We first deal with  the  parenthesis above (i.e. the integral about $x$ ).   By the polar coordinate transformation: 
$x - z_1 = r\omega, \ (r, \omega) \in \mathbb{R}^+ \times S^{n-1}$ and the change of variable $\eta r\longmapsto r,$ one has
\begin{align*}
	\int_{\R^n}	\frac{\chi_{r_1\approx r_2\gg1}}{r_1^{n-2}\big\langle \eta r_1 \big\rangle\big\langle \eta (r_1\pm r_2)\big\rangle^2}dx\lesssim \eta^{-2}\int_0^\infty\frac{r}{\langle  r \rangle\big\langle r\pm \eta r_2\big\rangle^2}dr\lesssim\eta^{-2}.
\end{align*}
Hence, taking into account  \eqref {Gamma2} and $|V(x)|\lesssim\langle x \rangle^{- n-},$ we conclude that 
\begin{align*}
	\int_{\R^n}\Big|\mathcal{K}_2^{\pm}(x, y)\Big|dx \lesssim
	\int_{\R^{2n}}\langle z_1\rangle^{-n-}\langle z_2\rangle^{-n-}\Big(\int_0^1 \eta^{n-3}
	d\eta \Big)dz_1dz_2\lesssim1,
\end{align*}
uniformly for $y.$ By symmetry, we also get the  $y$-integral is bounded
	uniformly for $x.$
	\vskip0.2cm
 $\bullet$	\underline{ Now, we start to analyze  the operator $K_3$} with  kernel $K_3(x,y)=\mathcal{K}^+_3(x,y)+\mathcal{K}^-_3(x,y),$ where
	\begin{align*}
		\mathcal{K}_3^{\pm}(x, y)=\frac{1}{\pi i}\int_{\R^{2n}}\frac {v_{12}\chi_{r_1\gg \langle r_2\rangle}}{r_1^{n-2}}\bigg(\int_0^\infty e^{i\eta(r_1\pm r_2)}\chi(\eta)\frac{\eta^{n-1}}{1+2\eta^2}F(\eta, r_1)\Phi_\pm(\eta r_2)\Gamma(\eta)(z_1, z_2)d\eta\bigg) dz_1dz_2.
	\end{align*}
	Insert the identity  $1=\chi(\eta r_1)+\widetilde{\chi}(\eta r_1)$ into the integrand above to decompose $\mathcal{K}_3^{\pm}(x, y)$ into the sum of $\mathcal{K}_{31}^{\pm}(x, y)$ and $\mathcal{K}_{32}^{\pm}(x, y)$. 
	
	 For $\mathcal{K}_{31}^{\pm}(x, y)$ corresponding to $\chi(\eta r_1)$,  note that
 $\eta r_2\ll\eta r_1\lesssim1$ such that $|F|,$ $|\Phi_\pm|\lesssim1,$ then
	\begin{align*}
\Big|\mathcal{K}_{31}^{\pm}(x, y) \Big|	&\lesssim \int_{\R^{2n}}\frac {v(z_1)v(z_2)\widetilde{\Gamma}(z_1, z_2)}{r_1^{n-2}}\chi_{r_1\gg \langle r_2\rangle}\Big(\int_0^{r_1^{-1}}\eta^{n-1}d\eta\Big) dz_1dz_2\\
&\lesssim\int_{\R^{2n}}
	\frac{\langle z_1\rangle^{-n-}\langle z_2\rangle^{-n-}\chi_{r_1\gg\langle r_2 \rangle}}{r_1^{2n-2}}dz_1dz_2.
\end{align*}
By Lemma \ref{Appendix 3} with $k=2n-2$ and  $\alpha=0$, one has  $\mathcal{K}_{31}$ is admissible. 

As for $\mathcal{K}_{32}^{\pm}(x, y)$ corresponding to $\widetilde{\chi}(\eta r_1)$,  integrating by parts $N=\lceil\frac{n}{2}\rceil+1$ times in $\eta$, note that   $\big\langle \eta r_2 \big\rangle^{\frac{1-n}{2}}\leq1,$ and $\eta r_1\gtrsim1$ such that $\big\langle \eta r_1 \big\rangle \approx \eta r_1,$  then
	\begin{align*}
\Big|\mathcal{K}_{32}^{\pm}(x, y)\Big|&\lesssim \int_{\R^{2n}}\frac {v_{12}\widetilde{\Gamma}(z_1, z_2)}{r_1^{n-2+N}}\chi_{r_1\gg \langle r_2\rangle}\bigg(\int_{r_1^{-1}}^1 \eta^{n-1-N}	\big( \eta r_1 \big)^{\frac{n-3}{2}} \big\langle \eta r_2 \big\rangle^{\frac{1-n}{2}} d\eta\bigg) dz_1dz_2\\
&\lesssim \int_{\R^{2n}}\frac {v_{12}\widetilde{\Gamma}(z_1, z_2)}
{r_1^{\frac{n-1}{2}+N}}\chi_{r_1\gg \langle r_2\rangle}\bigg(\int_{r_1^{-1}}^1 \eta^{n-1-N
	+\frac{n-3}{2}}d\eta\bigg) dz_1dz_2.
\end{align*}
Moreover, observe that $\frac{n-1}{2}+N=n+1,$ $n-1-N
+\frac{n-3}{2}=n-4$ if $n$ is odd and $\frac{n-1}{2}+N=n+\frac{1}{2},$ $n-1-N
+\frac{n-3}{2}=n-\frac{7}{2}$ if $n$ is even, then we conclude that 
	\begin{align*}
\Big|\mathcal{K}_{32}^{\pm}(x, y)\Big|
		\lesssim
		\int_{\R^{2n}}
		\frac{\langle z_1\rangle^{-n-}\langle z_2\rangle^{-n-}\chi_{r_1\gg\langle r_2 \rangle}}{r_1^{n+\frac{1}{2}}}dz_1dz_2.
\end{align*}
By Lemma \ref{Appendix 3} with $k=n+\frac{1}{2}$ and  $\alpha=0$, one has  $\mathcal{K}_{32}$ is admissible. Thus, we obtain $K_3$ is admissible.
\vskip0.2cm
 $\bullet$\underline{ Finally, we aim to establish  that $K_4$ is admissible} with the following kernel $K_4(x, y)$:
\begin{align}\label{K4(x,y)} 
\frac{1}{\pi i}\int_{\R^{2n}}v_{12}
	\chi_{r_2\gg\langle  r_1  \rangle}
	&\int_0^\infty\eta(2\eta^2+1)\chi(\eta)R_0^+(x, z_1)\Gamma(\eta)(z_1, z_2) \big(R_0^+-R_0^-\big)(z_2, y)d\eta dz_1dz_2,
\end{align}
where  $R_0^\pm:=R_0^\pm(\eta^4+\eta^2)$ for short.
We emphasize that  showing $K_4$ admissible is   the most challenging part of the analysis.
Denote
 $$B(\eta)(|x-y|):=R_0^+(\eta^4+\eta^2)(x, y)-R_0(0)(x, y),\ \ \ 
\mathcal{E}(\eta)(x, y):=
\Gamma(\eta)(x, y)-\Gamma(0)(x, y).$$
Note that $R_0(0)=(-\Delta)^{-1}-(-\Delta+1)^{-1},$ then by  \eqref{2-odd}-\eqref{2-even>},  its kernel  $R_0(0)(x, y)$ satisfies $\big|R_0(0)(x, y)\big|\lesssim |x-y|^{-n+2}.$
 Utilize  the mean value theorem and Lemma \ref{lemma Rj} to get for  $\ell\in\N^+\cup\{0\}, $ 
 	\begin{align}\label{B(eta)}
 	\sup_{0<\eta\ll1}
 	\Big|\partial_\eta^{\ell}B(\eta)(|x-y|) \Big|\lesssim
 	\begin{cases}
 	\eta^{1-\ell}	\big(|x-y|^{-n+2}+|x-y|^{-n+3} \big)& {\rm if}\ \eta |x-y|\ll1,\\ 
 	\eta^{1-\ell} |x-y|^{-\frac{n-1}{2}+\ell} & {\rm if}\ \eta |x-y|\gtrsim1.
 	\end{cases}
 \end{align}
 Going on to use the mean value theorem and combining with the definition $\widetilde{\Gamma}(x, y)$ given by \eqref{Gamma2},  we derive that
 \begin{align}\label{partial-E}
 	\sup_{0<\eta\ll1}
 	\Big|\partial_\eta^{\ell}\mathcal{E}(\eta)(x, y) \Big|\lesssim
 		\eta^{1-\ell}\widetilde{\Gamma}(x, y),\ \ \ \ell=0,1,\cdots,\Big\lceil\frac{n}{2}\Big\rceil +1.
 \end{align}
 
 Note that the identity:
 \begin{align*}
& R_0^+(\eta^4 + \eta^2)(x, z_1) \Gamma(\eta)(z_1, z_2)
 \\
&=R_0(0)(x, z_1) \Gamma(0)(z_1, z_2) + R_0^+(\eta^4 + \eta^2)(x, z_1) \mathcal{E}(\eta)(z_1, z_2) + B(\eta)(r_1) \Gamma(0)(z_1, z_2),
\end{align*}
 and insert it into the integrand of \eqref{K4(x,y)} to decompose $K_4(x, y)$ into the sum of $ K_{41}(x, y),$  $K_{42}(x, y)$, and $ K_{43}(x, y),$ corresponding to the terms $ R_0(0)(x, z_1) \Gamma(0)(z_1, z_2)$,  $ R_0^+(\eta^4 + \eta^2)(x, z_1) \mathcal{E}(\eta)(z_1, z_2)$, and $ B(\eta)(r_1) \Gamma(0)(z_1, z_2)$, respectively. 
 
 Hence,   it suffices to show that  each operator $K_{4j}$ with kernel  $ K_{4j}(x, y)$ $(j=1,2,3)$ is admissible.
 
 ${\rm(i)}$  \underline{ First, we take into account  the operator $ K_{41}$}. According to the splitting identity \eqref{split}, its kernel
 $ K_{41}(x, y)$ can be expressed as 
 \begin{align*}
 \frac{1}{\pi i}\int_{\R^{2n}}v_{12}
 	\chi_{r_2\gg\langle  r_1  \rangle} \Gamma(0)(z_1, z_2)R_0(0)(x, z_1)
 	\bigg[\int_0^\infty\eta\chi(\eta)\Big(R_\Delta^+(\eta^2)
 	-R_\Delta^-(\eta^2)\Big)(z_2, y)d\eta\bigg] dz_1dz_2.
 \end{align*}
 By  functional calculus,  it follows that
 \begin{align*}
	\frac{1}{\pi i}
	\int_0^\infty\eta\chi(\eta)\Big(R_\Delta^+(\eta^2)
	-R_\Delta^-(\eta^2)\Big)d\eta =\chi\big(\sqrt{-\Delta}\big),
\end{align*}
which gives that
\begin{align*}
	\frac{1}{\pi i}
	\int_0^\infty\eta\chi(\eta)\Big(R_\Delta^+(\eta^2)
	-R_\Delta^-(\eta^2)\Big)(z_2, y)d\eta =\frac{1}{(2\pi)^{\frac{n}{2}}}\Big[\mathcal{F}^{-1}\Big(\chi\big(|\cdot|\big)\Big)\Big](y-z_2).
\end{align*}
Since $\chi(|x|)\in \mathcal{S}(\R^n),$ then $\Big[\mathcal{F}^{-1}\Big(\chi\big(|\cdot|\big)\Big)\Big](y-z_2)=O\big(\big\langle y-z_2 \big\rangle^{-L}\big)$ for any $L\in\N^+.$ Hence,  choosing $L=n+1$ and combining with $\big|R_0(0)(x, z_1)\big|\lesssim r_1^{-n+2},$  we have 
\begin{align*}
K_{41}(x, y)\lesssim\int_{\R^{2n}}\frac {v_{12}\widetilde{\Gamma}(z_1, z_2)}
{r_1^{n-2}r_2^{n+1}}
	\chi_{r_2\gg\langle  r_1  \rangle}  dz_1dz_2
	\lesssim	\int_{\R^{2n}}
	\frac{\langle z_1\rangle^{-n-}\langle z_2\rangle^{-n-}\chi_{r_2\gg\langle r_1 \rangle}}{r_1^{n-2}r_2^{n+1}}dz_1dz_2.
\end{align*}
Utilize Lemma \ref{Appendix 3} with $k=n+1$ and $\alpha=n-2$ to obtain that $K_{41}$ is admissible.

${\rm(ii)}$  \underline{ Next, we analyze the operator $K_{42}$.} Its kernel $K_{42}(x,y)=\mathcal{K}^+_{42}(x,y)+\mathcal{K}^-_{42}(x,y),$ where
\begin{align*}
	\mathcal{K}_{42}^{\pm}(x, y)=\frac{1}{\pi i}\int_{\R^{2n}}\frac {v_{12}\chi_{r_2\gg \langle r_1\rangle}}{r_1^{n-2}}\bigg(\int_0^\infty e^{i\eta(r_1\pm r_2)}\chi(\eta)\frac{\eta^{n-1}}{1+2\eta^2}F(\eta, r_1)\Phi_\pm(\eta r_2)\mathcal{E}(\eta)(z_1, z_2)d\eta\bigg) dz_1dz_2.
\end{align*}
	Insert the identity  $1=\chi(\eta r_2)+\widetilde{\chi}(\eta r_2)$ into the integrand above to decompose $\mathcal{K}_{42}^{\pm}(x, y)$ into the sum of $\mathcal{K}_{421}^{\pm}(x, y)$ and $\mathcal{K}_{422}^{\pm}(x, y)$ corresponding to  $\chi(\eta r_2)$ and $\widetilde{\chi}(\eta r_2),$ respectively. 
	
For $\mathcal{K}_{421}^{\pm}(x, y),$ 	since $\eta r_1\ll\eta r_2\lesssim1$ such that $|F|,$ $|\Phi_\pm|\lesssim1,$ 
and  combining with  \eqref{partial-E},  then we obtain  that $	\Big|\mathcal{K}_{421}^{\pm}(x, y) \Big|$ satisfies 
	\begin{align*}
		&\lesssim \int_{\R^{2n}}\frac {v(z_1)v(z_2)\widetilde{\Gamma}(z_1, z_2)}{r_1^{n-2}}\chi_{r_2\gg \langle r_1\rangle}\Big(\int_0^{r_2^{-1}}\eta^n d\eta\Big) dz_1dz_2\lesssim\int_{\R^{2n}}
		\frac{\langle z_1\rangle^{-n-}\langle z_2\rangle^{-n-}\chi_{r_2\gg\langle r_1 \rangle}}{r_1^{n-2}r_2^{n+1}}dz_1dz_2.
	\end{align*}
	By Lemma \ref{Appendix 3} with $k=n+1$ and  $\alpha=n-2$, one has  $\mathcal{K}_{421}^\pm(x, y)$ satisfy \eqref{Schur test}.
	
	As for $\mathcal{K}_{422}^{\pm}(x, y)$, integrating  by parts $N=\lceil\frac{n}{2}\rceil+1$ times in $\eta$, note that \eqref{partial-E} and $\eta r_2\gtrsim1$ such that $\big\langle \eta r_2 \big\rangle \approx \eta r_2,$  then
	\begin{align*}
		\Big|\mathcal{K}_{422}^{\pm}(x, y)\Big|&\lesssim \int_{\R^{2n}} v_{12}\widetilde{\Gamma}(z_1, z_2)\chi_{r_2\gg \langle r_1\rangle}\bigg(\frac{1}{r_1^{n-2}r_2^N}\int_{r_2^{-1}}^1 \eta^{n-N}	\big\langle \eta r_1 \big\rangle^{\frac{n-3}{2}} \big(\eta r_2 \big)^{\frac{1-n}{2}} d\eta\bigg) dz_1dz_2.
	\end{align*}
	We first estimate the bound in the parenthesis by inserting $\big\langle \eta r_1 \big\rangle^{\frac{n-3}{2}} \lesssim1+\big( \eta r_1 \big)^{\frac{n-3}{2}}.$ Specifically,   
		\begin{align*}
		\frac{1}{r_1^{n-2}r_2^N}\int_{r_2^{-1}}^1 \eta^{n-N}	\big\langle \eta r_1 \big\rangle^{\frac{n-3}{2}} \big(\eta r_2 \big)^{\frac{1-n}{2}} d\eta&\lesssim
		\frac{1}{r_1^{n-2}r_2^{N+\frac{n-1}{2}}}\int_{r_2^{-1}}^1 \eta^{\frac{n+1}{2}-N} d\eta+
		\frac{1}{r_1^{\frac{n-1}{2}}r_2^{N+\frac{n-1}{2}}}\int_{r_2^{-1}}^1 \eta^{n-N-1} d\eta\\
		&\lesssim
		\frac{1}{r_1^{n-2}r_2^{n+\frac{1}{2}}}+	\frac{1}{r_1^{\frac{n-1}{2}}
			r_2^{n+\frac{1}{2}}},
	\end{align*}
	the last inequality is established by observing that $\frac{n+1}{2}-N=-1,$ $N+\frac{n-1}{2}=n+1,$  $n-N-1
	=\frac{n}{2}-\frac{5}{2}$  if $n$ is odd and  $\frac{n+1}{2}-N=-\frac{1}{2},$ $N+\frac{n-1}{2}=n+\frac{1}{2},$
	$n-N-1
	=\frac{n}{2}-2$ if $n$ is even. 
	Furthermore,  combining with \eqref{Gamma2}, we conclude that 
	\begin{align*}
		\Big|\mathcal{K}_{422}^{\pm}(x, y)\Big|
			&\lesssim \int_{\R^{2n}}
			\langle z_1\rangle^{-n-}\langle z_2\rangle^{-n-}\chi_{r_2\gg\langle r_1 \rangle}\bigg(	\frac{1}{r_1^{n-2}r_2^{n+\frac{1}{2}}}+	\frac{1}{r_1^{\frac{n-1}{2}}
				r_2^{n+\frac{1}{2}}}\bigg)dz_1dz_2.
	\end{align*}
	By Lemma \ref{Appendix 3} with $(k, \alpha)=(n+\frac{1}{2},  n-2)$ and  $(k, \alpha)=(n+\frac{1}{2}, \frac{n-1}{2}),$  it follows that  $\mathcal{K}_{422}^\pm(x, y)$ satisfy \eqref{Schur test}. Hence  the operator $K_{42}$ is admissible.
	
${\rm(iii)}$ 	\underline{ It remains to deal with  the operator $K_{43}.$} Observe that its kernel $K_{43}(x,y)=\mathcal{K}^+_{43}(x,y)+\mathcal{K}^-_{43}(x,y),$ where
	\begin{align*}
		\mathcal{K}_{43}^{\pm}(x, y)=\frac{1}{\pi i}\int_{\R^{2n}} v_{12}\chi_{r_2\gg \langle r_1\rangle}\Gamma(0)(z_1, z_2)\bigg(\int_0^\infty 
		e^{\pm i\eta r_2}\chi(\eta)\eta^{n-1}B(\eta) (r_1)\Phi_\pm(\eta r_2)d\eta\bigg) dz_1dz_2.
	\end{align*}
	Insert the identity $1=\chi(\eta r_2)+\widetilde{\chi}(\eta r_2)\chi(\eta r_1)+\widetilde{\chi}(\eta r_1)$ into the integrand above to decompose $\mathcal{K}_{43}^{\pm}(x, y)$ into the sum of $\mathcal{K}_{431}^{\pm}(x, y),$ $\mathcal{K}_{432}^{\pm}(x, y)$ and 
	$\mathcal{K}_{433}^{\pm}(x, y)$
	 corresponding to  $\chi(\eta r_2),$ $  \widetilde{\chi}(\eta r_2)\chi(\eta r_1)$ and $\widetilde{\chi}(\eta r_1),$  respectively. 
	 
	$\bullet$ For $\mathcal{K}_{431}^{\pm}(x, y),$ 
	since $\eta r_1\ll\eta r_2\lesssim1$ then by \eqref{B(eta)},  \eqref{Gamma2} and $\big|\Phi_\pm(\eta r_2)\big|\lesssim1,$  we derive 
	\begin{align*}
	\Big|\mathcal{K}_{431}^{\pm}(x, y)\Big|&\lesssim	\int_{\R^{2n}} v(z_1)v(z_2)\chi_{r_2\gg \langle r_1\rangle}
	\widetilde{\Gamma}(z_1, z_2)\bigg(\int_0^{r_2^{-1}} 
		\eta^n\Big(r_1^{-n+2}+r_1^{-n+3}\Big)d\eta\bigg) dz_1dz_2\\
		&\lesssim \int_{\R^{2n}}
		\langle z_1\rangle^{-n-}\langle z_2\rangle^{-n-}\chi_{r_2\gg\langle r_1 \rangle}\bigg(	\frac{1}{r_1^{n-2}r_2^{n+1}}+	\frac{1}{r_1^{n-3}
			r_2^{n+1}}\bigg)dz_1dz_2.
	\end{align*}
Utilize Lemma \ref{Appendix 3} with $(k, \alpha)=(n+1,  n-2)$ and  $(k, \alpha)=(n+1, n-3)$   to obtain that $\mathcal{K}_{431}^\pm(x, y)$ satisfy \eqref{Schur test}.

$\bullet$ For $\mathcal{K}_{432}^{\pm}(x, y)$, integrating by parts $N=\lceil\frac{n}{2}\rceil+1$ times in $\eta$ and setting $\Psi:=\Psi(\eta, r_1, r_2)=\chi(\eta)\chi(\eta r_1)\widetilde{\chi}(\eta r_2)$ for short,  it follows that
$\big|\mathcal{K}_{432}^{\pm}(x, y)\big|$ is dominated by
\begin{align*}
	\lesssim\bigg| \frac{1}{\pi i}\int_{\R^{2n}}\frac {v_{12}\chi_{r_2\gg \langle r_1\rangle}}{r_2^N}\Gamma(0)(z_1, z_2)\bigg[\int_0^\infty e^{\pm i\eta r_2}	\partial_\eta^N\bigg(\Psi \eta^{n-1}B(\eta) (r_1)\Phi_\pm(\eta r_2)\bigg)d\eta\bigg] dz_1dz_2\bigg|.
\end{align*}
Using \eqref{B(eta)} for $\eta r_1\lesssim1$ and  considering  $\eta r_2\gtrsim1$ such that $\big\langle \eta r_2 \big\rangle \approx \eta r_2,$ 
we get
	\begin{align*}
	\Big|\mathcal{K}_{432}^{\pm}(x, y)\Big|
	&\lesssim\int_{\R^{2n}}\frac{ v_{12}\widetilde{\Gamma}(z_1, z_2)}{r_2^{N+\frac{n-1}{2}}
		}\big(r_1^{-n+2}+r_1^{-n+3}\big)	\chi_{r_2\gg \langle r_1\rangle}
	\Big(\int_{r_2^{-1}}^1 \eta^{\frac{n+1}{2}-N}\Big)
	dz_1dz_2.
\end{align*}
Moreover,	observe  that $\frac{n+1}{2}-N=-1,$ $N+\frac{n-1}{2}=n+1,$   if $n$ is odd and  $\frac{n+1}{2}-N=-\frac{1}{2},$ $N+\frac{n-1}{2}=n+\frac{1}{2},$  if $n$ is even. Then 
	\begin{align*}
\Big|\mathcal{K}_{432}^{\pm}(x, y)\Big| \lesssim \int_{\R^{2n}}
	\langle z_1\rangle^{-n-}\langle z_2\rangle^{-n-}\chi_{r_2\gg\langle r_1 \rangle}\bigg(	\frac{1}{r_1^{n-2}r_2^{n+\frac{1}{2}}}+	\frac{1}{r_1^{n-3}
		r_2^{n+\frac{1}{2}}}\bigg)dz_1dz_2.
\end{align*}
	By Lemma \ref{Appendix 3} with $(k, \alpha)=(n+\frac{1}{2},  n-2)$ and  $(k, \alpha)=(n+\frac{1}{2}, n-3),$  we derive that  $\mathcal{K}_{432}^\pm(x, y)$ satisfy \eqref{Schur test}. 
	  
$\bullet$ As for $\mathcal{K}_{433}^{\pm}(x, y),$ note that  $\mathcal{K}_{433}^{\pm}(x, y)$  is written as
\begin{align*}
	 \frac{1}{\pi i}\int_{\R^{2n}} v_{12}\chi_{r_2\gg \langle r_1\rangle}\Gamma(0)(z_1, z_2)\bigg(\int_0^\infty 
	e^{\pm i\eta r_2}\chi(\eta)\widetilde{\chi}(\eta r_1)\eta^{n-1}B(\eta) (r_1)\Phi_\pm(\eta r_2)d\eta\bigg) dz_1dz_2.
\end{align*}
Set
$$\Lambda^{\pm}(x, y, z_1, z_2):=\int_0^\infty 
e^{\pm i\eta r_2}\chi(\eta)\widetilde{\chi}(\eta r_1)\eta^{n-1}B(\eta) (r_1)\Phi_\pm(\eta r_2)d\eta.$$
Next, we first to deal with $\Lambda^{\pm}(x, y, z_1, z_2).$ 
Since  $ B(\eta)(r_1):=R_0^+(\eta^4+\eta^2)(x, z_1)-R_0(0)(x, z_1),$ then
\begin{align*}
\Big|  \Lambda^{\pm}(x, y, z_1, z_2)  \Big|	\lesssim&
\frac{1}{r_1^{n-2}}\bigg|\int_0^\infty 
	e^{i\eta(r_1\pm r_2)}\chi(\eta)\widetilde{\chi}(\eta r_1)\frac{\eta^{n-1}}{1+2\eta^2}F(\eta, r_1)\Phi_\pm(\eta r_2)d\eta\bigg|\\
	&+\bigg |R_0(0)(x, z_1)\int_0^\infty 
	e^{\pm i\eta r_2}\chi(\eta)\widetilde{\chi}(\eta r_1)\eta^{n-1}\Phi_\pm(\eta r_2)d\eta\bigg|.
\end{align*}
Furthermore,   note that $\big|R_0(0)(x, z_1)\big|\lesssim r_1^{-n+2}.$   Integrate by parts $N=\lceil\frac{n}{2}\rceil+1$ times in $\eta$ for the two integrals above  and combine with $r_2\gg \langle r_1\rangle $  such that $|r_1\pm r_2|\approx r_2$ to derive that
\begin{align*}
	\Big|  \Lambda^{\pm}&(x, y, z_1, z_2)  \Big|	
\lesssim
\frac{1}{r_1^{n-2}r_2^N}\bigg( \int_{r_1^{-1}}^1
\eta^{n-1-N}(\eta r_1)^{\frac{n-3}{2}}(\eta r_2)^{\frac{1-n}{2}}d\eta
+\int_{r_1^{-1}}^1
\eta^{n-1-N}(\eta r_2)^{\frac{1-n}{2}}d\eta\bigg)\\
&\lesssim\frac{1}{r_1^{\frac{n-1}{2}}r_2^{N+\frac{n-1}{2}}}
\int_{r_1^{-1}}^1
\eta^{n-N-2}d\eta+\frac{1}{r_1^{n-2}r_2^{N+\frac{n-1}{2}}}
\int_{r_1^{-1}}^1
\eta^{\frac{n-1}{2}-N}d\eta\lesssim\frac{1}{r_1^{\frac{n-1}{2}}r_2^{n+\frac{1}{2}}}+\frac{1}{r_1^{n-3}r_2^{n+\frac{1}{2}}},
\end{align*}
the last inequality is followed by noting that  $n-N-2=\frac{n-1}{2}-3$, $N+\frac{n-1}{2}=n+1,$ $\frac{n-1}{2}-N=-2,$  if $n$ is odd and  $n-N-2=\frac{n}{2}-3,$ $N+\frac{n-1}{2}=n+\frac{1}{2},$   $\frac{n-1}{2}-N=-\frac{3}{2},$ if $n$ is even. 
Hence, we conclude that 
	\begin{align*}
	\Big|\mathcal{K}_{433}^{\pm}(x, y)\Big|  \lesssim \int_{\R^{2n}}
	\langle z_1\rangle^{-n-}\langle z_2\rangle^{-n-}\chi_{r_2\gg\langle r_1 \rangle}\bigg(\frac{1}{r_1^{\frac{n-1}{2}}r_2^{n+\frac{1}{2}}}+\frac{1}{r_1^{n-3}r_2^{n+\frac{1}{2}}}\bigg)dz_1dz_2.
\end{align*}
By Lemma \ref{Appendix 3} with $(k, \alpha)=(n+\frac{1}{2},  \frac{n-1}{2})$ and  $(k, \alpha)=(n+\frac{1}{2}, n-3),$  we get  that  $\mathcal{K}_{433}^\pm(x, y)$ satisfy \eqref{Schur test}.   Thus  $K_{43}$ is admissible. Thus we complete the proof. 
 \end{proof}		    
 		    
Putting  \eqref{WL2},  \eqref{Gamma} and  Proposition \ref{proposition-WL} together, we have finished the proof of Theorem \ref{theorem WL}.    
 			\section{$L^p$ boundedness of the high energy part $\Omega_H^\pm$}\label{section4}
In this section,  we are devoted to  establishing   the $L^p$ boundedness of the  high  energy part $\Omega_H^\pm$ defined by  \eqref{WH}   for all $1\leq p\leq\infty$ when dimensions $n\geq5.$ 		    
 		     \begin{align}\label{WH}
 		    	\Omega_H^\pm:=\frac{1}{\pi i}\int_0^\infty\eta(2\eta^2+1)\widetilde{\chi}(\eta)(R_0^\mp V)^kR_V^\mp(VR_0^\mp)^kV\big(R_0^+-R_0^-\big)d\eta.
 		    \end{align}
 		    Here, $R_0^\mp:=R_0^\mp(\eta^4+\eta^2)$  and $R_V^\mp:=R_V^\mp(\eta^4+\eta^2)$  for short.
 		    
 		   As the same before, because of $\Omega_H^+f = \overline{\Omega_H^-\bar{f}}$, the following analysis deal with only  $\Omega_H^-$.
 		    
 	\begin{theorem}\label{theorem WH}
 		Let  $n\geq5$ and $V(x)\lesssim\langle x \rangle^{-n-5-}$. Then provided $k$ enough large depending on
 		$n$, the operator  $\Omega_H^-$ defined by \eqref{WH} 
 		extends to a bounded operator on $L^p(\R^n)$ for all $1\leq p\leq\infty.$
 	\end{theorem} 
 	\subsection{The higher energy analyses of $R_0^\pm$ and $R_V^\pm$}
 		Before presenting the proof of Theorem~\ref{theorem WH}, we establish several estimates for  $R_0^\pm(\eta^4+\eta^2)$ and $R_V^\pm(\eta^4+\eta^2)$ in the regime $\eta \gtrsim 1$  by the following two Lemmas \ref{lemmaRRV} and \ref{lemma_high_4}.

 	\begin{lemma}\label{lemmaRRV}
 	Let $r=|x-y|$  and $\eta\gtrsim1.$   Then for  $\ell\in\N^+\cup\{0\} $,
 		\begin{align}\label{RR1}
 			\Big|\partial_\eta^{\ell}R_0^+(\eta^4+\eta^2)(x,y) \Big|
 		\lesssim
 		\begin{cases}
 			\eta^{\frac{n+1}{2}-4}	\big(r^{-\frac{n-1}{2}+\ell}+r^{-\frac{n-1}{2}}\big)
 	 & {\rm if}\ 5\leq n\leq7,\\ 
 				\eta^{\frac{n+1}{2}-4}	\big(r^{-\frac{n-1}{2}+\ell}+r^{-n+4}\big)& {\rm if}\ n>7.
 		\end{cases}
 	\end{align}
 	Moreover,  denote $\mathcal{G}^\pm(\eta)(x, y):=	e^{\mp i\eta|x|} R_0^\pm(\eta^4+\eta^2)(x, y).$ Then
 	\begin{align}\label{GG1}
 		\Big|\partial_\eta^{\ell}\mathcal{G}^\pm(\eta)(x, y)\Big|
 		\lesssim
 		\begin{cases}
 		\langle y \rangle^{\ell}\	\eta^{\frac{n+1}{2}-4}	r^{-\frac{n-1}{2}} & {\rm if}\ 5\leq n\leq7,\\ 
 		\langle y \rangle^{\ell}\	\eta^{\frac{n+1}{2}-4}	\big(r^{-\frac{n-1}{2}}+r^{-n+4}\big)& {\rm if}\ n>7.
 		\end{cases}
 	\end{align}
 		\end{lemma}
 		\begin{proof}
 			First, we consider the case $n\geq5$ is odd.  By  \eqref{4-odd}, $R_0^\pm(\eta^4+\eta^2)(x,y)$ can be written as 
 			\begin{align}\label{>>oddR}
 			\frac{\eta^{\frac{n+1}{2}-4}}{r^{\frac{n-1}{2}}}
 			\Bigg(\frac{\eta^2}{(1+2\eta^2)}(\eta r)^{-\frac{n-3}{2}}\bigg(e^{\pm i\eta r}\sum_{j=0}^
 			{\frac{n-3}{2}}c_j\big(\pm\eta r\big)^j-e^{ -r\sqrt{1+\eta^2}}\sum_{j=0}^
 			{\frac{n-3}{2}}c_j\big(ir\sqrt{1+\eta^2}\big)^j\bigg)\Bigg).
 			\end{align}
 			For $\eta r\gtrsim1$ such that $\eta^{-1}\lesssim r,$
 		applying the expression  \eqref{>>oddR}, we can check that  for  $\ell\in\N^+\cup\{0\} $,
 			\begin{align}\label{>>oddR1}
 			\Big|\partial_\eta^{\ell}R_0^\pm(\eta^4+\eta^2)(x,y) \Big|
 			\lesssim	\eta^{\frac{n+1}{2}-4}r^{-\frac{n-1}{2}+\ell}.
 			\end{align}
 	 When $\eta r\ll1$ (noting that $\eta r\ll1$ such that $r\ll1$ since $\eta\gtrsim1$  ),  it follows that  (see {\it e.g.} \cite{Jensen-80})
 	\begin{align*}
 		R_\Delta^\pm(\eta^2)(x,y)=r^{-n+2}\sum_{j=0}^
 		{\infty}\kappa_j\big(\pm i\eta r\big)^j, \ \ \ 
 		R_\Delta^\pm(-1-\eta^2)(x,y)=r^{-n+2}\sum_{j=0}^
 		\infty\kappa_j\big(-r\sqrt{1+\eta^2}\big)^j,
 	\end{align*}
 		where $\kappa_j$  depending on $n$ can be computed 
 		and  $\kappa_j=0$ if $j$ is odd with $0<j<n-2.$
 		  Hence, by the  splitting identity \eqref{split} and $\kappa_1=0,$ for $\eta r\ll 1$, we get 
 		  	\begin{align}\label{>>oddR2}
 		  	R_0^\pm(\eta^4+\eta^2)(x,y)=r^{-n+4}\sum_{j=2}^{\infty}\Bigg[\kappa_j\left(
 		  	\frac{(\pm i)^j \eta^2}{1+2\eta^2}-\Big(
 		  	\frac{-\sqrt{1+\eta^2}}{\eta}\Big)^{j-2}
 		  	\frac{1+\eta^2}{1+2\eta^2}\right)\big(\eta r\big)^{j-2}\Bigg].
 		  	 \end{align}
 For  $\eta r\ll 1,$	note  that $\eta\gtrsim1$, then by  \eqref{>>oddR2},  it follows that 
 	\begin{align*}
 	\Big|\partial_\eta^{\ell}R_0^\pm(\eta^4+\eta^2)(x,y) \Big|
 	\lesssim	r^{-n+4},\ \ \ \ \ \ell\in\N^+\cup\{0\}.
 \end{align*}
 Thus for $\eta r\ll 1$ and $\eta\gtrsim1$, combining with  $n>7$ such that $\frac{n+1}{2}-4>0$ and $n\leq7$  such that $(\eta r)^{\frac{n+1}{2}-4}\gtrsim1$,  we derive  that for $\ell\in\N^+\cup\{0\} $,
 	\begin{align}\label{oddRR3}
 	\Big|\partial_\eta^{\ell}R_0^+(\eta^4+\eta^2)(x,y) \Big|
 	\lesssim
 	\begin{cases}
 		\eta^{\frac{n+1}{2}-4}\	r^{-\frac{n-1}{2}}
 		& {\rm if}\ 5\leq n\leq7,\\ 
 		\eta^{\frac{n+1}{2}-4}\	r^{-n+4}& {\rm if}\ n>7.
 	\end{cases}
 \end{align}
 
As for  $\mathcal{G}^\pm(\eta)( x, y)$,  by \eqref{>>oddR}, it can be expressed as 
\begin{align*}
	\frac{\eta^{\frac{n+1}{2}-4}}{r^{\frac{n-1}{2}}}
	\Bigg(\frac{\eta^2}{(1+2\eta^2)}(\eta r)^{-\frac{n-3}{2}}&e^{\pm i\eta \big(|x-y|-|x|\big)}\bigg(\sum_{j=0}^
	{\frac{n-3}{2}}c_j\big(\pm\eta r\big)^j
	-e^{ -r\sqrt{1+\eta^2}\mp i\eta r}\sum_{j=0}^
	{\frac{n-3}{2}}c_j\big(ir\sqrt{1+\eta^2}\big)^j\bigg)\Bigg).
\end{align*}
When $\eta r\gtrsim1,$ by the expression above and $\eta\gtrsim1,$ we get
	\begin{align}\label{Godd2}
	\Big|\partial_\eta^{\ell}\mathcal{G}^\pm(\eta)(x, y)\Big|
	\lesssim
\langle y \rangle^{\ell}\	\eta^{\frac{n+1}{2}-4}	r^{-\frac{n-1}{2}},\ \ \ \ell\in\N^+\cup\{0\}.
\end{align}
Additionally,  consider that $\eta r\ll1$ with $\eta\gtrsim1$ such that $r\lesssim1$, then by \eqref{oddRR3} and  Leibniz formula, it follows that for $\ell\in\N^+\cup\{0\},$
	\begin{align}\label{Godd3}
\Big|\partial_\eta^{\ell}\mathcal{G}^\pm(\eta)(x, y)\Big|&=	\bigg|\partial_\eta^{\ell}\bigg(e^{\pm i\eta \big(|x-y|-|x|\big)}\Big(e^{\mp i\eta r}R_0^+(\eta^4+\eta^2)(x,y)\Big)\bigg) \bigg|\nonumber
	\\
	&\lesssim
	\begin{cases}
	\langle y \rangle^{\ell}\	\eta^{\frac{n+1}{2}-4}\	r^{-\frac{n-1}{2}}
		& {\rm if}\ 5\leq n\leq7,\\ 
	\langle y \rangle^{\ell}\	\eta^{\frac{n+1}{2}-4}\	r^{-n+4}& {\rm if}\ n>7.
	\end{cases}
\end{align}
Hence, taking into account  \eqref{>>oddR1}, \eqref{oddRR3} and \eqref{Godd2}, \eqref{Godd3},  we derive the desire results.

We deal with  the case $n\geq5$ is even by a similar decomposition based on $ \eta r\gtrsim1$ and $ \eta r\ll1$.

 Next,  we omit details for  simplicity. 
Analogously, 
	for $\eta r\gtrsim1$ and $\ell\in\N^+\cup\{0\},$ by the expression \eqref{4-even>}, we also obtain that the bounds \eqref{>>oddR1} and \eqref{Godd2}.
For  $\eta r\ll1$,  applying the expression \eqref{4-even<}, we also derive the estimates \eqref{oddRR3} and \eqref{Godd3}. Thus we complete the proof. 
 			\end{proof}

 		Before presenting the next lemma, we introduce the notation for the weighted $L^2$ spaces. Specifically, for $\sigma \in \mathbb{R}$, the weighted $L^2$ space is defined as  
 		$$
 		L^{2, \sigma}(\mathbb{R}^n) := \{ f \in L_{\text{loc}}^2(\mathbb{R}^n) \mid \langle \cdot \rangle^\sigma f \in L^2(\mathbb{R}^n) \}.
 		$$	
 	
 	\begin{lemma}[{\cite[Theorem 5.1]{Feng}}]
 		\label{lemma_high_4}
 		Let  $|V(x)|\lesssim \langle x\rangle^{-\beta}$ with $\beta>2+2\ell$ for $\ell=0,1,2,\cdots.$ Assume that $H=\Delta^2-\Delta+V$ has no positive embedded eigenvalues. Then, for any $\sigma>\ell+\frac{1}{2}$, the map $(0,\infty)\ni\eta \mapsto \langle x\rangle^{-\sigma}R_V^\pm(\eta^4+\eta^2) \langle x\rangle^{-\sigma}$ is of $C^\ell$-class in the norm topology on $L^2$ and satisfies
 		$$
 		\big\| \partial_\eta^{\ell} \big(R_V^\pm(\eta^4+\eta^2)\big)\big\|_{L^{2, \sigma}\to L^{2, -\sigma}}\lesssim\eta^{-3},\ \  \ \ \  \eta\gtrsim1.
 		$$
 	\end{lemma}
 	\subsection{The proof of Theorem  \ref{theorem WH}.}
 We start to show that $\Omega_H^-\in\mathbb{B}(L^p)$ for all $1\leq p\leq\infty.$  First, we provide   Lemma \ref{lermmaK} to show that an integral operator is admissible,  which plays an important role in the proof of Theorem  \ref{theorem WH}.
 	\begin{lemma}[{\cite[Lemma 5.2]{Erdogan-Green21}}]
 		\label{lermmaK}
 		Let $K$ be  an integral operator with kernel $K(x, y).$ If there exists $\varepsilon>0$ such that $K(x, y)$ satisfies the  bound:
 		\begin{align*} 
 			\Big| K(x, y)\Big|\lesssim\frac{1}{\big\langle x \big\rangle^{\frac{n-1}{2}}\big\langle y \big\rangle^{\frac{n-1}{2}}\big\langle |x| \pm |y| \big\rangle^{\frac{n+1}{2}+\varepsilon}}.
 			\end{align*}
 			Then $K\in\mathbb{B}\big(L^{p}(\R^n)\big)$ for all $1\leq p\leq\infty.$
 	\end{lemma}
 In what follows, we return to  prove
 	Theorem \ref{theorem WH}.
 	
 	\begin{proof}[Proof of Theorem \ref{theorem WH}]
 		In order to show that the high energy part $\Omega_H^-\in\mathbb{B}\big(L^{p}(\R^n)\big)$ for all $1\leq p\leq\infty,$ it suffices to prove that  its kernel  $\Omega_H^-(x, y)$ satisfies  the bound stated in Lemma \ref{lermmaK}.
 		
	By \eqref{WH}, $\Omega_H^-(x, y)$  can be written as
 		\begin{align}\label{high3}
 			\Omega_H^-(x, y)=&\frac{1}{\pi i}\int_0^\infty \eta(2\eta^2+1)\widetilde{\chi}(\eta)
 			\int_{\R^{n}}\Big(\big(R_0^+ V\big)^{s}\Big)(x, z_1)\times\nonumber\\
 			&\Bigg[\bigg(\big(R_0^+ V\big)^{k-s}R_V^+\big(VR_0^+\big)^{k-s+1}\bigg) \bigg(\big(VR_0^+\big)^{s-1}\Big(V\big(R_0^+-R_0^-\big)\Big)\bigg)\Bigg](z_1, y)dz_1d\eta,
 			\end{align}
 		where $s\in\N^+,$ and  $R_0^\mp:=R_0^\mp(\eta^4+\eta^2),$  $R_V^\mp:=R_V^\mp(\eta^4+\eta^2)$  for short.
 		
 	$\bullet$\underline{	Next,  taking $s=\lfloor \frac{n}{8}\rfloor+1, $  $\beta > \frac{n+1}{2}$ and $2(\beta  - \sigma) > n$ for $\sigma>\frac{1}{2}$}, we claim that   
 		\begin{align}\label{high1}
 		&\big\|\big(R_0^+ V\big)^{s}(x, \cdot)\big\|_{L^{2, \sigma}}\lesssim\eta^{s\big(\frac{n+1}{2}-4\big)} \big\langle x\big\rangle^{-\frac{n-1}{2}},\nonumber\\
 	&\big\|\big(VR_0^+\big)^{s-1}\big(V\big(R_0^+-R_0^-\big)\big)(\cdot, y)	\big\|_{L^{2, \sigma}}\lesssim\eta^{s\big(\frac{n+1}{2}-4\big)}\big\langle y\big\rangle^{-\frac{n-1}{2}}.
 		\end{align}
 	{\rm(i)	} For $5\leq n\leq7$ such that $s=\lfloor \frac{n}{8}\rfloor+1=1, $   by  \eqref{RR1} and Lemma \ref{Appendix 1},   provided $2(\beta-\sigma)>n,$ it follows that 
 			\begin{align*}
 			\left\|\big(R_0^+ V\big)(x, \cdot)\right\|_{L^{2, \sigma}}\lesssim\eta^{\frac{n+1}{2}-4}\bigg(\int_{\R^n}\bigg|
 			\frac{\langle z_1\rangle^{\sigma}V(z_1)}{|x-z_1|^{\frac{n-1}{2}}}\bigg|^2dz_1\bigg)^{\frac{1}{2}}\lesssim\eta^{\frac{n+1}{2}-4} \big\langle x\big\rangle^{-\frac{n-1}{2}}.
 		\end{align*}
 		 	{\rm(ii)	}	For $ n=8$ such that $s=\lfloor \frac{n}{8}\rfloor+1=2, $  utilize  \eqref{RR1} to derive that 
 		\begin{align}\label{n=8R}
 			\left\|\big(R_0^+ V\big)^2(x, \cdot)\right\|_{L^{2, \sigma}}\lesssim\eta^{2\big(\frac{n+1}{2}-4\big)}\bigg(\int_{\R^n}&\bigg|\langle z_1\rangle^{-(\beta-\sigma)}
 			\int_{\R^n}\Big(|x-u|^{-\frac{n-1}{2}}+|x-u|^{-n+4}\Big)\nonumber\\
 			&\times\langle u \rangle^{-\beta}\Big(|u-z_1|^{-\frac{n-1}{2}}+|u-z_1|^{-n+4}\Big)du\bigg|^2dz_1\bigg)^{\frac{1}{2}}.
  		\end{align}
 		Note that $n>7$ such that $n-4>\frac{n-1}{2}.$ Then we can check that 
 			\begin{align*}
 			&\Big(\frac{1}{|x-u|^{\frac{n-1}{2}}}+\frac{1}{|x-u|^{n-4}}\Big)\Big(\frac{1}{|u-z_1|^{\frac{n-1}{2}}}+\frac{1}{|u-z_1|^{n-4}}\Big)\\
 			&\lesssim\frac{1}{|x-z_1|^{\frac{n-1}{2}}}\bigg(\frac{1}{|x-u|^{\frac{n-1}{2}}}+\frac{1}{|x-u|^{n-4}}+\frac{1}{|x-u|^{\frac{3n+1}{2}-8}}+\frac{1}{|u-z_1|^{\frac{n-1}{2}}}+\frac{1}{|u-z_1|^{n-4}}+\frac{1}{|u-z_1|^{\frac{3n+1}{2}-8}}\bigg).
 		\end{align*}
 	Thus,	considering  $n=8$ such that $\frac{3n+1}{2}-8<n$, by \eqref{n=8R} and Lemma \ref{Appendix 1}, one has for  $n=8$,
 			\begin{align*}
 			\left\|\big(R_0^+ V\big)^2(x, \cdot)\right\|_{L^{2, \sigma}}\lesssim\eta^{2\big(\frac{n+1}{2}-4\big)}
 			\bigg(\int_{\R^n}\bigg|
 			\frac{\big\langle z_1\big\rangle^{-(\beta-\sigma)}}{|x-z_1|^{\frac{n-1}{2}}}\bigg|^2dz_1\bigg)^{\frac{1}{2}}\lesssim\eta^{2\big(\frac{n+1}{2}-4\big)}\big\langle x\big\rangle^{-\frac{n-1}{2}}. 
 		\end{align*}
 		{\rm(iii)	} For $n>8,$ 
 		utilize \eqref{RR1} to get   $\left\|\big(R_0^+ V\big)^s(x, \cdot)\right\|_{L^{2, \sigma}}$ is bounded by
 		\begin{align}\label{n>>R}
 		\lesssim \eta^{s\big(\frac{n+1}{2}-4\big)}\Bigg(\int_{\R^n}
 		\bigg| \big\langle u_s \big\rangle^{-(\beta-\sigma)}\int_{\R^{(s-1)n}}
 			\bigg(& \prod_{j=1}^{s-1}\Big(\frac{1}{|u_{j-1}-u_j|^{\frac{n-1}{2}}}+\frac{1}{|u_{j-1}-u_j|^{n-4}}\Big)\langle u_j \rangle^{-\beta}\bigg)\nonumber\\
 		&\ \  \times\Big(\frac{1}{|u_{s-1}-u_s|^{\frac{n-1}{2}}}+\frac{1}{|u_{s-1}-u_s|^{n-4}}\Big)d\vec{u}\bigg|^2\Bigg)^{\frac{1}{2}},
 	\end{align} 
 		where $x=u_0,\  z_1=u_s$  and $\vec{u}=(u_1, u_2, \cdots, u_{s-1}).$
 	For simplicity, set 
 	$$\Pi(u_0, u_j, u_{j+1}):=	\Big(\frac{1}{|u_0-u_j|^{\frac{n-1}{2}}}
 	+\frac{1}{|u_0-u_j|^{n-4a}}\Big)\langle u_j \rangle^{-\beta} \Big(\frac{1}{|u_j-u_{j+1}|^{\frac{n-1}{2}}}+\frac{1}{|u_j-u_{j+1}|^{n-4}}\Big).$$
 Next, we demonstrate  that 
 		\begin{align}\label{n>>R1}
 	\int_{\R^n}	\Pi(u_0, u_j, u_{j+1})du_j
 	&\lesssim
 	\begin{cases}
 	{|u_0-u_{j+1}|^{-\frac{n-1}{2}}}+{|u_0-u_{j+1}|^{-n+4(a+1 )}}
& {\rm if}\ a=1, 2, \cdots, \big\lfloor\frac{n}{8}\big\rfloor,\\
 	{|u_0-u_{j+1}|^{-\frac{n-1}{2}}} & {\rm if}\ a=\big\lfloor\frac{n}{8}\big\rfloor.
 	\end{cases}
 	\end{align}
 	Indeed, note that
 	\begin{align*}
\Pi(u_0, u_j, u_{j+1})=&\frac{\langle u_j \rangle^{-\beta}}{|u_0-u_j|^{\frac{n-1}{2}}|u_j-u_{j+1}|^{\frac{n-1}{2}}}+\frac{\langle u_j \rangle^{-\beta}}{|u_0-u_j|^{\frac{n-1}{2}}|u_j-u_{j+1}|^{n-4}}\\
 &+\frac{\langle u_j \rangle^{-\beta}}{|u_0-u_j|^{n-4a}|u_j-u_{j+1}|^{\frac{n-1}{2}}} +\frac{\langle u_j \rangle^{-\beta}}{|u_0-u_j|^{n-4a}|u_j-u_{j+1}|^{n-4}}.
\end{align*}
Consequently,  the integral with respect to 
 $u_j$ in \eqref{n>>R1} can be divided into four terms,  each corresponding to one of the four parts above. 
 For the first term, since 
 \begin{align*}
 	\frac{1}{|u_0-u_j|^{\frac{n-1}{2}}|u_j-u_{j+1}|^{\frac{n-1}{2}}}
 	&\lesssim	|u_0-u_{j+1}|^{-\frac{n-1}{2}}	\bigg(\frac{1}{|u_0-u_j|^{\frac{n-1}{2}}}+\frac{1}{|u_j-u_{j+1}|^{\frac{n-1}{2}}}\bigg).
 	\end{align*}
 	Then, by Lemma \ref{Appendix 1} and  $\beta>\frac{n+1}{2},$ one has
 	 \begin{align*}
 	\int_{\R^n}	\frac{\langle u_j \rangle^{-\beta}}{|u_0-u_j|^{\frac{n-1}{2}}|u_j-u_{j+1}|^{\frac{n-1}{2}}}du_j\lesssim	|u_0-u_{j+1}|^{-\frac{n-1}{2}}.
 	\end{align*}
 	Consider that  $n-4a\geq n-4\big\lfloor\frac{n}{8}\big\rfloor>\frac{n-1}{2},$ for $a=1, 2, \cdots, \big\lfloor\frac{n}{8}\big\rfloor,$  then similarly, by  Lemma \ref{Appendix 1} and  $\beta>4\big\lfloor\frac{n}{8}\big\rfloor,$ we derive the two terms (the second and the third terms)
 	 \begin{align*}
 		\int_{\R^n}	\frac{\langle u_j \rangle^{-\beta}}{|u_0-u_j|^{\frac{n-1}{2}}|u_j-u_{j+1}|^{n-4}}du_j\ \ \text{and}\ \  \int_{\R^n}	\frac{\langle u_j \rangle^{-\beta}}{|u_0-u_j|^{n-4a}|u_j-u_{j+1}|^{\frac{n-1}{2}}}du_j \lesssim	|u_0-u_{j+1}|^{-\frac{n-1}{2}}.
 	\end{align*}
 	As for the fourth  term,
 	note that $n-4a>\frac{n-1}{2},$   $n-4(a+1)>0$ for $a=1, 2, \cdots, \big\lfloor\frac{n}{8}\big\rfloor,$ then  utilize Lemma \ref{Appendix 2} to obtain that 
 	 \begin{align*}
 		\int_{\R^n}	\frac{\langle u_j \rangle^{-\beta}}{|u_0-u_j|^{n-4a}|u_j-u_{j+1}|^{n-4}}du_j\lesssim
 		\begin{cases} 
 		\frac{1}{|u_0-u_{j+1}|^{\frac{n-1}{2}}}& {\rm if}\  |u_0-u_{j+1}| > 1, \\
 				\frac{1}{|u_0-u_{j+1}|^{n-4(a+1)}}	 &{\rm if} \ |u_0-u_{j+1}| \leq1.
 		\end{cases}
 	\end{align*}
 Additionally,  take into account $n-4\big(\big\lfloor\frac{n}{8}\big\rfloor+1\big)
 	\leq\frac{n-1}{2},$ then 
 	\begin{align*}
 		|u_0-u_{j+1}|^{-n+4(\lfloor\frac{n}{8}\rfloor+1)}\leq|u_0-u_{j+1}|^{-\frac{n-1}{2}}, \ \ \ \ \   \text{for}\  |u_0-u_{j+1}| \leq1,
 		\end{align*}
 	which yields \eqref{n>>R1}.
 	
 	Consider \eqref{n>>R} and  \eqref{n>>R1}, then for $s=\lfloor \frac{n}{8}\rfloor+1 $  and  by Lemma \ref{Appendix 1}, it follows that 
 	  	\begin{align*}
 	  	\left\|\big(R_0^+ V\big)^s(x, \cdot)\right\|_{L^{2, \sigma}}\lesssim\eta^{s\big(\frac{n+1}{2}-4\big)}
 	  	\bigg(\int_{\R^n}\bigg|
 	  	\frac{\big\langle u_s\big\rangle^{-(\beta-\sigma)}}{|u_0-u_s|^{\frac{n-1}{2}}}\bigg|^2du_s\bigg)^{\frac{1}{2}}\lesssim\eta^{s\big(\frac{n+1}{2}-4\big)}\big\langle x\big\rangle^{-\frac{n-1}{2}}. 
 	  \end{align*}
 	   Similarly,  $ 	\big\|\big(VR_0^+\big)^{s-1}\big(V\big(R_0^+-R_0^-\big)\big)(\cdot, y)	\big\|_{L^{2, \sigma}}\lesssim\eta^{s\big(\frac{n+1}{2}-4\big)}\big\langle y\big\rangle^{-\frac{n-1}{2}}$ for $s=\lfloor \frac{n}{8}\rfloor+1$. 
 	  	
 	$\bullet$  \underline {Given $\beta>2$  and $\beta>2\sigma,$}, in the following we are devoted  to showing that 
 	  \begin{align}\label{high2}
 	  \left\|	\big(R_0^+ V\big)^{k-s}R_V^+\big(VR_0^+\big)^{k-s+1} \right\|_{L^{2, \sigma}\to L^{2, -\sigma}}\lesssim\eta^{-3(2k-2s+2)}, \ \ \ 
 	   \sigma>\frac{1}{2}.
 	  	\end{align}
 	  	In fact,   by Lemma \ref{lemma_high_4}, provided $\beta>2$  and $\beta>2\sigma$ such that $	\left\|\langle \cdot\rangle^{2\sigma}V\right\|_{L^\infty}<\infty,$ take $f\in L^{2, \sigma},$   then
 	  	  \begin{align*}
 	  \Big\|	\Big(\big(R_0^+ V\big)^{k-s}R_V^+\big(VR_0^+\big)^{k-s+1}\Big)f \Big\|_{ L^{2, -\sigma}}	  	&\lesssim
 	  	  	\left\|R_0^+\right\|_{L^{2, \sigma}\to L^{2, -\sigma}}	\left\|	\Big(V\big(R_0^+ V\big)^{k-s-1}R_V^+\big(VR_0^+\big)^{k-s+1}\Big)f \right\|_{ L^{2, \sigma}}\\
 	  	  	&  \lesssim
 	  	  \eta^{-3}
 	  	  	\left\|\langle \cdot\rangle^{2\sigma}V\right\|_{L^\infty}
 	  	  		\left\|	\Big(\big(R_0^+ V\big)^{k-s-1}R_V^+\big(VR_0^+\big)^{k-s+1}\Big)f \right\|_{ L^{2, -\sigma}}.
 	  	  \end{align*}
 	  	Repeating the process above, we conclude that 
 	  	  \begin{align*}
 	  		 \Big\|	\Big(\big(R_0^+ V\big)^{k-s}R_V^+\big(VR_0^+\big)^{k-s+1}\Big)f \Big\|_{ L^{2, -\sigma}}\lesssim\eta^{-3(2k-2s+1)}\left\|	R_0^+f \right\|_{ L^{2, -\sigma}}
 	  		 \lesssim\eta^{-3(2k-2s+2)}\left\|f \right\|_{ L^{2, \sigma}}, 
 	  	\end{align*}
 	  	     which yields \eqref{high2}.  
 	  	     
 	 Utilize \eqref{high1},  \eqref{high2} and H\"older's inequality to obtain that $\Omega_H^-(x, y)$ given by \eqref{high3} satisfies
 	 	\begin{align}\label{high4}
 	 	\Big|\Omega_H^-(x, y)\Big|&\lesssim\int_0^\infty\bigg( \eta(2\eta^2+1)\widetilde{\chi}(\eta)
 	 	\big\|\big(R_0^+ V\big)^{s}(x, \cdot)\big\|_{L^{2, \sigma}}  \big\|	\big(R_0^+ V\big)^{k-s}R_V^+\big(VR_0^+\big)^{k-s+1} \big\|_{L^{2, \sigma}\to L^{2, -\sigma}}\nonumber\\ &\ \ \ \ \ \ \ \ \ \ \ \ \ \ \ \ \ \ \  \ \ \ \ \ \ \ \ \ \ \ \ \ \  \times\big\|\big(VR_0^+\big)^{s-1}\big(V\big(R_0^+-R_0^-\big)\big)(\cdot, y)	\big\|_{L^{2, \sigma}}\bigg)d\eta\nonumber\\
 	 	&\lesssim\big\langle x\big\rangle^{-\frac{n-1}{2}}\big\langle y\big\rangle^{-\frac{n-1}{2}}\int_1^\infty \eta^{-6k+s(n-1)-3}d\eta\lesssim\big\langle x\big\rangle^{-\frac{n-1}{2}}\big\langle y\big\rangle^{-\frac{n-1}{2}},
 	 \end{align}
 	provided $k$ enough large such that  $-6k+s(n-1)-3<-1.$
 	
	$\bullet$ \underline{ Next,	it remains to demonstrate that }
 		\begin{align*}
 		\big|\Omega_H^-(x, y)\big|\lesssim\big\langle x\big\rangle^{-\frac{n-1}{2}}\big\langle y\big\rangle^{-\frac{n-1}{2}}\big||x|\pm|y|\big|^{-N}.
 	\end{align*}
Indeed, observe that $\Omega_H^-(x, y)=\Omega_{H,+}^-(x, y)-\Omega_{H,-}^-(x, y),$ where 
	\begin{align*}
	\Omega_{H,\pm}^-(x, y)=&\frac{1}{\pi i}\int_0^\infty e^{i\eta\big(|x|\pm|y|\big)} \eta(2\eta^2+1)\widetilde{\chi}(\eta)
\Bigg(	\int_{\R^{n}}\Big[\big(\mathcal{G}^+(\eta) V\big)\big(R_0^+ V\big)^{s-1}\Big](x, z_1)\nonumber\\
&	\times\bigg[\bigg(\big(R_0^+ V\big)^{k-s}R_V^+\big(VR_0^+\big)^{k-s+1}\bigg)
	\bigg(\big(VR_0^+\big)^{s-1}\big(V\mathcal{G}^{\pm}(\eta)\big)\bigg)\bigg](z_1, y)dz_1\Bigg)d\eta.
\end{align*} 	 
Here,	$\mathcal{G}^\pm(\eta)(x, y):=	e^{\mp i\eta|x|} R_0^\pm(\eta^4+\eta^2)(x, y)$ and $s\in\N^+.$
  Taking into account the support of $\widetilde{\chi}(\eta)$ and applying Lemma \ref{lemma_high_4}, we observe that by selecting $k$ sufficiently large, no boundary terms appear when performing integration by parts with respect to $\eta$. 

Hence, integrating by parts  $N=\lceil\frac{n}{2}\rceil+1$ times with respect to  $\eta$ and Leibniz formula, we derive
\begin{align}\label{high6}
\Big|\Omega_{H,\pm}^-(x, y)\Big|&\lesssim
\frac{1}{ \big||x|\pm|y|\big|^N}\sum_{\ell_1+\cdots+\ell_4=N}\ \int_1^\infty  \eta^{3-\ell_1}
	\Bigg|	\int_{\R^{n}}\partial_\eta^{\ell_2}\Big[\big(\mathcal{G}^+(\eta) V\big)\big(R_0^+ V\big)^{s-1}\Big](x, z_1)\nonumber\\
	&	\times\bigg[\partial_\eta^{\ell_3}\bigg(\big(R_0^+ V\big)^{k-s}R_V^+\big(VR_0^+\big)^{k-s+1}\bigg)
	\partial_\eta^{\ell_4}\bigg(\big(VR_0^+\big)^{s-1}\big(V\mathcal{G}^{\pm}(\eta)\big)\bigg)\bigg](z_1, y)dz_1\Bigg|d\eta.
\end{align} 
Consider \eqref {RR1}, then $\big| V(z)\partial_\eta^{\ell}R_0^+(\eta^4+\eta^2) (z, u)V(u)\big|$ is bounded by 
\begin{align*}
	\lesssim
	\begin{cases}
		\eta^{\frac{n+1}{2}-4}	
		\big\langle z\big\rangle^{-(\beta-\ell)}\big\langle u\big\rangle^{-(\beta-\ell)}\big |z-u\big|^{-\frac{n-1}{2}}
		& {\rm if}\ 5\leq n\leq7,\\ 
		\eta^{\frac{n+1}{2}-4}	
		\big\langle z\big\rangle^{-(\beta-\ell)}\big\langle u\big\rangle^{-(\beta-\ell)} \Big(|z-u|^{-\frac{n-1}{2}}+|z-u|^{-n+4}\Big)
		& {\rm if}\ n>7.
	\end{cases}
	\end{align*}
	Hence, following the approach used to establish \eqref{high1} and combining with \eqref{GG1},  provided $\beta > \frac{n+1}{2} + (\ell_2 + \ell_4)$ and $2(\beta - \ell_2 - \ell_4 - \sigma) > n$, it holds that
	\begin{align}\label{high7}
		&\left\|\partial_\eta^{\ell_2}\Big[\big(\mathcal{G}^+(\eta) V\big)\big(R_0^+ V\big)^{s-1}\Big](x, \cdot)\right\|_{L^{2, \sigma}}\lesssim\eta^{s\big(\frac{n+1}{2}-4\big)} \big\langle x\big\rangle^{-\frac{n-1}{2}},\nonumber\\
		&\left\|\partial_\eta^{\ell_4}\Big(\big(VR_0^+\big)^{s-1}\big(V\mathcal{G}^{\pm}(\eta)\big)\Big)(\cdot, y)	\right\|_{L^{2, \sigma}}\lesssim\eta^{s\big(\frac{n+1}{2}-4\big)}\big\langle y\big\rangle^{-\frac{n-1}{2}}.
	\end{align}
Additionally, by virtue of Lemma \ref{lemma_high_4} and under the condition $\beta>2\sigma$ and $\beta > 2 + 2\ell_3$ for $0 \leq \ell_3 \leq N$ and $\ell_3 \in \N$, similar to \eqref{high2}, it follows that 
	 \begin{align}\label{high8}
		\left\|	\partial_\eta^{\ell_3}\bigg(\big(R_0^+ V\big)^{k-s}R_V^+\big(VR_0^+\big)^{k-s+1}\bigg) \right\|_{L^{2, \sigma}\to L^{2, -\sigma}}\lesssim\eta^{-3(2k-2s+2)}, \ \ 
		\sigma>\ell_3+\frac{1}{2}.
	\end{align}
	By \eqref{high6}, \eqref{high7}  \eqref{high8} and H\"older's inequality, similar to \eqref{high4},  we  derive that 
		\begin{align*}
		\Big|\Omega_{H,\pm}^-(x, y)\Big|
		&\lesssim\big\langle x\big\rangle^{-\frac{n-1}{2}}\big\langle y\big\rangle^{-\frac{n-1}{2}} \big||x|\pm|y|\big|^{-N}\int_1^\infty \eta^{-6k+s(n-1)-3}d\eta\\
		&\lesssim\big\langle x\big\rangle^{-\frac{n-1}{2}}\big\langle y\big\rangle^{-\frac{n-1}{2}}\big||x|\pm|y|\big|^{-N},
	\end{align*}
which gives the desired result, combining with Lemma \ref{lermmaK}.

		Additionally, we emphasize that the decay rate $\beta > 2 + 2\left(\lceil \frac{n}{2} \rceil + 1\right)$ (i.e., $\beta > n + 5$ when $n$ is odd and $\beta > n + 4$ when $n$ is even) is required when $\ell_3 = \lceil \frac{n}{2} \rceil + 1$ and all derivatives are applied to either $R^+_V$ or  some $R^+_0.$	
		Thus we complete the proof. 
\end{proof}

 		\appendix
 		\section{Technical lemmas}  \label{section5}
 		In this section, we provide  several technical lemmas that play important roles in  this paper.
 		To begin, we present one of the Hardy-Sobolev type inequalities, which is frequently used in the proof of Lemma \ref{lemma F}.
 		\begin{lemma}[{\cite[Corollary 4, Page 139]{Sobolev}}]\label{Appendix 5}
 Let $n>\ell p,$ $1\leq p\leq\frac{np}{n-\ell p}$ and $\beta=\alpha-\ell>-\frac{n}{p}.$	Then		
 			$$\big \||y|^\beta u\big\|_{L^p(\R^n)}\lesssim\big\||y|^\alpha \nabla^\ell u\big\|_{L^p(\R^n)}, \ \ \ \ \ \ \  u\in C_0^\infty(\R^n).$$
 				\end{lemma}
The next two Lemmas \ref{Appendix 1}-\ref{Appendix 2} have been  already well-established and are frequently used.
	\begin{lemma}[{\cite[Lemma 3.8]{Goldberg-Visan-CMP} }]\label{Appendix 1}
 			Let $\mu, \beta$ satisfy that  $\mu < n$ and $n < \beta + \mu$. Then
 			$$
 			\int_{\R^n} \frac{\langle y\rangle ^{-\beta}}{ |x-y|^{\mu} }dy\lesssim 
 				\begin{cases} 
 					\langle x\rangle^{n-\beta - \mu} & {\rm if} \ \beta < n, \\
 					\langle x\rangle^{- \mu} & {\rm if}\ \beta >n.
 				\end{cases}
 				$$
 			\end{lemma}
 			\begin{lemma}[{\cite[Lemma 6.3]{Erdogan-Green10} }]\label{Appendix 2}
 		 Fix $u_1, u_2 \in \mathbb{R}^n.$ Let  $0 \leq k, \ell < n$, $\beta> 0$, $k + \ell + \beta > n$, $k + \ell \neq n$. Then
 			\begin{align*}
 			\int_{\mathbb{R}^n} \frac{\langle u \rangle ^{-\beta}}{|u- u_1|^k | u- u_2|^\ell} du
 			\lesssim 
 			\begin{cases} 
 			\Big(\frac{1}{|u_1-u_2|}\Big)^{\max(0, k + \ell - n) }& {\rm if}\  |u_1 - u_2| \leq 1, \\
 			\Big(\frac{1}{|u_1-u_2|}\Big)^{\min(k, \ell, k+\ell+\beta-n)}	 &{\rm if} \ |u_1 - u_2| > 1.
 			\end{cases}
 			\end{align*}
 			\end{lemma}
 		Below, we provide a sufficient condition to ensure that the operator is admissible, which is used in the proof of Proposition \ref{proposition-WL}.
 			
 			\begin{lemma}\label{Appendix 3}
 			Let $n-k\leq\alpha<n$ and $k>n.$ Then the  operator $K$ with kernel: 
 			\begin{align*}
 				K(x,y):=\int_{\R^{2n}}
 				\frac{\langle z_1\rangle^{-n-}\langle z_2\rangle^{-n-}\chi_{|y-z_2|\gg\langle x-z_1 \rangle}}{|x-z_1|^\alpha|y-z_2|^k}dz_1dz_2
 				\end{align*}
 			is admissible, i.e.
 				\begin{align*}
 				\sup_{x\in\R^n}\int_{\R^n}\big|K(x,y)\big|dy+\sup_{y\in\R^n}\int_{\R^n}\big|K(x,y)\big|dx<+\infty.
 				\end{align*}
 			\end{lemma}
 			\begin{proof}
 			Let $r_1=|x-z_1|$ and $r_2=|y-z_2|.$ We can check that
 				\begin{align*}
 				\int_{\R^n}\big|K(x,y)\big|dy
 				&\lesssim	\int_{\R^{2n}}
 				\frac{\langle z_1\rangle^{-n-}\langle z_2\rangle^{-n-}}{r_1^\alpha }\int_{ \langle r_1\rangle}^\infty \frac{1}{r^{k-n+1}}drdz_1dz_2\\
 				&\lesssim	\int_{\R^{2n}}
 				\frac{\langle z_1\rangle^{-n-}\langle z_2\rangle^{-n-}}{r_1^\alpha \langle r_1\rangle^{k-n}}dz_1dz_2\lesssim1.
 			\end{align*}
 		Here, the first inequality follows from the polar coordinate transformation: 
 			$y - z_2 = r\omega, \quad (r, \omega) \in \mathbb{R}^+ \times S^{n-1}.$
 		The second and the last inequalities are derived from $k > n$ and Lemma \ref{Appendix 1} with $\alpha < n$, $k - n + \alpha \geq 0$,  respectively.
 		
 		Similarly, we conclude that 
 			\begin{align*}
 			\int_{\R^n}\big|K(x,y)\big|dx
 			&\lesssim	\int_{\R^{2n}}
 			\frac{\langle z_1\rangle^{-n-}\langle z_2\rangle^{-n-}}{r_2^k  }r_2^{n-\alpha}\chi_{r_2\gg1}dz_1dz_2\lesssim1.
 		\end{align*}		
			Thus we complete the proof. 
 				\end{proof}

\end{document}